\documentclass[11pt]{article}

\input{style}

\usepackage{amsmath, amssymb, amstext, amsfonts, mathrsfs, bbold}
\usepackage{stmaryrd, paralist,url}
\usepackage{grffile}
\usepackage{listings}
\usepackage{color}
\usepackage[list=true, font=large, labelfont=bf, 
labelformat=brace, position=top]{subcaption}
\usepackage{calc}
\usepackage{accents}
\usepackage{mathtools}

\usepackage{geometry}
\mathtoolsset{showonlyrefs}

\newtheoremstyle{results}
  {1ex}
  {1ex}
  {\itshape}
  {}
  {\bfseries}
  {.}
  {.5em}
  {}

\theoremstyle{results}

	\newtheorem{Sat}{Satz}[section]
	\newtheorem{Thm}[Sat]{Theorem}
	\newtheorem{Lem}[Sat]{Lemma}
	\newtheorem{Pro}[Sat]{Proposition}
	\newtheorem{Cor}[Sat]{Corollary}

	\newtheorem{DefThm}[Sat]{Definition and Theorem}

	\newtheorem{Def}[Sat]{Definition}
	\newtheorem{Ass}[Sat]{Assumption}	
	
\newtheoremstyle{remarks}
  {1ex}
  {1ex}
  {}
  {}
  {\itshape}
  {.}
  {.5em}
  {}
  
\theoremstyle{remarks}
	\newtheorem{Ex}[Sat]{Example}
	\newtheorem{Rem}[Sat]{Remark}

	\newcommand{\EW}{\mathbb{E}}
	\newcommand{\WM}{\mathbb{P}}
	\newcommand{\NZ}{\mathbb{N}}
	\newcommand{\RZ}{\mathbb{R}}
	\newcommand{\md}{\mathrm{d}}

	\newcommand{\FA}{\mathcal{F}}
	\newcommand{\aFA}{\mathcal{F}^\Lambda}
	\newcommand{\midG}{\,\middle|\,}
	
	\newcommand{\uP}{\varphi}

	\newcommand{\st}{{\mathcal{S}^\mathcal{O}}}
	\newcommand{\stp}{{\mathcal{S}^\mathcal{P}}}
	\newcommand{\stm}{{\mathcal{S}^\Lambda}}
	\newcommand{\stmr}{\stm\left([0,\infty)\right)}
	\newcommand{\stmg}[1]{\stm\left([#1,\infty]\right)}
	\newcommand{\stmsg}[1]{\stm\left((#1,\infty]\right)}
	\newcommand{\stmd}{{{\mathcal{S}^{\Lambda,\mathrm{div}}}}}	
	\newcommand{\stmgs}[2]{\stm\left([#1,#2]\right)}

    \newcommand{\ubar}[1]{\underaccent{\bar}{#1}}
    \newcommand{\RM}[1]{\MakeUppercase{\romannumeral #1{}}}
	
	\newcommand{\stsetLO}[2]{\left\rrbracket #1,#2 \right\rrbracket}
	
	\newcommand{\stsetRO}[2]{\left\llbracket #1,#2 \right\llbracket}
	\newcommand{\stsetC}[2]{\left\llbracket #1,#2 \right\rrbracket}
	\newcommand{\stsetO}[2]{\left\rrbracket #1,#2 \right\llbracket}

	\newcommand{\lsl}[1]{{^\ast #1}}
	\newcommand{\lsr}[2]{{#1^\ast_{#2}}}
	
	\newcommand{\citing}[3]{\cite{#1}, #2, p.#3}

\begin{document}

\makeatletter
\let\@fnsymbol\@arabic
\makeatother

\title{\vspace*{-3cm} On a Stochastic Representation Theorem for Meyer-measurable Processes and its Applications in Stochastic Optimal Control and Optimal Stopping}

\author{ Peter Bank\footnote{Technische Universit{\"a}t Berlin,
    Institut f{\"u}r Mathematik, Stra{\ss}e des 17. Juni 136, 10623
    Berlin, Germany, email \texttt{bank@math.tu-berlin.de}. }
  \hspace{4ex} David Besslich \footnote{Technische Universit{\"a}t
    Berlin, Institut f{\"u}r Mathematik, Stra{\ss}e des 17. Juni 136,
    10623 Berlin, Germany, email \texttt{besslich@math.tu-berlin.de}.}  }
\date{\today}

\maketitle
\begin{abstract}
    In this paper we study a representation problem first considered in a simpler version by \cite{BK04}. A key ingredient to this problem is a random measure $\mu$ on the time axis which in the present paper is allowed to have atoms. Such atoms turn out to not only pose serious technical challenges in the proof of the representation theorem, but actually have significant meaning in its applications, for instance, in irreversible investment problems. These applications also suggest to study the problem for processes which are measurable with respect to a Meyer-$\sigma$-field that lies between the predictable and the optional $\sigma$-field. Technically, our proof amounts to a delicate analysis of optimal stopping problems and the corresponding optimal divided stopping times and we will show in a second application how an optimal stopping problem over divided stopping times can conversely be obtained from the solution of the representation problem. 
\end{abstract}

\begin{description}
	\item[Keywords:] Representation Theorem, Meyer-$\sigma$-fields, Divided stopping times, Optimal Stochastic Control.
\end{description}


	\section{Introduction}
        
        In this paper we study a stochastic representation problem,
	first considered in a simpler framework by \cite{BK04}. Specifically, we consider a Meyer-$\sigma$-field $\Lambda$
	such as the predictable or optional $\sigma$-field and, under weak regularity assumptions, construct a 
	$\Lambda$-measurable process $L$ 
	such that a given $\Lambda$-measurable process $X$ can be written as
	\begin{align}\label{sto:intro}
		X_S=\EW\left[\int_{[S,\infty)} g_t\left(\sup_{v\in [S,t]} L_v\right)
							 \mu(\md t)\midG \aFA_S\right]
	\end{align}
	at every $\Lambda$-stopping time $S$. 
	
	In \cite{BK04}, stochastic representations like \eqref{sto:intro}
	are proven for optional processes $X$ and atomless optional 
	random measures $\mu$ with full support.
	Our Main Theorem \ref{sto:results_thm_2} 
	generalizes their result in several ways.
	Most notably, we solve the representation problem
	for measures $\mu$ with atoms. Such atoms not only pose considerable technical challenges for the representation problem \eqref{sto:intro}, but also convey significant meaning in its applications.
	
	For instance in an application of this representation problem to a novel version of the singular stochastic control problem of irreversible investment with inventory risk (see \cite{BB18_3} and, e.g., \cite{RX11}, \cite{CF14} for earlier versions),
	$\mu$ and $g$ are used to measure the incurred risk and the atoms of $\mu$  reflect times of particular importance for the risk assessment;
	the process $X$ describes the revenue per additional investment unit. As proven in the companion paper \cite{BB18_3}, it then turns out that $(\sup_{v\in [0,t]}L_v)_{t\geq 0}$ yields an
	 optimal investment strategy. At any atom of $\mu$, the optimal control has to trade off an improvement in the impending risk assessment against any revenue from additional investment. How exactly this comes down also depends crucially on what information is available to the controller in this moment.
	This can be modelled by a Meyer-$\sigma$-field interpolating between ``reactive'' predictable controls and ``proactive'' optional ones;  see \cite{EL80} for a detailed account.
	To account for the full variety of such information dynamics, we solve \eqref{sto:intro} for an arbitrary  
	Meyer-$\sigma$-field $\Lambda$ instead of merely the optional
	$\sigma$-field. 
	Another extension over \cite{BK04} is the possibility to take \emph{every}  stopping time
	$\hat{T}$ instead of $\Hat{T}=\infty$ as the time horizon in \eqref{sto:intro}, not just predictable ones: We just need to ``freeze'' the problem's inputs as we reach $\hat{T}$. Moreover, 
	we are able to prove that our solution $L$ is maximal in the sense that
	any other $\Lambda$-measurable solution is less than or equal
	to ours up to an evanescent set.
	Instead of this natural maximality property, \cite{BK04} just prove uniqueness under additional assumptions on the paths of $L$ which do not always obtain.
	
	The technical challenges in establishing the representation \eqref{sto:intro} arise first due to the fact
	that the original construction of $L$ in \cite{BK04} is based on the properties of optimal stopping times for the family of auxiliary stopping problems
	\begin{align}\label{sto:auxil}
		Y^\ell_S=\esssup_{T\in \stmg{S}} 
					 		\EW\left[X_T+\int_{[S,T)}g_t(\ell)\mu(\md t)\midG \aFA_S\right], \ \ell\in \RZ, \text{ $S$ $\Lambda$-stopping time.} 
	\end{align}
	When $\mu$ has atoms, the running costs can exhibit upward and downward jumps and so, optimal stopping times for \eqref{sto:auxil} may exist only in the relaxed form of divided stopping times (or \emph{temps divis\'ees}, see \cite{EK81}). Divided stopping times are actually quadruples consisting of a stopping time and three disjoint sets decomposing the probability space. These require a considerable more refined analysis. Conversely, we show that a solution to \eqref{sto:intro} also allows us to solve such generalized stopping problems and thus offers an alternative to the usual approach via Snell envelopes as in \cite{EK81}.
	
	While conceptually very versatile for modelling information flows and technically convenient to unify the treatment of predictable and optional settings, the consideration of Meyer-$\sigma$-fields adds mathematical challenges of its own.
	For instance, the level passage times of a Meyer-measurable process may not necessarily be Meyer stopping times: The entry time of a predictable
	process is itself not predictable in general. Second, for optional processes,
	right-upper-semicontinuity in expectation implies pathwise 
	right-upper-semicontinuity; see \cite{BS77}. For $\Lambda$-measurable
	processes however,	this implication does not hold true in general.
	Such subtleties can be disregarded when $\mu$ does not have atoms as in \cite{BK04}, but become crucial for \eqref{sto:intro} when it does.
	
	This paper will be organized in the following way. Section
	\ref{sec:rep} introduces the framework and the main result.
	In Section \ref{sec:appl} we give applications of the main theorem first in irreversible investment and then in optimal stopping over divided stopping times.
	In Section \ref{sto:sec_upper} we
	prove maximality and in Section \ref{sto:sec_ex} we finally 
	prove existence of a solution $L$ to \eqref{sto:intro}. The technical proofs of auxiliary results are deferred to the appendix.
	    
	\section{A Representation Problem}\label{sec:rep}
	
	   Let us fix throughout a filtered probability space $(\Omega,\mathbb{F}, 
		\FA:=(\FA_t)_{t\geq 0},\WM)$ 
		and
		$\mathcal{F}_{\infty}:=\bigvee_{t} \FA_t	\subset \mathbb{F}$
		and $\FA$ satisfying the usual conditions of completeness and 
		right-continuity. 
		Furthermore let $\Lambda$ be a $\WM$-complete Meyer-$\sigma$-field which contains
		the predictable-$\sigma$-field with respect to $\FA$ and is contained in the optional-$\sigma$-field with respect to $\FA$. We will use the concept of Meyer-$\sigma$-fields as they can be used to model different information dynamics in optimal control problems (see our companion paper \cite{BB18_3}). But Meyer $\sigma$-fields also allow us to prove our main result simultaneously for the predictable and the optional-$\sigma$-field which are both special cases of Meyer-$\sigma$-fields. The theory of Meyer-$\sigma$-fields was initiated in \cite{EL80}. We review and expand some of this material in the companion paper \cite{BB18_2}. 
			Let us recall in the next Section \ref{sec:meyer} the basic concepts and results. Upon first reading, the reader is invited to think of $\Lambda$ as the optional-$\sigma$-field in which case $\Lambda$-stopping times $\stm$ are just classical stopping times and may then skip directly to Section \ref{sec:notation}.
	
	    \subsection{Meyer-$\sigma$-fields}\label{sec:meyer}
	    
			\begin{Def}[Meyer-$\sigma$-field, \citing{EL80}{Definition 2}{502}]
			\label{Main:45}
				A $\sigma$-field $\Lambda$ on $\Omega\times [0,\infty)$
				is called a \emph{Meyer-$\sigma$-field},
				if the following conditions hold:
				\begin{compactenum}[(i)]
					\item It is generated by some right-continuous, left-limited
					(rcll or c\`adl\`ag in short) processes.
					\item It contains 
					$\{\emptyset,\Omega\}\times \mathcal{B}([0,\infty))$,
					where $\mathcal{B}([0,\infty))$
					denotes the Borel-$\sigma$-field on $[0,\infty)$.
					\item It is stable with respect to stopping at deterministic
					time points, i.e. for a $\Lambda$-measurable process $Z$,
					$s\in [0,\infty)$, also the stopped process 
					$(\omega,t)\mapsto Z_{t\wedge s}(\omega)$ is 
					$\Lambda$-measurable.
				\end{compactenum}
			\end{Def}			
			
			\begin{Ex}
			    The optional $\sigma$-field 
    			$\mathcal{O}(\mathcal{F})$ and the predictable 
    			$\sigma$-field  $\mathcal{P}(\FA)$ associated to a given
    			filtration $\FA$ are examples of Meyer-$\sigma$-fields. 
			\end{Ex}
			
			Like for filtrations, there also exists a notion of completeness of Meyer-$\sigma$-fields with respect to a probability measure $\WM$:
			\begin{DefThm}[$\WM$-complete Meyer-$\sigma$-field, see \cite{EL80}, p.507-508]
			    A Meyer-$\sigma$-field $\Lambda\subset \mathbb{F}\otimes \mathcal{B}([0,\infty))$ is called $\WM$-complete if any process $\tilde{Z}$ which is indistinguishable
			    from a $\Lambda$-measurable process $Z$ is already $\Lambda$-measurable. For any Meyer-$\sigma$-field $\tilde{\Lambda}\subset \mathbb{F}\otimes \mathcal{B}([0,\infty))$
			    there exists a smallest $\WM$-complete Meyer-$\sigma$-field $\Lambda$ containing $\tilde{\Lambda}$, called the
			    $\WM$-completion of $\tilde{\Lambda}$.
			\end{DefThm}
			
			\begin{Ex}[\citing{EL80}{Example}{509}]
			    We have a filtration $\tilde{\FA}:=(\tilde{\FA}_t)_{t\geq 0}$ on a probability space $(\Omega,\mathbb{F},\WM)$
			    and denote by $\FA$ the smallest filtration satisfying the usual conditions containing $\tilde{\FA}$.
			    Then the $\WM$-completion of the $\tilde{\FA}$-predictable $\sigma$-field is the $\FA$-predictable $\sigma$-field
			    and the $\WM$-completion of the $\tilde{\FA}$-optional $\sigma$-field is contained in the $\FA$-optional $\sigma$-field.
			\end{Ex}
			
    		The following Theorem shows that the optional and predictable
    		$\sigma$-fields are the extreme cases of Meyer-$\sigma$-fields:
			
			\begin{Thm}[\citing{EL80}{Theorem 5}{509}]
        			A $\sigma$-field on  
        			$\Omega\times [0,\infty)$ generated by
        			c\`adl\`ag processes is a $\WM$-complete Meyer-$\sigma$-field if and only if
        			 it lies between the predictable
        			and the optional $\sigma$-field of a filtration satisfying the usual conditions.
        	\end{Thm}	
			
			\begin{Rem}[Meyer-$\sigma$-fields vs. Filtrations]
			   The main advantages of a Meyer-$\sigma$-field $\Lambda$ compared on a filtration are technical but powerful tools
			   like the upcoming Meyer Section Theorem, which for example gives us uniqueness up to indistinguishability of
			   two $\Lambda$-measurable processes once they coincide at every $\Lambda$-stopping time. 
			\end{Rem}
			
			In our main theorem we will need the following generalization of
			optional and predictable projections: 
			
			\begin{DefThm}[\citing{BB18_2}{Theorem 2.14}{6}]\label{app:meyer_thm_3}
						For any non-negative $\mathbb{F}\otimes
						\mathcal{B}([0,\infty))$-measurable process $Z$,  
						there exists a non-negative
						$\Lambda$-measurable process $^\Lambda Z$, unique up 
						to indistinguishability, such that 
						\[
							\EW\left[\int_{[0,\infty)} Z_s\md A_s\right]
							=\EW\left[\int_{[0,\infty)} {^\Lambda Z}_s\md A_s\right]
						\]
						for any c\`adl\`ag, $\Lambda$-measurable, increasing process
						$A$.
						This process ${^\Lambda Z}$ is called \emph{$\Lambda$-projection of $Z$}. \qed
			\end{DefThm}
			
			Uniqueness up to indistinguishability follows as usual from a suitable section theorem. 
			For stating this theorem we have to use a generalized notion of 
			stopping times:
			
			\begin{Def}[Following \citing{EL80}{Definition 1}{502}]\label{Main:1}
				A mapping $S$ from $\Omega$ to $[0,\infty]$
				is a \emph{$\Lambda$-stopping time},
				if 
				\[
					[[S,\infty[[:=\left\{(\omega,t)\in \Omega\times 
					[0,\infty)\, \middle | \, S(\omega)\leq t\right\}\in \Lambda.
				\]			
				The set of all $\Lambda$-stopping times is denoted
				by $\stm$.
				Additionally we define to each mapping 
				$S:\Omega\rightarrow [0,\infty]$ a
				$\sigma$-field 
				\[
				\mathcal{F}^\Lambda_S := \sigma(Z_S \,|\, Z \text{ $\Lambda$-measurable process}).
				\]	
			\end{Def}
			
			Having introduced the concept of $\Lambda$-stopping times
			we can now state the Meyer Section Theorem,
			which is the Meyer-$\sigma$-field extension of the
			powerful Optional and Predictable Section Theorems:
		
			\begin{Thm}[Meyer Section Theorem, 
						\citing{EL80}{Theorem 1}{506}]\label{Main:4}
						Let $B$ be an element of $\Lambda$. For every
						$\epsilon>0$, there exists $S\in \stm$
						such that $B$ contains the graph of $S$, i.e. 
						\[ 
							B\supset \mathrm{graph}(S):=\{(\omega,S(\omega))
							\in \Omega\times [0,\infty)	\, | \, S(\omega)<\infty\}
						\]								
						and
						\[
							\mathbb{P}(S<\infty)>\mathbb{P}(\pi(B))-\epsilon,
						\]
						where $\pi(B):=\left\{\omega\in \Omega \,
						\middle |\, (\omega,t)\in B
						\text{ for some $t\in [0,\infty)$}\right\}$
						denotes the projection of $B$ onto $\Omega$.
			\end{Thm}			
			
			An important consequence is the following corollary:
	
			\begin{Cor}[\citing{EL80}{Corollary}{507}]\label{app:meyer_cor_1}
				If $Z$ and $Z'$ are two $\Lambda$-measurable processes,
				such that for each bounded $T\in \stm$ we have
				$Z_T\leq Z_T'$ a.s. (resp. $Z_T=Z_T'$ a.s.), then the 
				set $\{Z>Z'\}$ is evanescent (resp. $Z$ and $Z'$ are
				indistinguishable).
			\end{Cor}
	
	    \subsection{Notation}\label{sec:notation}
	        
	        For the sake of notational simplicity, let us
			introduce the following notation:
			\paragraph{Sets of stopping times:} We set	
			\begin{align*}
				\stmr:=
					\left\{T\in \stm \midG T <\infty
						\ \WM\text{-a.s.}\right\}.
			\end{align*}
			Given $S\in \stm$, we shall furthermore make frequent use of
			\[
				\stmg{S}:=
				\left\{T\in \stm \midG T \geq S\ \WM\text{-a.s.}\right\}
			\]
			and
			\[
				\stmsg{S}
				:=\left\{T\in \stm \midG T > S\
				\WM\text{-a.s. on } \{S<\infty\}\right\}.
			\]		
			Analogously, for $R\in \stm$ define the sets
			 $\stm\left((S,R]\right)$, $\stm\left([S,R]\right)$
			as above.
			\paragraph{Stochastic Intervals:}
			Finally we define the stochastic interval for $S,T\in \st$ by
			\[
				\stsetC{S}{T}:=\left\{(\omega,t)\in \Omega\times [0,\infty)
				\midG S(\omega)\leq t \leq T(\omega)\right\}
			\]
			and analogously $\stsetRO{S}{T}$, $\stsetLO{S}{T}$
			and $\stsetO{S}{T}$. Observe that the stochastic intervals defined in this way
			are always subsets of $\Omega\times [0,\infty)$ even if the considered stopping times
			attain the value $\infty$ for some $\omega\in \Omega$.

			\paragraph{Other notation:}			
			We use the convention that $\inf \emptyset = \infty$, $\sup \emptyset = -\infty$, $\infty\cdot 0=0$, $\frac{\cdot}{0}=\infty$
			and $\mathbb{N}:=\{1,2,3,\dots\}$.
	
	    \subsection{Dramatis personae of the representation problem}
			
			Let us now set the stage for our main result. Apart from the Meyer-$\sigma$-field $\Lambda$, our representation problem needs a random Borel-measure $\mu$ on $[0,\infty)$ and a random field $g:\Omega\times [0,\infty)\times \RZ\rightarrow \RZ$ as input. These are assumed to satisfy the following conditions:
	
			\begin{Ass}\label{sto:frame_ass}
				\begin{compactenum}[(i)]				
    				\item 	The random measure $\mu$ on $[0,\infty)$ is optional, i.e. such that its random distribution function $(\mu([0,t])$, $t\geq 0$, is a real-valued, c\`adl\`ag, non-decreasing $\FA$-adapted process, and $\mu(\{\infty\}):=0$.
					\item 	The random field $g:\Omega\times [0,\infty)\times \RZ \rightarrow \RZ$ satisfies:
					\begin{compactenum}
						\item 	For each $\omega\in \Omega$, $t\in [0,\infty)$, the 
							 	function $g_t(\omega,\cdot):\RZ\rightarrow \RZ$ is continuous and strictly increasing from $-\infty$ to $\infty$. \label{sto:ass_bullet_1}
						\item For each $\ell \in \RZ$, the \label{sto:ass_bullet_2}	process $g_\cdot(\cdot,\ell): \Omega\times [0,\infty) \rightarrow \RZ$ is $\FA$-progressively measurable with
					 	\[
					 		\EW\left[\int_{[0,\infty)} |g_t(\ell)|\mu
							(\md t)\right]<\infty.
					 	\]
					\end{compactenum}
				\end{compactenum}
			\end{Ass}						
			
			Furthermore the $\Lambda$-measurable process $X=(X_t,\ 0\leq t\leq \infty)$ 
			with $X_{\infty}=0$ to be represented should exhibit certain regularity properties which we specify next:
			
			\begin{Def}\label{def_conditions_1}	
			    $X$ is of class($\text{D}^\Lambda$) if $\left\{X_T\midG T\in \stm \right\}$ is uniformly integrable, i.e. if
				\[
					\lim_{r\rightarrow \infty}
						\sup_{T\in \stm}
						\EW\left[|X_T|\mathbb{1}_{\{|X_T|>r\}}\right]=0.
				\]
			\end{Def}
				
			\begin{Rem}
					In \citing{EK81}{Proposition 2.29}{127} the
					condition class($D^\Lambda$) is introduced as the appropriate notion for Meyer-measurable processes which only focusses on $\Lambda$-stopping times 
					with respect to which the process
					$X$ has to be uniformly integrable, whereas the classical notion of class($D$) requires this for all $\FA$-stopping times.
			\end{Rem}
				
			\begin{Def}\label{def_conditions_2}		
			\label{sto:def_usc} A $\Lambda$-measurable process
			$X$ of class($D^\Lambda$) will be called
			 \emph{$\Lambda$-$\mu$-upper-semi\-continuous in expectation}
			\label{sto:semi_def} if:
				\begin{compactenum}[(a)]
					\item\label{sto:def_lusc} The process $X$ is
						\emph{left-upper-semicontinuous
						 in expectation at every $S\in \stp$}
						in the sense that for any non-decreasing sequence 
						$(S_n)_{n\in \NZ}\subset {\stm}$ 
						 with $S_n<S$ on $\{S>0\}$ and $\lim_{n\rightarrow \infty}
						S_n=S$ we have
								\begin{align*}
									\EW\left[X_S\right]\geq
									\limsup_{n\rightarrow \infty}
											\EW\left[X_{S_n}\right].
								\end{align*}
					\item \label{sto:def_rusc}
						The process $X$ is
						\emph{$\mu$-right-upper-semicontinuous
						in expectation at every $S\in \stm$} in the
						sense that for $S\in \stm$ and any sequence 
						$(S_n)_{n\in \NZ}\subset
						\stmg{S}$ with
							$\lim_{n\rightarrow \infty} 
							\mu([S,S_n))=0$ almost surely
								we have
						\[
							\EW\left[X_S\right]
							\geq \limsup_{n\rightarrow \infty}
							\EW\left[X_{S_n}\right].  
						\]				
					\end{compactenum}
			\end{Def}
			
			\begin{Rem}[Remark to the $\Lambda$-$\mu$-upper-semicontinuity]\label{sto:frame_rem}
					\begin{compactenum}[(a)]
						\item 
							For $\mu$ with no atoms and full support 
							we see that
							$\mu$-right-upper-semicontinuity in expectation is equivalent to the property that, for all $S\in \stm$ and every sequence $S_n\in \stmg{S}$ which converges
							to $S$ from above, we have that
							\[
									\EW[X_{S}]\geq
											\limsup_{n\rightarrow \infty}
											\EW\left[X_{S_n} \right].
							\]
							It then gives for $\Lambda=\mathcal{O}$ 
							that the classical condition of right-upper-semicontinuity
							in expectation
							 in all $S\in \st$  (cf.
							\citing{EK81}{Proposition 2.42}{141-142})
							is equivalent to
							our definition.
						\item
							Notice that in our notions of $\Lambda$-$\mu$-upper-semicontinuity
							we only require to approximate with
							$\Lambda$-stopping
							times. In order to deduce path properties of $X$ or its
							$\Lambda$-projection we thus extend in \citing{BB18_2}{Lemma 4.4}{21},
							some results of \cite{BS77} who confine themselves to the optional
							case $\Lambda=\mathcal{O}$.
				\end{compactenum}
			\end{Rem}
	
	\subsection{The Main Theorem -- Statement and Discussion}
	
    Now we will state and discuss the main theorem of this paper:
	
	\begin{Thm}\label{sto:results_thm_2}
			Suppose $g$ and $\mu$ satisfy Assumption \ref{sto:frame_ass} 
			and let $X$ be a
			 $\Lambda$-measurable process of
			  class($\text{D}^\Lambda$) which
					 is $\Lambda$-$\mu$-upper-semicontinuous in expectation 
					and satisfies $X_S=0$ for $S\in \stm$ with $\mu([S,\infty))=0$
					almost surely. 
					
					Then $X$ admits a representation of the form 
					  \begin{align}\label{sto:frame_gl}
							X_S=\EW\left[\int_{[S,\infty)} 
								g_t\left(\sup_{v\in [S,t]} L_v\right)
							 \mu(\md t)\midG \aFA_S\right]
							,\ S\in \stm,
					\end{align}
					for the unique (up to indistinguishability) $\Lambda$-measurable process $L$ such that
					\begin{align}\label{sto:results_gl}
						L_S= \essinf_{T\in \stmsg{S}} \ell_{S,T}, \quad 
						S\in \stmr,
					\end{align}					
					where for $S\in \stmr$ and $T\in \stmsg{S}$,
					$\ell_{S,T}$ is the unique (up to a $\WM$-null set)
					$\aFA_S$-measurable random variable such that 
					\begin{align}\label{sto:results_gl_2}
						\EW\left[X_S-X_T\midG \aFA_{S}\right]
						=\EW\left[\int_{[S,T)} g_t(\ell_{S,T})
						\mu(\md t)\midG \aFA_{S}\right]\; \text{ on }
					\{\WM\left(\mu([S,T))>0|\aFA_S\right)>0\}
					\end{align}
					and $\ell_{S,T}=\infty$ on 
					$\{\WM\left(\mu([S,T))>0|\aFA_S\right)=0\}$.
					Furthermore this process $L$ satisfies 				
					\begin{align}
							\EW\left[\int_{[S,\infty)} \left\lvert 
								g_t\left(\sup_{v\in [S,t]} L_v\right)
							 \right\rvert \mu(\md t)\right]<\infty\;
							\text{ for any $S\in \stm$},
							 \label{sto:frame_gl_2}
					\end{align}
					and it is maximal in the sense that 
					$\tilde{L}_S\leq L_S$ for any $S\in \stmr$					
					for every other $\Lambda$-measurable process $\tilde{L}$ 
					satisfying, 	mutatis mutandis, 
					(\ref{sto:frame_gl}) and (\ref{sto:frame_gl_2}).
		\end{Thm}		
			 
	    Let us highlight one by one three different aspects of the preceding result, 
	    the way they go beyond the work of \cite{BK04} and how they can be used in the applications
	    of Section \ref{sec:appl}. First of all,
		Theorem \ref{sto:results_thm_2} can be used for an optional 
		measure $\mu$ with atoms and not necessarily full support. In particular we can embed discrete time frameworks and can take any $\FA$-stopping time $\hat{T}$ as
		time horizon. 
		Second, we can represent any $\Lambda$-measurable 
		process $X$, satisfying the stated conditions
		in the form \eqref{sto:frame_gl} for some $\Lambda$-measurable process $L$.
		This can be used in stochastic control problems to account for different information dynamics for the controller. 
		In the companion paper \cite{BB18_3} we develop this new information modeling idea for stochastic optimal control in greater detail. From this work, we obtain in Section~\ref{sec:control} a first illustration of how different Meyer-$\sigma$-fields $\Lambda$ lead to different solutions $L=L^\Lambda$ to \eqref{sto:frame_gl}. 
		Third, we characterize the maximal solution up to indistinguishability 
		by \eqref{sto:results_gl} without additional assumptions on~$L$.
		Notice that it is not obvious that 
		the family of random variables defined
		for $S\in \stm$ by the right hand side of \eqref{sto:results_gl} 		
		can be aggregated into a $\Lambda$-measurable process $L$.
		So \eqref{sto:results_gl}	does not itself give a construction
		of a stochastic process $L$ and, in fact,
		we are going to construct $L$ instead by using methods from
		optimal stopping;	see Section \ref{sto:sec_ex}. The characterization 
		\eqref{sto:results_gl} can be used in some cases to calculate $L$ explicitly
		as this for example was done in \cite{BB18_3}, Section 2.
		A third application of Theorem \ref{sto:results_thm_2} is an explicit solution to an optimal stopping problem over divided stopping times, where the optimal divided stopping time will
		only depend on $L$. Concerning the proof of Theorem \ref{sto:results_thm_2} we will use the concept of the Snell envelope. In contrast to \cite{BK04} we will not get a stopping time attaining the value of the optimal stopping problem connected to the Snell envelope. This is mainly due to the atoms of $\mu$.
		To overcome this problem we will
		use divided stopping times, which offer another application which we discuss in Section~\ref{sec:stopping}.
		
		The next result shows that for a representation as in \eqref{sto:frame_gl} the process $X$ has to
	be $\mu$-right-upper-semicontinuous in expectation:
	
	\begin{Pro}\label{sto:nec_rusc}
		Assume we have a $\Lambda$-measurable process $X$ of 
					class($D^\Lambda$) with $X_{\infty}=0$ which admits the
		 representation as in (\ref{sto:frame_gl})
		for some $\Lambda$-measurable process $L$ with the integrability condition
		(\ref{sto:frame_gl_2}).	Then we have  
		 that the process $X$ is $\mu$-right-upper-semicontinuous in expectation
		 (see Definition \ref{def_conditions_2} (b))
		 at all $S \in \stm$.
	\end{Pro}

	\begin{proof}
		This is straight forward adaption of the proof of \citing{BK04}{Theorem 2}{1048}, adapted 
		to our stochastic setting.
		\end{proof}		
	
	\begin{Rem}[Necessecity of Left-upper-semicontinuity in expectation]
        One can construct simple examples, which show that there are processes, which are not left-upper-semicontinuous in expectation at every
		$\Lambda$-stopping time and which can not be represented by a process
		$L$ in the form \eqref{sto:frame_gl}. On the other hand there are also simple examples, which
		show that there are processes which can be represented
		as in \eqref{sto:frame_gl} and are not left-upper-semicontinuous in expectation at every
		$\Lambda$-stopping time. Hence, the condition of left-upper-semicontinuity in expectation
		at every $\Lambda$-stopping time is neither necessary nor can we improve the conditions used in Theorem \ref{sto:results_thm_2} in general.
	\end{Rem}
			
    \section{Applications of the Extended Representation Theorem}\label{sec:appl}
    
        Let us discuss briefly the applications mentioned in the introduction and in the discussion after Theorem \ref{sto:results_thm_2}.
        
        \subsection{Irreversible Investment with Inventory Risk}\label{sec:control}
            
            In this section we illustrate in a simple L{\'e}vy process specification of $X$ and $\mu$ how in our companion paper \cite{BB18_2} the representation result Theorem \ref{sto:results_thm_2} leads to different optimal policies 
            for an irreversible investment problem when the information flow is described by different Meyer-$\sigma$-fields. Along the way, we thus will also obtain first nontrivial explicit solutions to our general representation problem \eqref{sto:frame_gl}. 
            
        Controls in our irreversible investment problem are given by increasing processes $C$ starting in $C_{0-} = \varphi$. They generate expected rewards 
            \begin{align}\label{Main:rewards}
                \EW\left[\int_{[0,\infty)} e^{-rt} \tilde{P}_t\md C_{t+}\right].
            \end{align}
       Here, $\tilde{P}$ is an increasing compound Poisson process
		\begin{align}\label{Main:36}
			\tilde{P}_t=\tilde{p}+\sum_{k=1}^{N_t}Y_k,	\quad t\in [0,\infty),
		\end{align}
		where $\tilde{p}\in \RZ$ and where $N$ is a Poisson process with intensity $\lambda>0$, independent of the i.i.d. sequence of \emph{strictly positive} jumps $(Y_k)_{k\in \mathbb{N}}\subset \mathrm{L}^2(\WM)$. 
		
		Controls also incur risk which is assessed at the jump times of the Poisson process $N$:
	    \begin{align}\label{Main:risk}
            \EW\left[\int_{[0,\infty)} e^{-rt}\frac12 C_t^2 \,dN_t\right].
        \end{align}
              
		  The information flow in our control problem is generated by an imperfect sensor which gives timely warnings about impending jumps of $\tilde{P}$ only if these are large enough. Mathematically, this is described by the Meyer $\sigma$-field $\Lambda$ which is the  $\WM$-completion of 
    	\begin{align}\label{Main:92}
			\tilde{\Lambda}= 
			\sigma\left(Z \text{ is $\tilde{\mathcal{F}}^\eta$-adapted and c\`adl\`ag}\right),
		\end{align}
    	where $\tilde{\mathcal{F}}^\eta$ describes the filtration generated by the sensor process 
		\begin{align}\label{Main:56}
		    \tilde{P}^\eta:= \tilde{P}_-+\Delta  \tilde{P}\mathbb{1}_{\{\Delta  \tilde{P} \geq \eta\}}
		\end{align}
		for some sensitivity threshold $\eta\in [0,\infty]$. One can readily check that
		$\mathcal{P}(\mathcal{F})\subset \Lambda\subset \mathcal{O}(\mathcal{F})$. 
		
		The controller's optimization problem can thus be summarized by
		\begin{align}\label{Main:42}
                \sup_{C \geq \varphi \text{ increasing, $\Lambda$-measurable}}
                \EW\left[\int_{[0,\infty)} e^{-rt}\left( \tilde{P}_t\md C_{t+}
                -\frac12 C_t^2  \md N_t\right)\right],
            \end{align}
        where the last expectation is supposed to be $-\infty$ if $P^-$ is not $\P \otimes e^{-rt}\md C_{t+}$-integrable or $C$ is not $\P \otimes e^{-rt}dN_t$-square integrable.
		
		As shown in \cite{BB18_3}, this optimization problem can be reduced to our representation problem~\eqref{sto:frame_gl} by choosing, for $t \in [0,\infty)$,
\begin{align}\label{Main:choices}
    X_t := e^{-rt}\tilde{P}^\eta_t, \quad 
    g_t(\ell) := \ell, \; \ell \in \RZ, \quad
    \mu(dt) := e^{-rt} dN_t,
\end{align}
        and we obtain:
	
			\begin{Thm}
			    If $p(\eta):=\WM(Y_1< \eta)\in(0,1)$, then the representation problem~\eqref{sto:frame_gl} with $X$, $g$, and $\mu$ as in \eqref{Main:choices} and $\Lambda$ the $\WM$-completion of \eqref{Main:92} is solved by the $\Lambda$-measurable process 
			    \begin{align}\label{Main:57}
						L_t^\Lambda = \begin{cases}
							0,& \tilde{P}^\eta_t\geq b,\
					\Delta\tilde{P}^\eta_t \geq \eta,									\vspace{2ex}\\
						\frac{r}{\lambda}	(\tilde{P}^\eta_t-b),
									& \tilde{P}^\eta_t\geq b,\
									\Delta\tilde{P}^\eta_t
									< \eta,\vspace{2ex}\\
						\inf\limits_{\gamma \in \mathcal{D}(\tilde{P}^\eta_t)} 
						f^\eta(\gamma,\tilde{P}^\eta_t) 
							,& \tilde{P}^\eta_t< b,\
									\Delta
						\tilde{P}^\eta_t\geq \eta,\vspace{2ex}\\
						\frac{1}{p(\eta)}\frac{r}{\lambda}
								(\tilde{P}^\eta_t-b),& \tilde{P}^\eta_t
								< b,\ \Delta
												\tilde{P}^\eta_t< \eta,\\
						\end{cases}
				\end{align}
				where $b:=m\frac{\lambda}{r}$,
				$
				\mathcal{D}(z):=\left[0,\left(1-\frac{\lambda}{1+\lambda}p(\eta)\right)(b-z)\right)$,  $z\in (-\infty,b),
				$ 
				and	the function $f^\eta:
				[0,\infty)\times \RZ
				\rightarrow \RZ$
				is given by
				\begin{align}\label{Main:16}
					f^\eta(\gamma,z):=\frac{\left(1-
					\EW\left[\mathrm{e}^{-rT(\gamma)}\right]\right)z
					-\EW\left[\mathrm{e}^{-rT(\gamma)} 
					\sum\limits_{k=1}^{N_{T(\gamma)}}Y_k\right]}{1+
					\frac{\lambda}{r}\left(1-
					\EW\left[\mathrm{e}^{-rT(\gamma)}\right]\right)
					-\EW\left[\mathrm{e}^{-rT(\gamma)}
					\mathbb{1}_{\{Y_{N_{T(\gamma)}}\geq \eta\}}\right]}
				\end{align}
				with
				\begin{align}\label{Main:52}
					T(\gamma):=\inf\left\{t\in \{\Delta N>0\}\midG 
					Y_{N_t}\geq \eta \text{ or }
					 \sum_{k=1}^{N_t} Y_{k}\geq \gamma \right\}.
				\end{align}
			    Furthermore, an optimal strategy for \eqref{Main:42} is given by the l{\`a}dl{\`a}g control
				\[
					C_t^{\Lambda}:=\uP \vee\sup_{v\in [0,t]} 
					L^{\Lambda}_v, \quad t \in [0,\infty). 
				\]
			\end{Thm}
			\begin{proof}
			 This is a reformulation of Theorem~2.5 in \cite{BB18_3}.
			\end{proof}
			\begin{Rem}
			The maximal solution for the predictable case $p(\eta)=1$ is shown in \cite{BB18_3} to emerge as the limit $L^{\mathcal{P}}:=\lim_{\eta\uparrow \infty} L^{\Lambda}$;
			for the optional case $p(\eta)=0$ it is shown there that a solution emerges from $L^{\mathcal{O}}:=\lim_{\eta\downarrow 0} L^{\Lambda}$.
		\end{Rem}
		
		We want to conclude	with Figure \ref{figure:6}, which plots a trajectory of the process $\tilde{P}$ (brown) with its critical level $b=5.5$ (grey) along with the processes $L^\Lambda$ for three different choices of $\eta \in \{ 5,6,16\}$.
				\begin{figure}[h]
					\centering
					\includegraphics[width=0.49\textwidth]{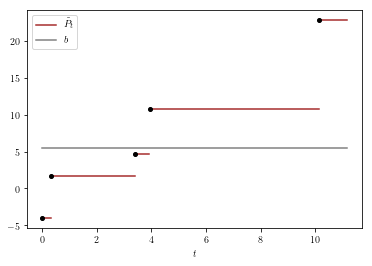}
						\includegraphics[width=0.49\textwidth]{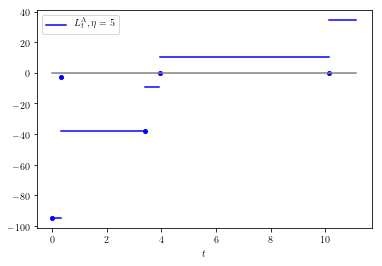}	\\
						\includegraphics[width=0.49\textwidth]{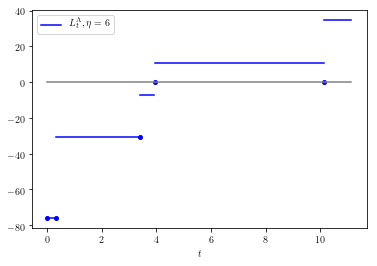}
						\includegraphics[width=0.49\textwidth]{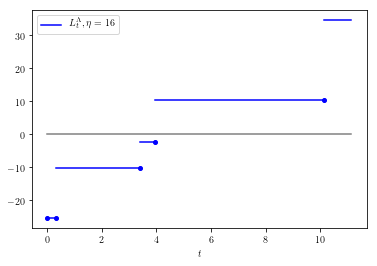}
					
					\caption{Upper left corner:
				$\tilde{P}_t$,  Upper right corner: $L^\Lambda$ for $\eta=5$,
				Lower left corner: $L^\Lambda$ for $\eta=6$, Lower right corner: $L^\Lambda$ for $\eta=16$}
					\label{figure:6}
				\end{figure} 
				As one can see, the solution $L^\Lambda$ (and hence the induced optimal control $C^\Lambda$) strongly depends on the considered sensor sensitivity $\eta$. We refer to our companion paper \cite{BB18_3} for a more thorough discussion.
        
        \subsection{An Optimal Stopping Problem over Divided Stopping Times}\label{sec:stopping}
            
             An extension of classical stopping times is given by divided stopping times introduced in \cite{EK81}:
			\begin{Def}[\citing{EK81}{Definition 2.37}{136-137}]\label{app:optstop_def_4}
    			A given quadrupel $\tau:=(T,H^-,H,H^+)$ is called a \emph{divided
    			($\Lambda$-)stopping time}, if $T$ is an $\mathcal{F}$-stopping time
    			and $W^-,W,W^+$ build a partition 
    			of $\Omega$ such that
    			\begin{compactenum}[(i)]
    				\item $W^-\in \mathcal{F}_{T-}$ and $W^-\cap \{T=0\}=\emptyset$,
    				\item $W\in \mathcal{F}^\Lambda_T$,
    				\item $W^+ \in \mathcal{F}_{T}^\Lambda$ and $W^+\cap 
    				\{T=\infty\}=\emptyset$,
    				\item $T_{W^-}$ is an $\mathcal{F}$-predictable 
    				stopping time,
    				\item $T_{W}$ is a $\Lambda$-stopping time.
    			\end{compactenum}
    			The set of all such divided stopping times will be denoted 
    			as $\stm^{\, \mathrm{div}}$.			
    			For a $\Lambda$-measurable positive process $Z$
    			we define the values attained 
    			at a divided stopping time $\tau=(T,H^-,H,H^+)$ as
    			\begin{align}\label{sto:eq_9}                 
    			Z_\tau=\lsl{Z}_{T}\mathbb{1}_{H^-}+Z_T\mathbb{1}_{H}
    				+\lsr{Z}{T}\mathbb{1}_{H^+}.
    			\end{align}
        	\end{Def}
        	Here we used the notation $^\ast Z$, 
			$Z^\ast$ for the left- and
			right-uppercontinuous envelopes of a process $Z:\Omega\times [0,\infty)\rightarrow \RZ$
			defined for $t\in [0,\infty)$ by
			\begin{align*}
					\lsr{Z}{t}(\omega)&:=\limsup_{s\downarrow t} Z_s(\omega)
						:=\lim_{n\rightarrow \infty} \sup_{s\in (t,t+\frac1n)} 
						Z_s(\omega), \quad &\lsr{Z}{\infty}(\omega)
						:=Z_{\infty}(\omega),
			\end{align*}
			and for $t\in (0,\infty)$ by
			\begin{align*}
				&\lsl{Z}_t(\omega):=\limsup_{s\uparrow t} Z_s(\omega)
						:=\lim_{n\rightarrow \infty} \sup_{s\in (t-\frac1n,t)
						\cap [0,\infty)} Z_s(\omega),\nonumber \\
					& \lsl{Z}_{0}(\omega):=Z_0(\omega), \quad
					\lsl{Z}_{\infty}(\omega):=\limsup_{t
					\uparrow \infty} Z_t(\omega):=\lim_{n\rightarrow \infty}
					\sup_{s\in [n,\infty)} Z_s(\omega).
			\end{align*}
        	
			One main benefit of divided stopping times comes from the fact that an optimal divided stopping time always exists 
			(see \citing{EK81}{Theorem 2.39}{138}). More precisely, assume we have 
			a $\Lambda$-measurable non-negative process $Z$ of class($D^\Lambda$).
			Then there exists a divided stopping time $\hat{\tau}:=(\hat{T},\hat{H}^-,\hat{H},\hat{H}^+)$ attaining the supremum
			\[
				\sup_{\tau \in \stm^{\, \mathrm{div}}}
				\EW[Z_\tau]=\EW[{^\ast Z_{\hat{T}}}\mathbb{1}_{\hat{H}^-}+{Z_{\hat{T}}}\mathbb{1}_{\hat{H}}+ Z_{\hat{T}}^\ast\mathbb{1}_{\hat{H}^+}].
			\]
			For more on divided stopping times we refer to \cite{BB18}, Section~3.5.
		
			In the rest of this section we want to tackle the question how to describe an optimal divided stopping time for 
			\begin{align}\label{eq:0}
				\sup_{\tau\in \stmd}
				\EW\left[X_\tau+\int_{[0,\tau)}g_t(\ell)\mu(\md t)\right],
			\end{align}
		    where $g$ and $\mu$ satisfy Assumption \ref{sto:frame_ass}, 
		    $\ell\in \RZ$ is fixed, and where $X$ satisfies the assumptions stated in Theorem~\ref{sto:results_thm_2}. In the optional case, this is discussed in \citing{BF03}{Theorem 2}{6}, albeit only for atomless $\mu$ with full support. Without these regularity properties, optimal stopping times can no longer be expected from a representation as in Theorem~\ref{sto:results_thm_2}. But we still can describe optimal divided stopping times in terms of the representing process $L$ and thus provide an optimal stopping characterization alternative to the Snell-envelope approach of \citing{EK81}{Theorem 2.39}{138}:
		    
		    \begin{Thm}\label{thm:opt_stop}
                Let $X$, $g$, and $\mu$ be as in Theorem~\ref{sto:results_thm_2} and denote by $L$ the unique $\Lambda$-measurable process with
			\eqref{sto:frame_gl}, \eqref{sto:results_gl} and \eqref{sto:frame_gl_2}. Then $L$ is a universal stopping signal for~\eqref{eq:0} in the sense that, for any $\ell \in \RZ$, the divided stopping time
    			\[
    					\tau_\ell:=(T_{\ell},\emptyset,\{L_{T_{\ell}}\geq \ell\},
    				\{L_{T_{\ell}}<\ell\}),
    			\]
    			with
    			\[
    				T_\ell:=\inf\left\{t\geq 0\midG \sup_{v\in [0,t]}L_v
    				\geq \ell\right\}
    			\]
    			attains the supremum in \eqref{eq:0}.
            \end{Thm}
            
            This theorem is actually a corollary to our construction of a solution to our representation problem \eqref{sto:frame_gl} and its proof is thus deferred to Appendix \ref{sec:opt_stop}.
            
	    \section{Proof of the upper bound for $\Lambda$-measurable solutions to the representation problem}\label{sto:sec_upper}

		In this section we will prove the maximality of the solution
		to \eqref{sto:frame_gl} in the sense described in Theorem \ref{sto:results_thm_2}, which satisfies \eqref{sto:results_gl}
		and \eqref{sto:frame_gl_2},
		 assuming it exists. This existence will be proven in Section \ref{sto:sec_ex}.
	    More precisely we deduce an upper bound on all $\Lambda$-measurable
	    solutions to \eqref{sto:frame_gl} and \eqref{sto:frame_gl_2}, which will
	    be attained by the solution, which additionally satisfies \eqref{sto:results_gl}. The proofs of the upcoming results can be found in appendix \ref{app:A}.
	
    	Let us first note a result, which shows that the 
    	random variables $\ell_{S,T}$ from \eqref{sto:results_gl_2} exist.
		
		\begin{Lem}\label{Lem:um1}
			For any $S\in \stmr$ and $T\in \stmsg{S}$, there exists a unique (up to a $\WM$-null set) random	variable $\ell_{S,T}\in \mathrm{L}^0(\aFA_S)$  
			such that we have 
			\begin{align}\label{Gl:sp30}
				\EW\left[X_S-X_T\midG \aFA_S\right]=
				\EW \left[\int_{[S,T)} g_t(\ell_{S,T})
					\mu(\md t) \midG \aFA_S\right] \text{ on $\{\WM[\mu([S,T))>0|\aFA_S]>0\}$}
			\end{align}
			and 
			\[
				\ell_{S,T}=\infty \ \text{ on }\ \{\WM[\mu([S,T))>0|\aFA_S]=0\}.
			\]
		\end{Lem}
		
		The next result states an upper bound on all $\Lambda$-measurable processes which satisfy \eqref{sto:frame_gl} and \eqref{sto:frame_gl_2}:
	
	    \begin{Pro}\label{Main:7}
	        Any $\Lambda$-measurable process $\tilde{L}$ satisfying mutatis mutandis \eqref{sto:frame_gl} and \eqref{sto:frame_gl_2} will fulfill
	        \[
	            \tilde{L}_S \leq \essinf_{T\in \stmsg{S}} \ell_{S,T}, \quad 
						S\in \stmr,
	        \]
	        where for $S\in \stmr$ and $T\in \stmsg{S}$, $\ell_{S,T}$ is defined in Lemma \ref{Lem:um1}.
	    \end{Pro}
		
		\begin{Cor}\label{Main:2}
		    A solution $L$ satisfying \eqref{sto:frame_gl}, \eqref{sto:results_gl}
		and \eqref{sto:frame_gl_2} 
		    is maximal in the sense that 
					$\tilde{L}_S\leq L_S$ for any $S\in \stmr$					
					for every other $\Lambda$-measurable process $\tilde{L}$ 
					satisfying, 	mutatis mutandis, 
					(\ref{sto:frame_gl}) and (\ref{sto:frame_gl_2}).
		\end{Cor}

\section{Construction of a solution to the representation problem}\label{sto:sec_ex}

	In this section we will construct a solution to the
	representation	problem \eqref{sto:frame_gl} stated in Theorem
	\ref{sto:results_thm_2}. As in \cite{BK04} the idea will be to
	introduce suitable stopping problems which can be analyzed
	using the general results of \cite{EK81}. \cite{EK81} also used the theory
	of Meyer-$\sigma$-fields, developed by
	\cite{EL80}, and introduced stopping problems
  for those $\sigma$-fields. One key tool to describe optimality
	results but also to get an intuition about those problems is the Snell 
	envelope. We will also use this concept and more precisely we
	construct, as in \cite{BK04},
	 a regular version of a family of Snell envelopes
	$(Y^\ell)_{\ell\in \RZ}$ given by
	 \begin{align*}
			Y^{\ell}_S=\esssup_{T\in \stmg{S}}\EW\left[X_T+\int_{[S,T)}
			g_t(\ell)\mu(\md t)\midG \aFA_S\right],
			\ S\in \stm, 
	\end{align*}
	and show that a solution $L$ to the representation problem
	(\ref{sto:frame_gl}) is given by
	\begin{align*}
				L_t(\omega):=\sup\left\{\ell\in \RZ \midG
										Y^{\ell}_t(\omega)=X_t(\omega)\right\},
										\quad (\omega,t)\in \Omega\times [0,\infty),
	\end{align*}	
	i.e. $L$ is the maximal value $\ell$ for which
	the optimal stopping problem introduced by $Y^\ell$ is solved by 
	stopping immediately. 
	
	One key problem in our work will be that, for fixed $\ell$, we will in
	general	not get  a stopping time which
	 solves the optimal stopping problem introduced by $Y^\ell$. 
	This is due to the atoms of $\mu$ and because the process $X$
	will not be pathwise right-upper-semicontinuous in general in contrast to the situation
	of \cite{BK04}. Therefore we will
		use the so-called ``temps divis\'ees'', which we call divided 
		stopping times. These will be the solutions to a suitable relaxation of our initial optimal stopping problem.
		
		This section will be organized as follows. We start with a short explanation about a convenient transformation
		of $g$.
		Afterwards we construct in
		Section \ref{sto:ssec_143} the mentioned processes $(Y^\ell)_{\ell\in \RZ}$ and its aggregation
		$Y$. In Section
		\ref{sto:ssec_145} we define $L$ and prove some properties of it, which we will use
		in Section \ref{sto:ssec_146} to show that $L$ solves the representation problem
		\eqref{sto:frame_gl} with integrability condition \eqref{sto:frame_gl_2}.				
		
	    \paragraph{Normalization assumption on $g$:} We will assume that
	    $g$ is normalized in the sense that
		\begin{align}\label{sto:sep_gl}
			g_t(\omega,0)=0,\quad (\omega,t)\in \Omega\times [0,\infty).
		\end{align}
		As a consequence, we then have
		\begin{align*}
			g_t(\omega,\ell)	\begin{cases}
									<0 	&	\text{ for } \ell<0,\\
									=0 	&	\text{ for } \ell=0,\\			
									>0 	&	\text{ for } \ell>0,
							\end{cases}\quad (\omega,t)\in \Omega\times [0,\infty).
		\end{align*}		
		This normalization is in fact without loss of generality as we could define the processes $\tilde{g}$ and $\tilde{X}$ by
		\[
			\tilde{g}(\ell):=g(\ell)-g(0),
			\quad \tilde{X}:=X-\sideset{^\Lambda}{}{\mathop{\left(\int_{[\cdot,\infty)}
					 g_t(0)\mu (\md t)\right)}}.			
		\]
		These two processes again fulfill the assumptions of Theorem \ref{sto:results_thm_2}
		and if there exists a solution $L$ to (\ref{sto:frame_gl}) for $\tilde{X}$
		and $\tilde{g}$ then it is already a solution to (\ref{sto:frame_gl}) for $X$
		and $g$.
	    
	    \subsection{A family of optimal stopping problems}\label{sto:ssec_143}

	Our first lemma in this section introduces some auxiliary stopping problems and 
	provides a suitably regular choice of the corresponding Snell-envelopes $(Y^\ell)_{\ell \in \RZ}$ which
	will be crucial for the construction of the maximal solution $L$ to our representation problem
	\eqref{sto:frame_gl} in Lemma \ref{Lem:s4}. The proof of this lemma will be carried out in Section \ref{proof:4.1}. 
		
	\begin{Lem}\label{Lem:sp3}
			There is a jointly measurable mapping
			\begin{align*}
				Y:\Omega\times [0,\infty] \times \RZ &\rightarrow \RZ,\\
				(\omega,t,\ell)&\mapsto Y^\ell_t(\omega)
			\end{align*}
			with the following properties:
		\begin{compactenum}[(i)]
			\item 	For each $\ell\in \RZ$, the process
					 $Y^\ell:\Omega\times [0,\infty]\rightarrow \RZ$,
					$(\omega,t)\mapsto Y_t^\ell(\omega)$
					  is $\Lambda$-measurable, l\`adl\`ag and of class($\text{D}^\Lambda$) with
					  $Y^\ell_\infty=0$ such that for all $S\in \stm$ we have almost surely
					\begin{align}\label{sto:eq_5}
						Y^\ell_S=\esssup_{T\in \stmg{S}} 
					 		\EW\left[X_T+\int_{[S,T)}g_t(\ell)\mu(\md t)\midG \aFA_S\right].
					\end{align}			
			\item Define the l\`adl\`ag $\Lambda$-measurable processes $E^\ell$ of class($D^\Lambda$), $\ell\in \RZ$, 
				 by
						\[
							E^\ell:=\int_{[0,\cdot)} g_t(\ell)\mu(\md t)
									+\sideset{^\Lambda}{}{\mathop{\left(\int_{[0,\infty)}
									 |g_t(\ell)|\mu(\md t)\right)}}+M^X+1
						\]
						and
						\[
							E^\ell_{\infty}:=\int_{[0,\infty)} g_t(\ell)\mu(\md t)
									+\int_{[0,\infty)}
									 |g_t(\ell)|\mu(\md t)+M^X_{\infty}+1
						\]
					with $M^X$ as in Lemma \ref{sto:tools_lem}. Then there is a version of the stochastic field $(E^\ell)_{\ell\in \RZ}$
					such that
					the following holds for all $\omega \in \Omega$:
					\begin{enumerate}
						\item[(1)] Uniform continuity in $\ell \in \RZ$, i.e.  
								   \begin{align}\label{sto_conv_ass_2}
										\lim_{\delta \downarrow 0}\sup_{\overset{\ell,\ell'\in C}{|\ell'-\ell|
		\leq \delta}}\sup_{t\in [0,\infty]}
										 |E^{\ell'}_t(\omega)-E_t^\ell(\omega)|=0
					\end{align}
					for all compact sets $C\subset \RZ$.
					\item[(b)] Monotonicity in $\ell\in \RZ$, i.e. 
					for all $\ell,\ell'\in \RZ$ with
					$\ell\leq \ell'$ and all
								$t\in [0,\infty]$,
								$E^{\ell}_t(\omega)\leq E^{\ell'}_t(\omega)$.
					\item[(c)] L\`adl\`ag paths for $\ell\in \RZ$, i.e. for any
					$\ell\in \RZ$ the mapping $t\mapsto 
							E_t^\ell(\omega)$ is real valued and l\`adl\`ag. In
							particular the paths $t\mapsto E^\ell_{t+}(\omega)$ 
							and $t\mapsto E^\ell_{t-}(\omega)$ are bounded on
							compact intervals.			
					\item[(d)]	We have
							\begin{align}\label{sto:useful_gl_23}
								Y^\ell_t(\omega)+E_t^\ell(\omega)\geq 1 \quad \text{ for all } \ell\in \RZ \text{ and all }
								 t\in [0,\infty].
							\end{align}
					\end{enumerate}
			\item \label{sto:item_1} For any $\ell\in \RZ$
						and $S\in \stm$, 
				the family of $\aFA_+$-stopping times
					\[
						T_{S,\ell}^\lambda(\omega)
						:=\inf\left\{t\in [S(\omega),\infty] \midG X_t(\omega)\geq 
						\lambda Y^\ell_t(\omega) -(1-\lambda) E_t^\ell(\omega)\right\},\ \lambda\in [0,1), 
					\]
					is non-decreasing in $\lambda$ for all $\omega\in \Omega$ with limit
					\[
						\lim_{\lambda\uparrow 1} T_{S,\ell}^\lambda =:T_{S,\ell}.
					\]
					Moreover we have on all of $\Omega$ that
					\begin{align}\label{sto:useful_gl_8}
						T_{S,\ell}=\min 
							\left\{t\in [S,\infty] \midG 
											Y_t^{\ell}=X_t 
											 \text{ or } Y_{t-}^{\ell}=\lsl{X}_t
											\text{ or } Y_{t+}^{\ell}=\lsr{X}{t}
											\right\},
					\end{align}
					and, for every $\omega\in \Omega$, the mapping $\ell\mapsto T_{S,\ell}(\omega)$
					is non-decreasing.
					\item The following inclusions hold for any $S\in \stm$, $\ell\in \RZ$:
					\begin{align}
						 H_{S,\ell}^-&:=\left\{ T_{S,\ell}^\lambda<	T_{S,\ell}
										\text{ for all } \lambda\in [0,1)\right\}
										\subseteq \left\{		Y^\ell_{T_{S,\ell}-}
						=\lsl{X}_{T_{S,\ell}}\right\},\label{sto:eq_277}\\
						 H_{S,\ell}&:=\left\{ Y^\ell_{T_{S,\ell}}
									=X_{T_{S,\ell}}\right\}\cap
						\left(H^-_{S,\ell}\right)^c \subseteq \left\{Y^\ell_{T_{S,\ell}}
						=X_{T_{S,\ell}}\right\},\label{sto:eq_1277}\\
						 H_{S,\ell}^+&:=\left\{Y^\ell_{T_{S,\ell}}>X_{T_{S,\ell}}\right\}
						\cap \left({ H_{S,\ell}^-}\right)^c
						\subseteq \left\{	Y^\ell_{T_{S,\ell}+}
						=\lsr{X}{T_{S,\ell}}\right\}.\label{sto:eq_77}
					\end{align}
					The sets $H_{S,\ell}$ and $H_{S,\ell}^+$ are contained 
					in $\mathcal{F}_{T_{S,\ell}}^\Lambda$ and
			$H_{S,\ell}^-\in \mathcal{F}_{T_{S,\ell}-}^\Lambda$.
					Moreover, we have up to a $\WM$-null set
					\begin{align}\label{sto:eq_12}
						H_{S,\ell}^-\cup H_{S,\ell}
						=\left\{Y^\ell_{T_{S,\ell}}
						=X_{T_{S,\ell}}\right\},\quad
						\left({ H_{S,\ell}^-}\right)^c=
						\left\{Y^\ell_{T_{S,\ell}}>X_{T_{S,\ell}}\right\}	,				
					\end{align}
					and $(T_{S,\ell})_{H_{S,\ell}^-}$ is an $\FA$-predictable stopping
					time, $(T_{S,\ell})_{H_{S,\ell}}$ is a $\Lambda$-stopping
					time and $(T_{S,\ell})_{H_{S,\ell}^+}$ is an $\aFA_+$-stopping time. 
				\item For $S\in \stm$, $\ell\in \RZ$, the quadrupel
					\begin{align}	\label{sto:eq_69}				
						\tau_{S,\ell}:=(T_{S,\ell}, H_{S,\ell}^-, H_{S,\ell}, H_{S,\ell}^+)
					\end{align}					
					 is a divided 
					stopping time (see Definition \eqref{app:optstop_def_4}).
					
				\item For $S\in \stm$, the mapping $\ell\mapsto \tau_{S,\ell}\in \stmd$
					is increasing in the sense that 
					\begin{align}\label{sto:eq_50}					
						[S,\tau_{S,\ell})
					\subset [S,\tau_{S,\ell'})
					\end{align}
					for all $\ell,\ell'\in \RZ$
					with $\ell\leq \ell'$, where the definition of
					$[S,\tau_{S,\ell})$ is given by
        			\begin{align}\label{sto:eq_4}
        				[S,\tau_{S,\ell}):=\begin{cases}
        									[S,T)&\text{ on } H^-\cup H,\\
        									[S,T]&\text{ on } H^+.\\
        								\end{cases}
        			\end{align}			
					Furthermore for $\ell\in \RZ$ the divided stopping time
					$\tau_{S,\ell}$ attains the value of the optimal stopping problem
					in \eqref{sto:eq_5}, i.e. almost surely
					\begin{align}\label{sto:useful_gl_16}
						Y^\ell_S=\ &\EW\left[X_{\tau_{S,\ell}}
								+\int_{[S,\tau_{S,\ell})} g_t(\ell)\mu(\md t)
								\midG \aFA_S\right],
					\end{align}
					where $X_{\tau_{S,\ell}}$ follows definition
					\eqref{sto:eq_9}.
				\item For $S\in \stm$ and $\ell,\ell' \in \RZ$, we have almost surely that
					\begin{align*}
						Y_S^\ell\geq \ &\EW\left[X_{\tau_{S,\ell'}}
								+\int_{[S,\tau_{S,\ell'})} g_t(\ell)\mu(\md t)\midG \aFA_S\right].
					\end{align*}
					 Moreover, for fixed $\ell\leq \ell'$, we have
					\begin{align}\label{sto:useful_gl_24}
						Y^{\ell'}_s\geq Y^{\ell}_s \geq Y^{\ell'}_s
							+\ {^\Lambda \left(\int_{[0,\infty)}
								 g_t(\ell)\mu(\md t)\right)}_s
								 -\ {^\Lambda \left(\int_{[0,\infty)} 
								 g_t(\ell')\mu(\md t)\right)}_s,
					\end{align}
					 for all $s\in [0,\infty]$ and all $\omega\in \Omega$, where the latter $\Lambda$-projections
					 are chosen such that
					 for all $\omega\in \Omega$ we have:
					\begin{align*}
										\lim_{\delta \downarrow 0}\sup_{\overset{\ell,\ell'\in C}{|\ell'-\ell|
		\leq \delta}}\sup_{t\in [0,\infty]}
										 \left|{^\Lambda \left(\int_{[0,\infty)}
								 g_t(\ell)\mu(\md t)\right)}_t-{^\Lambda \left(\int_{[0,\infty)}
								 g_t(\ell')\mu(\md t)\right)}_t\right|=0
					\end{align*}
					for all compact sets $C\subset \RZ$.
			\item \label{sto:useful_item_1} For fixed $(\omega,s)\in \Omega\times [0,\infty]$, the mapping $\ell \mapsto
							Y^\ell_s(\omega)$ is continuous and non-decreasing. 
							Furthermore, we have for every $S\in \stm$, that almost surely
							\[
								Y^{-\infty}_S:=\lim_{\ell\downarrow -\infty} 
								Y^\ell_S= X_S, 
							\]
							i.e. $X=\inf_{\ell \in \mathbb{Q}} Y^\ell$ up to
							indistinguishability.
				\item 		For $S\in \stm$, we have $X_S=Y_S^\ell$ and $T_{S,\ell}=S$ for all $\ell\in \RZ$ almost surely on the set 
						$\{\WM\left(\mu([S,\infty)>0)\midG \aFA_S\right)=0\}$.
			\end{compactenum}
		\end{Lem}	
		
    \subsection{Construction of the solution}\label{sto:ssec_145}

		With the help of the stochastic field
		$Y=(Y^\ell_t)_{\ell \in \mathbb{R}, t \geq 0}$ we will now construct the process $L$ which will turn out to be the solution to our
		 stochastic representation problem. At any time $t \geq 0$, it is defined in the same 
		way as in \citing{BK04}{Lemma 4.13}{1051} as
			the threshold value $\ell\in \RZ$ up to which one would immediately stop
			in the optimal stopping problems with the Snell envelopes  $(Y^\ell_t)_{\ell\in \RZ}$:
		
		\begin{Lem}\label{Lem:s4}
					For $Y=(Y^\ell)_{\ell\in \RZ}$ as in Lemma \ref{Lem:sp3}, the process $L$ defined by
					\begin{align}\label{sto:eq_43}
						L_t(\omega):=\sup\left\{\ell\in \RZ \midG
										Y^\ell_t(\omega)=X_t(\omega)\right\},\quad (\omega,t)\in \Omega\times [0,\infty)		,			
					\end{align}
					and
					\[
						 L_{\infty}(\omega):=\infty, \quad \omega\in \Omega,
					\]
					is $\Lambda$-measurable. Furthermore we have for $S\in \stm$ that
					\[
						\WM(\{L_S=\infty\}\cap \{\WM\left(\mu([S,\infty))>0
						\midG \aFA_S\right)>0\})=0
					\]
					and
					\[
						\WM(\{L_S=-\infty\}\cap \{\WM\left(\mu(\{S\})>0
						\midG \aFA_S\right)>0\})=0.
					\]
		\end{Lem}

        For the proof of this and the other lemmas in this section we refer to Section \ref{proof:4.2}, \ref{proof:4.3} 
        and \ref{proof:4.2} below.
		
		Next we see that the process $L$ constructed in Lemma \ref{Lem:s4}
		is the essential infimum over 
		  the family of random variables $\ell_{S,T}$ introduced in Theorem \ref{sto:results_thm_2}. This
		  will imply the maximality of the solution in the sense of Theorem \ref{sto:results_thm_2} by Corollary \ref{Main:2}.
			
				\begin{Lem}\label{Lem:s6}
					For $L$ as in the previous Lemma \ref{Lem:s4}
					we have
					\begin{align}\label{sto:sol_gl_1}
						L_S=\essinf_{T\in \stmsg{S}} \ell_{S,T},\quad S\in \stmr,
					\end{align}					
					where $\ell_{S,T}$ is defined in Lemma \ref{Lem:um1} or Theorem \ref{sto:results_thm_2}. Moreover,
					 with $Y$ from Lemma \ref{Lem:sp3} we have that
					 \begin{align}\label{sto:eq_1}
					 	X_S=Y_S^{L_S} \quad \text{ almost surely for any } 
						S\in \stmr.
					 \end{align}
			\end{Lem}						
					
			Next we clarify how $L$ to the stopping times $T_{S,\ell}$ ($\ell\in \RZ$, $S\in \stm$) constructed
		  in Lemma \ref{Lem:sp3}. Our result reveals that $L$ can be seen as a threshold for those optimal stopping problems
		  as already mentioned in Section~\ref{sec:stopping}.
				 
		\begin{Lem}\label{Lem:s5}
					For every $S\in\stm$ there exists a $\WM$-null set $\mathcal{N}$ 
					such that with
					\begin{align}\label{eq_267}
						\bar{\Omega}_S^\mathcal{N}:=\left\{(\omega,t,\ell)\in \Omega\times[0,\infty)\times \RZ\midG
							\omega\in \mathcal{N}^c,\ S(\omega)\leq t\right\}
					\end{align}				
					the stopping times $(T_{S,\ell})_{\ell\in \RZ}$ from Lemma 
					\ref{Lem:sp3} (iii) and the process $L$ from Lemma \ref{Lem:s4} are related by the inclusions
					\begin{align}
						&\ A:=\left\{(\omega,t,\ell)\in \bar{\Omega}_S^\mathcal{N}
						\midG \sup_{v\in [S(\omega),t]}
										L_v(\omega)<\ell \right\}\label{sto:set_a}\\
					\subset&\ 
						B:=\left\{(\omega,t,\ell)\in \bar{\Omega}_S^\mathcal{N}
						\midG  t\leq T_{S,\ell}(\omega)\right\}\label{sto:set_b}\\
						\subset &\ C:=\left\{(\omega,t,\ell)\in \bar{\Omega}_S^\mathcal{N}		\midG 
							\sup_{v\in [S(\omega),t)} L_v(\omega)
							\leq \ell\right\}\label{sto:set_c}
					\end{align}
					and
					\begin{align}
						  &\ \tilde{A}:=A \cap \left\{(\omega,t,\ell) \in 
							\bar{\Omega}_S^\mathcal{N}
						\midG X_{T_{S,\ell}(\omega)}(\omega)=Y_{T_{S,\ell}(\omega)}^\ell
						(\omega)\right\}\label{sto:set_ta}\\
				\subset&
					\	\tilde{B}:= \left\{(\omega,t,\ell) 
					\in \bar{\Omega}_S^\mathcal{N}
							\midG t<T_{S,\ell}(\omega)\right\}\label{sto:set_tb}		\\			
						\subset&\ \tilde{C}:=\left\{(\omega,t,\ell) 
						\in \bar{\Omega}_S^\mathcal{N}
							\midG \sup_{v\in [S(\omega),t]}
										L_v(\omega) \leq \ell\right\}.\label{sto:set_tc}
					\end{align}
				\end{Lem}	
     
	\subsection{Verification of the solution}\label{sto:ssec_146}
		
		We follow the blueprint of the proof from \cite{BK04} and we prove
		 in four steps that the process $L$ of Lemma
		\ref{Lem:s4} satisfies
		\eqref{sto:frame_gl}, i.e. that
		\begin{align}\label{sto:eq_346}
			X_S=\EW\left[\int_{[S,\infty)}g_t\left(\sup_{v\in [S,t]}
			L_v\right)\mu(\md t)
			\midG \aFA_S\right]
		\end{align}
		for any $S\in \stm$, establishing among the way also the
			 integrability property \eqref{sto:frame_gl_2}.
			The idea of the proof is the following. We first show 
			that  for any $S\in \stm$ and all $\ell \in \mathbb{Q}$ we have that
			\begin{align}\label{sto:eq_27}
			X_S=&\EW\left[X_{\tau_{S,\ell}}
									+\int_{[S,\tau_{S,\ell})} g_t\left(\sup_{v\in [S,t]}L(v)\right)
					\mu(\md t)
						\midG \aFA_S\right].
		\end{align}
		Afterwards we let $\ell$ tend to infinity which lets the
		$X_{\tau_{S,l}}$-term in the preceding expectation vanish while the integral converges
		 to an integral over all of $[S,\infty)$. This then establishes
		the desired representation~\ref{sto:frame_gl}. 
		
		The starting point~\eqref{sto:eq_27} will be established in Lemma \ref{sto:Lem_12} using a disintegration formula from the following Proposition \ref{sto:dis_pro};
		that $X_{\tau_{S,\ell}}$ vanishes as $\ell \uparrow \infty$ is obtained in Lemma \ref{sto:Lem_13}. All the proofs of the upcoming results are deferred to Section~\ref{proof:4.5}, \ref{proof:4.6} and \ref{proof:4.7} to allow us to conclude this section with the proof of our main result Theorem~\ref{sto:results_thm_2}.
							
		We start now with the following disintegration formula: 

    	\begin{Pro}\label{sto:dis_pro}
    			For every $S\in \stm$, the nonnegative random Borel-measure $Y_S(\md \ell)$
    							associated with the non-decreasing continuous random mapping
    							$\ell\mapsto Y_S^\ell$ (see Lemma
    							\ref{Lem:sp3} (viii)) can be disintegrated in the form 
    							\begin{align}\label{sto:dis_gl_1}
    								\int_{\RZ} \phi(\ell)Y_S(\md \ell)=
    										\EW\left[\int_{[S,\infty)}\left\{\int_{\RZ}\phi(\ell)
    										\mathbb{1}_{[S,\tau_{S,\ell})}(t)\
    										 g_t(\md \ell)\right\}\mu(\md t)\midG \aFA_S\right]
    							\end{align}
    							for any nonnegative, $\aFA_S\otimes \mathcal{B}(\RZ)$-measurable 
    							$\phi:\Omega\times\RZ
    							\rightarrow \RZ$. Here $\tau_{S,\ell}$ is the divided 
    							stopping time from \eqref{sto:eq_69} and $[S,\tau_{S,\ell})$ is given by \eqref{sto:eq_4}.
    	\end{Pro}
		
		Now the following lemma establishes \eqref{sto:eq_27}:
		
		\begin{Lem}\label{sto:Lem_12}
		For fixed $S\in \stm$ and any $\ell \in \RZ$, we have $\mathbb{1}_{[S,\tau_{S,\ell})}
		g(\sup_{v\in [S,t]}L_v)\in \mathrm{L}^1(\WM\otimes\mu)$ with
		$\tau_{S,\ell}$ from \eqref{sto:eq_69}
		and
		\begin{align}\label{sto:eq_42}
			X_S=&\EW\left[X_{\tau_{S,\ell}}
									+\int_{[S,\tau_{S,\ell})} g_t\left(\sup_{v\in [S,t]}L(v)\right)
					\mu(\md t)
						\midG \aFA_S\right].
		\end{align}
		\end{Lem}
		
	    As a last preparation step, the following Lemma will allow us to let 
	    $\ell$ converge to infinity in \eqref{sto:eq_27}:
	
		\begin{Lem}\label{sto:Lem_13}
			For fixed $S\in \stm$ and $T_{S,\infty}:=\lim_{\mathbb{Q}\ni l\uparrow \infty} T_{S,\ell}$	we have the following:
		\begin{enumerate}[(i)]
			\item We have
				\begin{align}\label{sto:eq_36}
					\mu((T_{S,\infty},\infty))=0,\quad \WM\text{-almost surely.}
				\end{align}						
			\item Additionally, we have with
					\[
						\Gamma:=\bigcap_{n=1}^\infty \left\{T_{S,n}<T_{S,\infty}\right\}
					\]
					that
					\begin{align}\label{sto:eq_101}
						&\WM(\Gamma \cap \{\mu(\{T_{S,\infty}\})>0\})=0,\\														
					&\WM(\Gamma^c \cap \{\mu(\{T_{S,\infty}\})>0\}
					\cap \{X_{T_{S,\infty}}=Y_{T_{S,\infty}}^\ell
						\text{ for all $\ell$}\})=0,\label{sto:eq_104}
				\end{align}
				and almost surely we get the pointwise limit
				\begin{align}\label{sto:eq_105}
					\lim_{n\rightarrow \infty} \mathbb{1}_{(H_{S,n}^-\cup H_{S,n})\cap\Gamma^c}
					=\mathbb{1}_{\Gamma^c\cap\{X_{T_{S,\infty}}=Y_{T_{S,\infty}}^\ell
						\text{ for all $\ell$}\}}.
				\end{align}
			\item Finally, we have 
				\begin{align}\label{sto:eq_38} 
					\lim_{\mathbb{Q}\ni \ell\uparrow \infty} \EW\left[X_{\tau_{S,\ell}}\midG \aFA_S\right]=0.
				\end{align}
		\end{enumerate}
		\end{Lem}

		Now we can put all pieces together to finally prove 
		Theorem \ref{sto:results_thm_2}:

		\textbf{Proof of Theorem \ref{sto:results_thm_2}:}
			We get by Lemma \ref{sto:Lem_12} for $\mathbb{Q}\ni 
			\ell\geq 0$ that
			\begin{align*}
			X_S&-\EW\left[X_{\tau_{S,\ell}}\midG \aFA_S\right]
						=\EW\left[\int_{[S,\tau_{S,\ell})} g_t(\bar{L}_{S,t})
					\mu(\md t)
						\midG \aFA_S\right]
		\end{align*}			
		and by \eqref{sto:eq_38} we can let 
		$\ell \uparrow \infty$ along the rationals to obtain		
			\begin{align}\label{sto:eq_520}
			X_S=\lim_{\mathbb{Q}\ni \ell\rightarrow \infty}\EW\left[
			\int_{[S,\tau_{S,\ell})} g_t(\bar{L}_{S,t})
					\mu(\md t)
						\midG \aFA_S\right].
		\end{align}		
		Next we get by Lemma \ref{Lem:sp3} (vi) that $[S,\tau_{S,\ell})$
		is increasing. 
		Further let $\tilde{\Omega}:=\bar{\Omega}\backslash \mathcal{N}$
		with $\mathcal{N}$ the $\WM$-null set from Lemma \ref{Lem:s5} 
		and $\bar{\Omega}\subset\Omega$ with $\WM(\bar{\Omega})=1$ such that on $\bar{\Omega}$
		relation \eqref{sto:eq_12} holds for $\ell=0$. Then the following claim 
		holds:
		
		\textbf{Claim:} For $\omega\in \tilde{\Omega}$, $\ell\geq 0$
		and $t\in [S,\tau_{S,\ell})(\omega)\backslash 
		[S,\tau_{S,0})(\omega)$, we have
		$\bar{L}_{S,t}(\omega)\geq 0$.
		
		Indeed, for $t>T_{S,0}(\omega)$ we get immediately from $A\subset B$ in Lemma \ref{Lem:s5} 
		(\eqref{sto:set_a} and \eqref{sto:set_b}) that 
		$\bar{L}_{S,t}(\omega)\geq 0$. Hence we can focus on
		$t=T_{S,0}(\omega)$. By assumption also $t\in [S,\tau_{S,\ell})(\omega)
				\backslash [S,\tau_{S,0})(\omega)$
		and so 
		$\omega \in H_{S,0}^-\cup H_{S,0}$. Hence
		we get from $\tilde{A}\subset \tilde{B}$ in Lemma \ref{Lem:s5} 
		(\eqref{sto:set_ta} and \eqref{sto:set_tb}) with $\omega\in \bar{\Omega}$
		 again that
		 $\bar{L}_{S,t}(\omega)\geq 0$, which we wanted to show.
		
		As by Lemma \ref{sto:Lem_13} we see that $[S,\tau_{S,\ell})\nearrow
		[S,\infty)$ up to a $\WM\otimes\mu$-null set. Hence, we can use by the above claim and normalization \eqref{sto:sep_gl} monotone convergence 
		to obtain from \eqref{sto:eq_520} that
		\begin{align*}
			X_S&= \EW\left[\int_{[S,\tau_{S,0})} g_t(\bar{L}_{S,t})
							\mu(\md t)
						\midG \aFA_S\right]
						+\lim_{\mathbb{Q}\ni \ell\rightarrow \infty}
					\EW\left[\int_{[S,\tau_{S,\ell})\backslash [S,\tau_{S,0})} g_t(\bar{L}_{S,t})
			\
					\mu(\md t)
						\midG \aFA_S\right]\\
						&=\EW\left[\int_{[S,\infty)} g_t\left(\sup_{v\in [S,t]}
						L_v\right) \mu(\md t)\midG \aFA_S\right].
		\end{align*}
		From this result we infer that in fact the a priori generalized conditional expectation on
		the right-hand side has finite mean by integrability of $X_S$.
		This additionally yields
		$g_t(\sup_{v\in [S,t]} L_v)\mathbb{1}_{[S,\infty)}(t)
		\in \mathrm{L}^1(\WM\otimes \mu(\md t))$. 
		
		Hence $L$ satisfies 
		\eqref{sto:frame_gl}, \eqref{sto:results_gl} and 
		\eqref{sto:frame_gl_2}. By Corollary \ref{Main:2}
		this solution is also maximal and uniqueness of such a maximal solution follows
		by a corollary to the Meyer Section Theorem (see Corollary
		 \ref{app:meyer_cor_1}), which completes the proof of Theorem \ref{sto:results_thm_2}.
		\hfill $\blacksquare$


\begin{appendix}

    \section{Proofs for the results in Section \ref{sto:sec_upper}}\label{app:A}
	
	        In the proof of Lemma \ref{Lem:um1} we need the following result:
	        \begin{Lem}\label{technical_result_1}			
    			Let $S\in \stmr$, $T\in \stmsg{S}$ and $\ell>\ell'$. Then we have
    			\[
    				\EW\left[\int_{[S,T)}g_t(\ell)-g_t(\ell')\mu(\md t)\midG \aFA_S\right]>0
    			\]
    			on $\left\{\WM\left(\mu([S,T))>0\midG \aFA_S\right)>0\right\}$ almost surely.
    		\end{Lem}

    		\begin{proof}
    			Without loss of generality we have $\WM(\Gamma)>0$ with
    			\[
    				\Gamma:=\left\{\WM\left(\mu([S,T))>0\midG \aFA_S\right)>0\right\}\in \aFA_S.
    			\]
    			By strict monotonicity of $g$
    			in $\ell$ we know that $g_t(\omega,\ell)-g_t(\omega,\ell')$ is strictly positive
    			for all $\omega\in \Omega$ and all $t\in [0,\infty)$. Furthermore we have by 
    			Proposition \ref{sto:Pro_46} below 
    			that we have $0<\WM\left(\mu([S,T))>0\midG \aFA_S\right)=0$
    			on $\Gamma\cap\left\{\EW\left[\mu([S,T))\midG \aFA_S\right]=0\right\}$
    			and therefore we have up to a $\WM$-null set
    			\[
    				\Gamma \subset \left\{\EW\left[\mu([S,T))\midG \aFA_S\right]>0\right\}
    				\subset \left\{\EW\left[\int_{[S,T)}g_t(\ell)-g_t(\ell')\mu(\md t)\midG \aFA_S\right]>0\right\}.
    			\]
    		\end{proof}	
		
	    \begin{proof}[Proof of Lemma \ref{Lem:um1}]
	        
	        The uniqueness claim of Lemma \ref{Lem:um1} now follows immediately from strict monotonicity of $g=g_t(\omega,\ell)$	in $\ell \in \RZ$. For existence set
			\[
				\Gamma:=\EW\left[X_S-X_T\midG \aFA_S\right]
			\]
			and 
			\[
				G(\ell):=\EW\left[\int_{[S,T)}g_t(\ell)\mu(\md t)\midG \aFA_S\right],
				\quad \ell\in \RZ.
			\]
			As $\{\WM(\mu([S,T))>0|\aFA_S)>0\}\in \aFA_S$, we can define $\ell_{S,T}$ separately on this set and its complement. Therefore we focus in the following on $\{\WM(\mu([S,T))>0|\aFA_S)>0\}$. By Lemma \ref{technical_result_1} we have $G(\ell)>G(\ell')$ on $\{\WM(\mu([S,T))>0|\aFA_S)>0\}$ for $\ell,\ell'\in \RZ$ with $\ell>\ell'$ and hence it is enough to construct $\ell_{S,T}$ on $E:=\{\WM(\mu([S,T))>0|\aFA_S)>0\}\cap \{G(z)\leq \Gamma < G(z+1)\}$ for any $z\in \mathbb{Z}$. On $E$ we set
			\[
				\ell_{S,T}:=\lim_{n\rightarrow \infty} \sum_{k\in \mathbb{Z}} \frac{k}{n}\mathbb{1}_{\left\{G\left(\frac{k-1}{n}\right)\leq
						\Gamma < G\left(\frac{k}{n}\right)\right\}}
			\]
			and this will satisfy the required properties by continuity of $\ell\mapsto g(\ell)$, the $\WM\otimes \mu$-integrability of $t\mapsto g_t(\ell)$ and dominated convergence.
			
		\subsection{Proof of Proposition \ref{Main:7}}
		
		    The proof of this result is the first part of the proof of \citing{BK04}{Theorem 1}
		    adapted to our context. So assume we are given a $\Lambda$-measurable process
			$\tilde{L}$ which satisfy mutatis mutandis
			 \eqref{sto:frame_gl} and \eqref{sto:frame_gl_2}.
			Fix a stopping time $S\in  \stmr$.
			Consider $T\in \stmsg{S}$ and
			use the representation property and the integrability property of $\tilde{L}$
			to write
			\begin{align*}
				X_S= &\EW\left[\int_{[S,T)} g_t\left(\sup_{v\in [S,t]} \tilde{L}_v
						\right)  \mu(\md t) \midG \aFA_{S}\right]+\EW\left[\int_{[T,\infty)} g_t\left(\sup_{v\in [S,t]} \tilde{L}_v
							\right) \mu(\md t) \midG \aFA_{S}\right].
			\end{align*}
			As $\ell \mapsto g_t(\ell)$ is non-decreasing by Assumption \ref{sto:frame_ass}, we may estimate the first integrand from below
			 by $g_t(\tilde{L}_S)$ and
			the second integrand by $g_t(\sup_{v\in [T,t]} \tilde{L}_v)$ to obtain
			\begin{align*}
				X_S&\geq \EW\left[\int_{[S,T)} g_t\left(\tilde{L}_S\right) \mu(\md t) 
				\midG \aFA_{S}\right]+\EW\left[\int_{[T,\infty)} g_t\left(\sup_{v\in [T,t]} \tilde{L}_v
							\right) \mu(\md t) \midG \aFA_{S}\right].
			\end{align*}
			From the representation property of $\tilde{L}$ at time $T$, it follows that we may rewrite the second of the above
			summands as
			\[
				\EW\left[\int_{[T,\infty)} g_t\left(\sup_{v\in [T,t]} \tilde{L}_v\right) \mu(\md t) 
				\midG \aFA_S\right]=\EW\left[X_T\midG\aFA_S\right]
			\]
			and, therefore, we get the estimate
			\begin{align}\label{eq_1}
				\EW\left[X_S-X_T\midG\aFA_{S}\right]\geq \EW\left[\int_{[S,T)} 
				g_t\left(\tilde{L}_S\right) \mu(\md t) \midG \aFA_{S}\right].
			\end{align}
			As $\tilde{L}_S$ is $\aFA_{S}$-measurable, this shows $\tilde{L}_S\leq \ell_{S,T}$ 
			almost surely.
			Since in the above estimate $T\in \stmsg{S}$ was
			arbitrary, we deduce
			\[
				\tilde{L}_S\leq \essinf_{T\in \stmsg{S}} \ell_{S,T}.
			\]
		\end{proof}
	

    \section{Proofs for the results in Section \ref{sto:sec_ex}}
	    
	    \subsection{Preliminary Path regularity results}\label{sto:tools}
	
        	In this section we will state three results, two concerning the path properties of the process $X$ considered
        	 in Theorem \ref{sto:results_thm_2}
        	and one about the regularity of $\Lambda$-projections of random fields
        	(see Definition and Theorem \ref{app:meyer_thm_3}). These results will be needed in several arguments in the upcoming proofs.

        	First we adapt \citing{BK04}{Lemma 4.11}{1050}, now in the context of $\Lambda$-measurable processes. The proof and
        	the changed statements are mainly based on \citing{BS77}{Theorem II.1}{305}, and \citing{DL82}{Lemma 6}{303}.
        	
        	\begin{Lem}\label{sto:tools_lem}
        			Any $\Lambda$-measurable process $X$ of class($\text{D}^\Lambda$)
        			which is left-upper-semicontinuous in expectation at every $S\in \stp$
        			 with $X_{\infty}=0$ has the following properties:
        			\begin{compactenum}[(i)]
        				\item $X$ is pathwise bounded from above and below by a positive
        				 $\Lambda$-martingale of class($D^\Lambda$), i.e there is a positive $\Lambda$-martingale 
        				 $M^X:\Omega\times [0,\infty]\rightarrow [0,\infty)$  (see 
        				 \citing{BB18_2}{Definition 3.3}{8})
        							such that $-M^X_t(\omega) \leq X_t(\omega)\leq M^X_t(\omega)$ for $(\omega,t)\in \Omega\times 
        							[0,\infty]$.
        				\item We have up to an evanescent set in $\Omega\times
        				[0,\infty)$ that $^\mathcal{P} X\geq \lsl{X}$ and $^\mathcal{P} X_\infty\geq \lsl{X}_\infty$.
        			\end{compactenum}
        		\end{Lem}
        		
        	\begin{proof}
        		Part (i) follows as the proof of \citing{BK04}{Lemma 4.11}{1050} with the help of
        		\citing{BB18_2}{Theorem 3.7}{10} and \citing{BB18_2}{Proposition 3.9}{11}.
        		Part (ii) follows by applying \citing{BB18_2}{Lemma 4.4}{21},
        		the left-upper-semicontinuity in expectation of $X$ at every $S\in \stp$ and $\lsl{X}_0=X_0$.
        	\end{proof} 
        	
        	The next part gives us a useful consequence of $\mu$-right-upper-semicontinuity in expectation:
	
        	\begin{Pro}\label{eq_rusc}
        		Assume we have a process $X$ of class($D^\Lambda$) with $X_\infty=0$ which is 
        		$\mu$-right-upper-semicontinuous in expectation in all $S\in \stm$.
        		
        		Then we have for any $S\in \stm$ and any sequence $(S_n)_{n\in \NZ}\subset
        				\stmg{S}$ such that $\mu([S,S_n))$ vanishes almost surely and such that $\lim_{n\rightarrow \infty}
        				\EW[X_{S_n}|\aFA_S]$ exists that almost surely
        					\[
        						X_S
        						\geq \lim_{n\rightarrow \infty}
        						\EW\left[X_{S_n}|\aFA_S\right].  
        					\]				
        	\end{Pro}	
        	
        	\begin{proof}
        		Fix $S\in \stm$ and a sequence $(S_n)_{n\in \NZ}\subset \stmg{S}$ with
        		$\mu([S,S_n))$ converging to zero almost surely	and
        		such that $\lim_{n\rightarrow \infty}
        				\EW[X_{S_n}|\aFA_S]$ exists. Now define
        		\begin{align*}
        			\Gamma_S:=\left\{\lim_{n\rightarrow \infty}
        					\EW\left[X_{S_n}|\aFA_S\right]>X_S\right\},
        		\end{align*}
        		which is a subset of $\{S<\infty\}$ by $X_\infty=0$ and $S_n=\infty$ on $\{S=\infty\}$.
        		Assume by way of contradiction that $\WM(\Gamma_S)>0$. Since $\Gamma_S\in \aFA_S$, 
        		$\tilde{S}:=S_{\Gamma_S}$ and
        		$\tilde{S}_n:=(S_n)_{\Gamma_S}$ are in $\stm$. Hence,
        		we obtain by $\mu$-right-upper-semicontinuous in expectation and the class($D^\Lambda$) property of $X$
        		that
        		\begin{align*}
        			\EW\left[X_{\tilde{S}}\right]=\EW\left[X_{S}\mathbb{1}_{\Gamma_S}\right]
        			&<\EW\left[\lim_{n\rightarrow \infty}
        					\EW\left[X_{S_n}|\aFA_S\right]\mathbb{1}_{\Gamma_S}\right]\\
        			&= \lim_{n\rightarrow \infty} \EW\left[
        					X_{S_n}\mathbb{1}_{\Gamma_S}\right]
        			= \lim_{n\rightarrow \infty} \EW\left[
        					X_{\tilde{S}_n}\right]\leq \EW\left[X_{\tilde{S}}\right],
        		\end{align*}
        		which is a contradiction and proves our result.
        	\end{proof}
	
	        	The next result uses \cite{KP17}. Specifically, we will use \citing{KP17}{Section 5, Corollary 2}{324}
        	 for $d=1$, but for general Meyer-$\sigma$-fields rather than just for the optional and predictable-$\sigma$-field.
        	This is possible as in the proof of their results the authors of \cite{KP17} are not
        	 using special properties of the optional or predictable $\sigma$-field, but merely the
        	Optional and the Predictable Section Theorem and existence of respective projections.
        	 Hence with an application of the Meyer Section Theorem (see Theorem \ref{Main:4})
        	  and the definition of $\Lambda$-projections (see Definition and Theorem \ref{app:meyer_thm_3})
        	  one can use their results mutatis mutandis in our setting. Moreover to extend the
        	  pointwise convergence obtained by \cite{KP17} to uniform convergence we are using
        	  the idea of \citing{BK14}{Proof of Lemma C.1}{56-58} and the results on optional strong supermartingales  of \cite{DM82}, Appendix 1, properly adapted to $\Lambda$-measurable processes.

        	\begin{Lem}\label{sto:tools_lem_2}
        		The $\Lambda$-projections of
        		\[
        			h^\ell:=\int_{[0,\infty)} g_s(\ell) \mu(\md s), \quad \ell\in \RZ,
        		\]
        		can be chosen such that for all $\omega \in \Omega$ we have
        		\[
        			\lim_{\delta \downarrow 0}
        			\sup_{\overset{\ell,\ell'\in C}{|\ell'-\ell|
        			\leq \delta}}\sup_{t\in [0,\infty]}
        			 \left|{^\Lambda}{}{\mathop{(h^{\ell'})}}_t(\omega)-{^\Lambda}{}{\mathop{(h^{\ell})}}_t(\omega)\right|
        			=0
        		\]
        		for any compact set $C\subset \RZ$.
        	\end{Lem}	
	    
	    	\begin{proof}
	    	    By \cite{KP17} the $\Lambda$-projections of $h^\ell$, $\ell \in \RZ$,
    		can be chosen such that
    		\begin{align}\label{continuity}
    			\lim_{\ell'\rightarrow \ell} {^\Lambda}{}{\mathop{(h^{\ell'})}}_t(\omega)
    			={^\Lambda}{}{\mathop{(h^{\ell})}}_t(\omega)\quad \text{ for all } \ell\in \RZ,
    			\ (\omega,t)\in \Omega\times [0,\infty]. 
    		\end{align}
    		Fix in the following a compact set $C\subset \RZ$ and
    		$\delta>0$. Then we define
    		\[
    			h(\delta):=\sup_{\overset{\ell,\ell'\in C}{|\ell'-\ell|
    			\leq \delta}} \int_{[0,\infty)}
    						|g_s(\ell')-g_s(\ell)| \mu(\md s)\geq 0,
    		\]
    		which converges to zero a.s. and in $\mathrm{L}^1(\WM)$ by Assumption \ref{sto:frame_ass}.
    		Furthermore for fixed $\ell,\ell'\in C$ with $|\ell-\ell'|\leq \delta$ 
    		and any $T\in \stm$ we have 
    		\begin{align*}
    			\left|{^\Lambda}{}{\mathop{(h^{\ell'})}}_T
    							-{^\Lambda}{}{\mathop{(h^{\ell})}}_T\right|
    				\leq \EW\left[|h^{\ell'}-h^\ell|\middle| \aFA_T\right]
    				\leq \EW\left[h(\delta)\middle| \aFA_T\right]
    				={^\Lambda}{}{\mathop{h(\delta)}}_T\,\text{ a.s. on } \{T<\infty\}.
    		\end{align*}
    		Hence, by the Meyer Section Theorem, 
    		\[
    			|{^\Lambda}{}{\mathop{(h^{\ell'})}}
    							-{^\Lambda}{}{\mathop{(h^{\ell})}}|
    							\leq {^\Lambda}{}{h(\delta)}
    		\]
    		up to an evanescent set for any fixed $\ell,\ell'\in C$ with $|\ell-\ell'|\leq \delta$.
    		Therefore, almost surely, for any rational $\ell,\ell'\in C$ with $|\ell-\ell'|\leq \delta$
    		\[
    			\sup_{t\in [0,\infty]} |{^\Lambda}{}{\mathop{(h^{\ell'})}}_t
    							-{^\Lambda}{}{\mathop{(h^{\ell})}}_t|
    				\leq \sup_{t\in [0,\infty]} {^\Lambda}{}{h(\delta)}_t,
    		\]
    		which implies almost surely
    		\[
    			\sup_{\overset{\ell,\ell'\in C\cap \mathbb{Q}}{|\ell'-\ell|
    		\leq \delta}} \sup_{t\in [0,\infty]} |{^\Lambda}{}{\mathop{(h^{\ell'})}}_t
    							-{^\Lambda}{}{\mathop{(h^{\ell})}}_t|
    				\leq \sup_{t\in [0,\infty]} {^\Lambda}{}{h(\delta)}_t.
    		\]
    		Interchanging its two suprema, the left-hand side can be rewritten as 
    		\begin{align*}
    			\sup_{\overset{\ell,\ell'\in C\cap \mathbb{Q}}{|\ell'-\ell|
    		\leq \delta}} \sup_{t\in [0,\infty]} |{^\Lambda}{}{\mathop{(h^{\ell'})}}_t
    							-{^\Lambda}{}{\mathop{(h^{\ell})}}_t|
    				&=\sup_{t\in [0,\infty]} \sup_{\overset{\ell,\ell'\in C\cap \mathbb{Q}}{|\ell'-\ell|
    		\leq \delta}}  |{^\Lambda}{}{\mathop{(h^{\ell'})}}_t
    							-{^\Lambda}{}{\mathop{(h^{\ell})}}_t|\\
    				&= \sup_{t\in [0,\infty]} \sup_{\overset{\ell,\ell'\in C}{|\ell'-\ell|
    		\leq \delta}} |{^\Lambda}{}{\mathop{(h^{\ell'})}}_t
    							-{^\Lambda}{}{\mathop{(h^{\ell})}}_t|,
    		\end{align*}
    		where we used the pointwise continuity \eqref{continuity} in the last equality.
    		Therefore we have almost surely
    		\[
    		 \sup_{\overset{\ell,\ell'\in C}{|\ell'-\ell|
    		\leq \delta}} 	\sup_{t\in [0,\infty]}|{^\Lambda}{}{\mathop{(h^{\ell'})}}_t
    							-{^\Lambda}{}{\mathop{(h^{\ell})}}_t|
    					\leq \sup_{t\in [0,\infty]} {^\Lambda}{}{h(\delta)}_t,
    		\]
    		and now it suffices to argue that ${^\Lambda}{}{h(\delta)}_t\rightarrow 0$ almost surely uniformly in $t\in [0,\infty]$.
    		For this, we will use Doob's maximal martingale inequality, suitably generalized for $\Lambda$-martingales
    		(\citing{DM82}{Appendix 1, (3.1)}{394}). More precisely, for any $\lambda \in [0,\infty)$ we have by dominated convergence that
    		\[
    			\lambda \WM\left(\sup_{t\in [0,\infty]} |{^\Lambda}{}{h(\delta)}_t|> \lambda\right)\leq \EW\left[|{^\Lambda h}(\delta)_\infty|\right]
    			=\EW\left[h(\delta)\right]\overset{\delta\downarrow 0}{\longrightarrow} 0,
    		\]
    		which finishes our proof.
    	\end{proof}
    	    
    	    
    	    \subsection{Proof of Lemma \ref{Lem:sp3}}\label{proof:4.1}
    	    
            We start by constructing in Proposition \ref{sto:useful_pro} below processes 
    		$\tilde{Y}^\ell$, $\ell\in \RZ$, which will fulfill the conditions (i)-(vii) of 
    		Lemma \ref{Lem:sp3} for fixed $\ell$. The random field $Y$ will then be constructed 
    		as a limit of the processes $\tilde{Y}^\ell$.
    		The idea to construct a process $\tilde{Y}^\ell$ for fixed $\ell$
    		is to use the optimal stopping
    		results of \cite{EK81}. Specifically, we will construct $\tilde{Y}^\ell$
    		as a Snell-envelope and the properties will follow
    		with the help of divided stopping times (see Definition \ref{app:optstop_def_4}).
	    
	    \subsubsection{First step for the Proof 
		of Lemma \ref{Lem:sp3}: Construction of the process 
		\texorpdfstring{$Y$}{Y} for fixed \texorpdfstring{$\ell$}{l}}
		
	 \begin{Pro}\label{sto:useful_pro}
	 	\begin{compactenum}[(i)]
	 		\item For each $\ell\in \RZ$, there is a $\Lambda$-measurable l\`adl\`ag process
			 	$\tilde{Y}^\ell:\Omega\times [0,\infty]\rightarrow \RZ$ of class($D^\Lambda$)
			 	with $\tilde{Y}_{\infty}^\ell=0$, unique up
			 	 to indistinguishability, such that for all $S\in \stm$ we have
			 	 almost surely
				\begin{align}\label{sto:eq_6}
					\tilde{Y}^\ell_S=\esssup_{T\in \stmg{S}}\EW\left[X_T+\int_{[S,T)}g_t(\ell)\mu(\md t)\midG \aFA_S\right].
				\end{align}
				Moreover, for $\ell\leq \ell'$, we have 
				\[
					\WM\left(\tilde{Y}^{\ell}_t\leq \tilde{Y}^{\ell'}_t \text{ for all } t\geq 0\right)=1.
				\]
			\item  Define the l\`adl\`ag $\Lambda$-measurable processes $E^\ell$ of class($D^\Lambda$), $\ell\in \RZ$, 
			 by
				\[
					E^\ell:=\int_{[0,\cdot)} g_t(\ell)\mu(\md t)
							+\sideset{^\Lambda}{}{\mathop{\left(\int_{[0,\infty)}
							 |g_t(\ell)|\mu(\md t)\right)}}+M^X+1
				\]
				and
				\[
					E^\ell_{\infty}:=\int_{[0,\infty)} g_t(\ell)\mu(\md t)
							+\int_{[0,\infty)}
							 |g_t(\ell)|\mu(\md t)+M^X_{\infty}+1
				\]
			with $M^X$ as in Lemma \ref{sto:tools_lem}. Then there is a version of the stochastic field $(E^\ell)_{\ell\in \RZ}$
			such that
			the following results hold for all $\omega \in \Omega$:
			\begin{enumerate}
				\item[(1)] Uniform continuity in $\ell \in \RZ$, i.e.  
						   \begin{align}\label{sto_conv_ass}
								\lim_{\delta \downarrow 0}\sup_{\overset{\ell,\ell'\in C}{|\ell'-\ell|
\leq \delta}}\sup_{t\in [0,\infty]}
								 |E^{\ell'}_t(\omega)-E_t^\ell(\omega)|=0
			\end{align}
			for all compact sets $C\subset \RZ$.
			\item[(b)] Monotonicity in $\ell\in \RZ$, i.e. 
			for all $\ell,\ell'\in \RZ$ with
			$\ell\leq \ell'$ and all
						$t\in [0,\infty]$,
						$E^{\ell}_t(\omega)\leq E^{\ell'}_t(\omega)$.
			\item[(c)] L\`adl\`ag paths for $\ell\in \RZ$, i.e. for any
			$\ell\in \RZ$ the mapping $t\mapsto 
					E_t^\ell(\omega)$ is real valued and l\`adl\`ag. In
					particular the paths $t\mapsto E^\ell_{t+}(\omega)$ 
					and $t\mapsto E^\ell_{t-}(\omega)$ are bounded on
					compact intervals.			
			\item[(d)]	We have
					\begin{align*}
						\tilde{Y}^\ell_t(\omega)+E_t^\ell(\omega)\geq 1 \quad \text{ for all } \ell\in \RZ \text{ and all }
						 t\in [0,\infty].
					\end{align*}
			\end{enumerate}
			\item For $\ell\in \RZ$ and $S\in \stm$ the family of $\aFA_+$-stopping times
			\[
				\tilde{T}_{S,\ell}^\lambda(\omega)
				:=\inf\left\{t\in [S(\omega),\infty] \midG X_t(\omega)\geq 
				\lambda \tilde{Y}^\ell_t(\omega) -(1-\lambda) E_t^\ell(\omega)\right\},\ \lambda\in [0,1),
			\]
			is non-decreasing in $\lambda$ for all $\omega\in \Omega$ with limit 
			\begin{align*}
				\lim_{\lambda
			\uparrow 1} \tilde{T}_{S,\ell}^\lambda=:\tilde{T}_{S,\ell}=\min 
					\left\{t\in [S,\infty] \midG 				\tilde{Y}_t^{\ell}=X_t
									 \text{ or } \tilde{Y}_{t-}^{\ell}=\lsl{X}_t
									\text{ or } \tilde{Y}_{t+}^{\ell}=\lsr{X}{t}
									\right\}.
			\end{align*}
			\item The following inclusions hold for any $S\in \stm$, $\ell\in \RZ$:
			\begin{align}
				 \tilde{H}_{S,\ell}^-&:=\left\{
								 {\tilde{T}_{S,\ell}^\lambda <
								\tilde{T}_{S,\ell}}
								\text{ for all } \lambda\in [0,1)\right\}\subset
					\left\{	\tilde{Y}^\ell_{\tilde{T}_{S,\ell}-}
				=\lsl{X}_{\tilde{T}_{S,\ell}}\right\},\label{sto:eq_10}\\
				 \tilde{H}_{S,\ell}&:=\left\{\tilde{Y}^\ell_{\tilde{T}_{S,\ell}}
							=X_{\tilde{T}_{S,\ell}}\right\}\cap
				\left({ \tilde{H}_{S,\ell}^-}\right)^c\subset
				\left\{\tilde{Y}^\ell_{\tilde{T}_{S,\ell}}
				=X_{\tilde{T}_{S,\ell}}\right\},\label{sto:eq_11}\\
				 \tilde{H}_{S,\ell}^+&:=\left\{  \tilde{Y}^\ell_{\tilde{T}_{S,\ell} }>X_{\tilde{T}_{S,\ell} } \right\}
				 \cap \left({ \tilde{H}_{S,\ell}^-}\right)^c
				\subset
				\left\{\tilde{Y}^\ell_{\tilde{T}_{S,\ell}+}
				=\lsr{X}{\tilde{T}_{S,\ell}}\right\}.\label{sto:eq_13}
			\end{align}
			The sets $\tilde{H}_{S,\ell}$ and $\tilde{H}_{S,\ell}^+$ are contained in $\mathcal{F}_{\tilde{T}_{S,\ell}}^\Lambda$ and
			$\tilde{H}_{S,\ell}^-\in \mathcal{F}_{\tilde{T}_{S,\ell}-}^\Lambda$. Moreover, we have up to a $\WM$-null set
			\begin{align*}
				{ \tilde{H}_{S,\ell}^-\cup \tilde{H}_{S,\ell}}=
				\left\{
				\tilde{Y}^\ell_{\tilde{T}_{S,\ell}}
				=X_{\tilde{T}_{S,\ell}}\right\},\quad
				\left\{  \tilde{Y}^\ell_{\tilde{T}_{S,\ell} }>X_{\tilde{T}_{S,\ell} } \right\}
				\subset\left({ \tilde{H}_{S,\ell}^-}\right)^c
			\end{align*} 
			and $(\tilde{T}_{S,\ell})_{\tilde{H}_{S,\ell}^-}$ is an $\FA$-predictable stopping
			time, $(\tilde{T}_{S,\ell})_{\tilde{H}_{S,\ell}}$ is a $\Lambda$-stopping
			time and $(\tilde{T}_{S,\ell})_{H_{S,\ell}^+}$ is an $\aFA_+$-stopping time.
		\item For $S\in \stm$, $\ell\in \RZ$, the quadrupel
			\begin{align*}	
				\tilde{\tau}_{S,\ell}:=(\tilde{T}_{S,\ell}, \tilde{H}_{S,\ell}^-, \tilde{H}_{S,\ell}, \tilde{H}_{S,\ell}^+)
			\end{align*}					
			 is a divided 
			stopping time (see Definition \ref{app:optstop_def_4}).
			
		\item For $S\in \stm$, the mapping $\ell\mapsto \tilde{\tau}_{S,\ell}\in \stmd$
			is increasing in the sense that for all $\ell,\ell'\in \RZ$
			with $\ell\leq \ell'$ we have $[S,\tilde{\tau}_{S,\ell})
			\subset [S,\tilde{\tau}_{S,\ell'})$ with the definition of
			$[S,\tilde{\tau}_{S,\ell})$ given in \eqref{sto:eq_4}.
			Furthermore for $\ell\in \RZ$ the divided stopping time
			$\tilde{\tau}_{S,\ell}$ attains the value of the optimal stopping problem
			in \eqref{sto:eq_5}, i.e. almost surely
			\begin{align*}
				\tilde{Y}^\ell_S=\ &\EW\left[X_{\tilde{\tau}_{S,\ell}}
						+\int_{[S,\tilde{\tau}_{S,\ell})} g_t(\ell)\mu(\md t)
						\midG \aFA_S\right].
			\end{align*}
		\item For $S\in \stm$ and $\ell,\ell'\in \RZ$ we have almost surely
			\begin{align}\label{sto:useful_gl_3}
				\tilde{Y}^\ell_S\geq \EW\left[X_{\tilde{\tau}_{S,\ell'}}
									+\int_{[S,\tilde{\tau}_{S,\ell'})} g_t(\ell)\mu(\md t)\midG \aFA_S\right].
			\end{align}
			 Moreover, for fixed $\ell\leq \ell'$, we have
			\begin{align}\label{sto:useful_gl_4}
				\tilde{Y}^{\ell'}_s\geq \tilde{Y}^{\ell}_s \geq \tilde{Y}^{\ell'}_s
					+\ {^\Lambda \left(\int_{[0,\infty)}
						 g_t(\ell)\mu(\md t)\right)}_s
						 -\ {^\Lambda \left(\int_{[0,\infty)} 
						 g_t(\ell')\mu(\md t)\right)}_s
			\end{align}
			 for all $s\in [0,\infty]$ $\WM$\text{-almost surely}, where the latter $\Lambda$-projections
			 are chosen such that
			 for all $\omega\in \Omega$ we have:
			\begin{align*}
								\lim_{\delta \downarrow 0}\sup_{\overset{\ell,\ell'\in C}{|\ell'-\ell|\leq \delta}}\sup_{t\in [0,\infty]}
								 \left|{^\Lambda \left(\int_{[0,\infty)}
						 g_t(\ell)\mu(\md t)\right)}_t-{^\Lambda \left(\int_{[0,\infty)}
						 g_t(\ell')\mu(\md t)\right)}_t\right|=0
			\end{align*}
			for all compact sets $C\subset \RZ$.
			\end{compactenum}
		\end{Pro}	
		
		\begin{proof}
			\textbf{Part (i)-(iv):}
			Fix $\ell\in \RZ$. Lemma \ref{sto:tools_lem_2}, monotonicity and continuity of $\ell \mapsto g_t(\omega,\ell)$ for any 
			$(\omega,t)\in \Omega\times [0,\infty)$ ensure that the processes $E^\ell$, $\ell \in \RZ$, can be chosen such that \eqref{sto_conv_ass} holds. Furthermore each $E^\ell$ is l\`adl\`ag as $M^X$ and the $\Lambda$-projection
			are $\Lambda$-martingales and hence l\`adl\`ag by \citing{BB18_2}{Proposition 3.6}{10}. This also gives us the boundedness 
			on compact intervals of $E^{\ell}_+$ and $E^{\ell}_-$.
			The class($D^\Lambda$) property of $E^\ell$ follows because $M^X$ is of class($D^\Lambda$) and
			because $\EW[\int_{[0,\infty)} |g_t(\ell)|\mu(\md t)]<\infty$ by
			Assumption \ref{sto:frame_ass}. 
			
			Now consider the $\Lambda$-measurable process of class($D^\Lambda$)
			defined by $Z^\ell:=X+E^\ell\geq 1$.
			By Theorem \citing{BB18_2}{Theorem 3.7}{10}
			and \citing{BB18_2}{Proposition 3.9}{11} we can
			define the Snell envelope
			$\bar{Z}^\ell$ of $Z^\ell$, i.e. the $\Lambda$-supermartingale $\bar{Z}^\ell$ such that
			\begin{align*}
				\bar{Z}^{\ell}_S=\esssup_{T\in \stmg{S}}\EW\left[Z^\ell_T\midG \aFA_S\right],
				\quad S\in \stm,
			\end{align*}
			and the envelope $\bar{Z}^\ell$ is again of class($\text{D}^\Lambda$). Here we can assume
			that for any $(\omega,t)\in \Omega\times [0,\infty)$ we have $\bar{Z}_t^\ell(\omega)
			\geq Z_t(\omega)$.
			Now we have, for $S \in \stm$, that
			\begin{align*}
				\bar{Z}^\ell_S=\esssup_{T\in \stmg{S}}\EW\left[Z^\ell_T\midG \aFA_S\right]
				=&\esssup_{T\in \stmg{S}}\EW\left[
				X_T+\int_{[S,T)} g_t(\ell)\mu(\md t)\midG \aFA_S\right]
				+E^\ell_S,
			\end{align*}
			which shows that $\tilde{Y}^\ell:=\bar{Z}^\ell-E^\ell$ satisfies (i) by a corollary
			of the Meyer Section Theorem (see Corollary \ref{app:meyer_cor_1}). Here we also see
			that $\tilde{Y}^\ell$ and $E^\ell$ will satisfy the last part of (ii).
			Furthermore, we obtain part (iii) by applying \citing{BB18_2}{Proposition 3.11}{11} and \citing{BB18_2}{Proposition 3.13}{12} to $Z$
			and using afterwards $\tilde{Y}^\ell=\bar{Z}^\ell-E^\ell$.
			Finally part (iv) follows from \citing{BB18_2}{Proposition 3.13}{12}
			and \citing{BB18_2}{Proposition 4.6}{24}.
			
			\textbf{Proof of Part (v) and (vi):} That $\tilde{\tau}_{S,\ell}$
			is a divided stopping time
			follows by \citing{BB18_2}{Theorem 3.17}{14}. 
			To prove the monotonicity fix $\ell\leq \ell'$
			and $(\omega,t)\in \Omega \times [0,\infty)$. 
			If $\tilde{T}_{S,\ell}(\omega)<\tilde{T}_{S,\ell'}(\omega)$ 
			or $\tilde{T}_{S,\ell}(\omega)=\tilde{T}_{S,\ell'}(\omega)$
				and  $\omega \notin H_{S,\ell}^+$  there
			is nothing to show. Therefore
			let $t=\tilde{T}_{S,\ell}(\omega)=\tilde{T}_{S,\ell'}(\omega)$
			and assume $\omega \in H_{S,\ell}^+=\{X_{\tilde{T}_{S,\ell}}
			<\tilde{Y}_{\tilde{T}_{S,\ell}}^\ell\}\cap (\tilde{H}_{S,\ell}^-)^c$.
			We now have to show 
			$\omega \in \tilde{H}_{S,\ell'}^+$. By monotonicity of $\ell\mapsto Y_t^\ell(\omega)$ and $T_{S,\ell}(\omega)=t=\tilde{T}_{S,\ell'}(\omega)$ we get 
			\begin{align}\label{sto:eq_7}
				\omega \in \{X_{\tilde{T}_{S,\ell'}}<Y_{\tilde{T}_{S,\ell'}}^{\ell'}\}.
			\end{align}
			On the other hand, as $\omega \in (H_{S,\ell}^-)^c$,
			there exists $\lambda\in [0,1)$ such that $\tilde{T}_{S,\ell}^\lambda(\omega)=\tilde{T}_{S,\ell}(\omega)=t=\tilde{T}_{S,\ell'}(\omega)$. In particular,
			for all $s\in [S(\omega),\tilde{T}_{S,\ell}(\omega))$ we achieve with $Y^\ell_s(\omega)+E_s^\ell(\omega)\geq 1$ by the definition
			of $T_{S,\ell}^\lambda$
			\[
				X_s(\omega)<
				Y_s^\ell(\omega)-(1-\lambda)(E_s^\ell(\omega)+Y_s^\ell(\omega))\leq  Y_s^\ell(\omega)-(1-\lambda)
			\]
			and therefore, recalling  $\tilde{T}_{S,\ell}(\omega)=t=\tilde{T}_{S,\ell'}(\omega)$,
			\[
				\lsl{X}_{\tilde{T}_{S,\ell'}}(\omega)=\lsl{X}_{\tilde{T}_{S,\ell}}(\omega)<\tilde{Y}^{\ell}_{\tilde{T}_{S,\ell}-}(\omega)
				\leq \tilde{Y}^{\ell'}_{\tilde{T}_{S,\ell'}-}(\omega).
			\]
			Hence $\omega \in \left\{	\tilde{Y}^{\ell'}_{\tilde{T}_{S,\ell'}-}
					>\lsl{X}_{\tilde{T}_{S,\ell'}}\right\}
					\subset (\tilde{H}_{S,\ell}^-)^c$ by \eqref{sto:eq_10}.
					Together with 
					\eqref{sto:eq_7} we get
					$\omega\in H_{S,\ell'}^+$. This establishes 
					the claimed monotonicity of $\ell\mapsto \tau_{S,\ell}$.

			Next we have by \citing{BB18_2}{Proposition 3.13}{12}
			\begin{align*}
				\bar{Z}^\ell_S
				=\EW\left[Z^\ell_{\tilde{\tau}_{S,\ell}}\midG \aFA_S\right],
			\end{align*}
			which, $\tilde{Y}^l$ and $E^l$ being l\`adl\`ag, is equivalent to
			\begin{align*}
				\tilde{Y}^\ell_S
				=&\EW\left[X_{\tilde{\tau}_{S,\ell}}
							+\int_{[S,\tilde{\tau}_{S,\ell})} g_t(\ell)\mu(\md t) \midG \aFA_S\right]-
							 		G_S
					+\EW\left[G_{\tilde{\tau}_{S,\ell}}\midG \aFA_S\right],
			\end{align*}
			where $G$ is the $\Lambda$-martingale
			\[
				G:=\sideset{^\Lambda}{}{\mathop{\left(\int_{[0,\infty)} |g_t(\ell)|
				\mu(\md t)\right)}}+M^X, \quad
				G_{\infty}:=\int_{[0,\infty)} |g_t(\ell)|
				\mu(\md t)+M^X_{\infty}.
			\]
			This shows the final part of $(v)$ as $G_S=\EW\left[G_{\tilde{\tau}_{S,\ell}}\midG \aFA_S\right]$
			by \citing{BB18_2}{Lemma 3.16}{13}.
							
			\textbf{Proof of Part (vii):} 
			By \citing{BB18_2}{Lemma 3.9}{11} we can write
			$\bar{Z}^\ell=\bar{M}^\ell-\bar{A}^\ell-\bar{B}^\ell_-$, where $\bar{M}^\ell$ is a $\Lambda$-martingale
			and $\bar{A}^\ell$, $\bar{B}^\ell$ are positive $\Lambda$-measurable
			and non-decreasing. Hence, by \citing{BB18_2}{Lemma 3.16}{13},
			\begin{align}\label{sto:useful_gl_44}
			\bar{Z}^\ell_S=\ & \bar{M}^\ell_S-\bar{A}^\ell_{S}-\bar{B}^\ell_{S-}
			 =\ \EW\left[\bar{M}_{\tilde{\tau}_{S,\ell'}}^\ell	\midG \aFA_S\right]
										-\bar{A}^\ell_{S}-\bar{B}^\ell_{S-}\\
						\geq\ & 		
				\EW\left[{\bar{Z}}_{\tilde{\tau}_{S,\ell'}}	\midG \aFA_S\right]
				\geq\ 
				\EW\left[Z_{\tilde{\tau}_{S,\ell'}}^\ell		\midG \aFA_S\right],
			\end{align}
			where we have used that on $\tilde{H}_{S,\ell'}^-$ holds $\tilde{T}_{S,\ell'}>S$
			 and therefore also $\bar{B}_{\tilde{T}_{S,\ell'}-}^\ell\geq 
			\bar{B}_S^\ell$.						
			By \citing{BB18_2}{Lemma 3.16}{13}, the inequality (\ref{sto:useful_gl_44}) is equivalent 
			to (\ref{sto:useful_gl_3}).

			Next we show  (\ref{sto:useful_gl_4}),which follows
			as in the proof of \citing{BK04}{Lemma 4.12 (i)}{1050}: For $S\in \stm$, $\ell\leq \ell'$, we have
			\begin{align*}
				\tilde{Y}^{\ell'}_S\geq \tilde{Y}^\ell_S
					\geq \tilde{Y}^{\ell'}_S
						+ \EW\left[\int_{[0,\infty)}g_t(\ell)
					\mu(\md t)\midG \aFA_S\right]
					-\EW\left[\int_{[0,\infty)}g_t(\ell')
					\mu(\md t)\midG \aFA_S\right].
			\end{align*}
			As $S\in \stm$ is arbitrary, 
			$\Lambda$-measurability of $\tilde{Y}^\ell$ and $\tilde{Y}^{\ell'}$
			gives the pathwise estimate
			\begin{align*}
				\tilde{Y}^{\ell'}_s\geq \tilde{Y}^{\ell}_s\geq \tilde{Y}^{\ell'}_s
					+\ {^\Lambda \left(\int_{[0,\infty)}
						 g_t(\ell)\mu(\md t)\right)}_s
						 -\ {^\Lambda \left(\int_{[0,\infty)} 
						 g_t(\ell')\mu(\md t)\right)}_s, 
			\end{align*}
			 for all $s\in [0,\infty]$ $\WM$\text{-a.s.} 
				 by a corollary of the Meyer Section Theorem (see 
				Corollary \ref{app:meyer_cor_1}). This finishes the proof of (vii).
		\end{proof}
		
		
		\subsubsection{Proof of Lemma \ref{Lem:sp3}} 	With the help of Proposition \ref{sto:useful_pro}, we
		 can now construct the random field $Y$ with the stated properties.	
			
			\textbf{Construction of $Y$:}
				First we can choose for all $q\in \mathbb{Q}$ the processes $(\tilde{Y}^q)_{q\in \mathbb{Q}}$
				 from Proposition \ref{sto:useful_pro} such that (\ref{sto:useful_gl_4}),
			$\tilde{Y}^{q'}_t(\omega)\geq \tilde{Y}^{q}_t(\omega)\geq X_t(\omega)$
			 and $\tilde{Y}^q_t(\omega)+E_t^q(\omega)\geq 1$ holds true
			 simultaneously 
		for all $q,q'\in \mathbb{Q}$ with $q\geq q'$
			and any $(\omega,t)\in \Omega\times [0,\infty]$. 
			$E^\ell$ is the process from Proposition \ref{sto:useful_pro} (ii)
			 and satisfies therefore equation \eqref{sto_conv_ass_2}.				
			Now we define 
			\[
				Y^\ell_t(\omega):=\lim_{\mathbb{Q}\ni q\uparrow \ell} \tilde{Y}^q_t(\omega)
				=\sup_{\mathbb{Q}\ni q<\ell} \tilde{Y}_t^q(\omega).
			\]			
			As in the proof of \citing{BK04}{Lemma 4.12 (i)}{1050} we show now that $Y^\ell$
			and $\tilde{Y}^\ell$ are indistinguishable for all $\ell\in \RZ$.	By the Meyer Section Theorem it
			suffices to show $Y_S^\ell=\tilde{Y}^{\ell}$ for any $S\in \stm$.			
		Fix $S\in \stm$. Here ``$\leq$'' follows by $\tilde{Y}^q_S\leq \tilde{Y}^\ell_S$ for all
		rational $q<\ell$, which implies $Y_S^\ell\leq \tilde{Y}_S^\ell$. 
		For ``$\leq$'' we use that for
		any $T\in \stmg{S}$ we have
		\begin{align*}
			Y^\ell_S=\lim_{\mathbb{Q}\ni q\uparrow \ell} \tilde{Y}^q_S&\geq
					\limsup_{\mathbb{Q}\ni q\uparrow l} \EW\left[X_T+\int_{[S,T)}g_t(q)\mu(\md t)\midG \aFA_S\right]\\
					&= \EW\left[X_T+\int_{[S,T)}g_t(\ell)\mu(\md t)\midG \aFA_S\right].
		\end{align*}
		Since
		this estimate holds true for all $T\in \stmg{S}$, we may pass to the essential
		supremum on its right-hand side to obtain $Y^\ell_S\geq \tilde{Y}^\ell_S$ and therefore by the
		first part $Y^\ell_S=\tilde{Y}^\ell_S$.
		As $S\in \stm$ was arbitrary and 
		$\Lambda$-measurability  of both process, this entails with a corollary of the Meyer Section Theorem
		(see Corollary \ref{app:meyer_cor_1}) the claimed result.

		\textbf{Proof of Part (i), (ii) and the last part of (vi):}
			Result (i) and the property in (vi) that the divided stopping time
			attains the value of the optimal stopping problem 
			are stated for fixed $\ell \in \RZ$ and thus 
			follow directly as $\tilde{Y}^\ell$ is indistinguishable from $Y^\ell$.
			For (ii) we fix $\omega\in \Omega$, $t\in [0,\infty)$
			 and obtain a sequence $(q_n)_{n\in \NZ}\subset \mathbb{Q}$
			converging strictly from below to $\ell$ such that $\lim_{n\rightarrow \infty} \tilde{Y}^{q_n}_t(\omega)
			=Y^\ell_t(\omega)$. As $\tilde{Y}^{q_n}_t(\omega)+E^{q_n}_t(\omega)\geq 1$ 
			for any $n\in \NZ$ this leads by \eqref{sto_conv_ass_2} to 
			$Y^{\ell}_t(\omega)+E^{\ell}_t(\omega)\geq 1$. 
				
		\textbf{Proof of (vii) and monotonicity and continuity of $\ell\mapsto Y^\ell$:}	
			By taking again non-increasing rational limits in \eqref{sto:useful_gl_4}, we see that
			$Y^\ell$ will fulfill
			\eqref{sto:useful_gl_24} and 
			$Y^\ell_t(\omega)\geq X_t(\omega)$ for all $\ell,\ell'\in \mathbb{R}$
			and any $(\omega,t)\in \Omega\times [0,\infty]$. Hence by \eqref{sto:useful_gl_24} 
			and 	the convergence property of the $\Lambda$-projections			
			we obtain the continuity and monotonicity of the mapping $\ell\mapsto Y^\ell_t(\omega)$
			for any $(\omega,t)\in \Omega\times [0,\infty]$.
				
		\textbf{Proof of Part (iii), (iv), (v) and the rest of (vi):}
			First of all one can adapt the proof of \citing{EK81}{Proposition 2.35}{133}, to show the inclusions
			 \eqref{sto:eq_277}, \eqref{sto:eq_1277}, \eqref{sto:eq_77}. For \eqref{sto:useful_gl_8}, we have
			 that ``$\geq$'' follows from $H_{S,\ell}^-\cup H_{S,\ell}\cup H_{S,\ell}^+=\Omega$.
			Furthermore we get for $0\leq \lambda<1$ that
			\begin{align*}
				\inf \left\{t\in [S(\omega),\infty] \midG 
										Y_t^{\ell}(\omega)=X_t (\omega)
										 \text{ or } Y_{t-}^{\ell}(\omega)=\lsl{X}_t(\omega)
										\text{ or } Y_{t+}^{\ell}(\omega)=\lsr{X}{t}(\omega)
										\right\}\\
				\geq \inf 
				\left\{t\in [S(\omega),\infty] \midG X_t(\omega)\geq 
					\lambda Y^\ell_t(\omega) -(1-\lambda) E_t^\ell(\omega)\right\}=T_{S,l}^\lambda.
			\end{align*}
			This leads to ``$\leq$'' in \eqref{sto:useful_gl_8} by $T_{S,\ell}=\lim_{\lambda\uparrow 1} T_{S,\ell}^\lambda$.
			Finally as $H_{S,\ell}^-\cup H_{S,\ell}\cup H_{S,\ell}^+=\Omega$
			and (iv) we get for any $\omega \in \Omega$ that 
			$T_{S,l}(\omega)$ is contained in the set			
			\begin{align*}
				\left\{t\in [S(\omega),\infty] \midG 
										Y_t^{\ell}(\omega)=X_t (\omega)
										 \text{ or } Y_{t-}^{\ell}(\omega)=\lsl{X}_t(\omega)
										\text{ or } Y_{t+}^{\ell}(\omega)=\lsr{X}{t}(\omega)
										\right\}
			\end{align*}
			and hence the infimum is a minimum.
			
			 Next we have
			by construction of $Y$ that $Y^{\ell'}_s(\omega)\geq Y^\ell_s(\omega)\geq X_s(\omega)$
			for all 	$\ell\leq \ell'$, $s\in [0,\infty]$ and all $\omega\in \Omega$. Hence,
			for fixed $\omega$ and $\ell\leq \ell'$,
			\begin{align*}
				& \left\{t\geq S(\omega)\midG 				Y_t^{\ell'}(\omega)=X_t(\omega) 
											\text{ or } Y_{t-}^{\ell'}(\omega)=\lsl{X}_t(\omega)
											\text{ or } Y_{t+}^{\ell'}(\omega)=\lsr{X}{t}(\omega)  \right\}\\
				& \qquad \subset 				\left\{t\geq S(\omega)\midG 	Y_t^{\ell}(\omega) =X_t(\omega)
											\text{ or } Y_{t-}^{\ell}(\omega)=\lsl{X}_t(\omega)
											\text{ or } Y_{t+}^{\ell}(\omega)=\lsr{X}{t}(\omega)  \right\}.
			\end{align*}
			This shows that $T_{S,\ell'}(\omega)\geq T_{S,\ell}(\omega)$ for all $\omega\in \Omega$
			as we have shown that those stopping times are respectively the minimum
			over the left and right hand side of the previous inclusion. For the properties
			of $(T_{S,\ell})_{H_{S,\ell}^-}$, $(T_{S,\ell})_{H_{S,\ell}}$ and $(T_{S,\ell})_{H_{S,\ell}^+ }$
			one can again adapt the proof of \citing{EK81}{Proposition 2.35}{133}, and
			\eqref{sto:eq_12} follows as $Y^\ell$ is indistinguishable from $\tilde{Y}^\ell$. Having proven (i)-(iv)
			we can adopt the proof of (v) and (vi)
			in Proposition \ref{sto:useful_pro} to prove (v) and the rest of
			(vi) here.
				
		\textbf{Proof of Part (viii):} It remains to prove for $S\in \stm$ that $Y_S^{-\infty}=X_S$ almost surely.	
			By $X_S\leq Y_S^\ell$ for all $\ell\in \RZ$ and the monotonicity of $\ell \mapsto
			Y^\ell$, we obtain $X_S\leq Y_S^{-\infty}$. Hence it suffices to show
				$\EW[Y_S^{-\infty}]\leq \EW[X_S]$.
			We will follow at the beginning the argument of \citing{BK04}{Lemma 4.12}{1050}.
			By part (vi) we get
			\begin{align}\label{sto:useful_gl_5}
				X_S\leq Y^\ell_S= 	\EW\left[X_{\tau_{S,\ell}}
										+\int_{[S,\tau_{S,\ell})}
				g_t({\ell})\mu(\md t)		 \midG \aFA_S\right]
			\end{align}
			for any $\ell\in \RZ$. Furthermore we obtain from the monotonicity of $\ell \mapsto { T_{S,\ell}}$ that the limit
				$ T_{S,-\infty}:=	\lim_{\ell \downarrow -\infty} { T_{S,\ell}}$ exists
		pointwise. Now the following results hold:
		\begin{align}
			&\WM(\mu([S,T_{S,-\infty}))=0)=1,\label{sto:eq_111}\\
			&\WM\left(\{\mu(\{T_{S,-\infty}\})>0\}\cap \Gamma\right)=0, \label{sto:eq_113}\\
			&\WM\left(\{\mu(\{T_{S,-\infty}\})>0\}\cap \Gamma^c\cap
			\left\{X_{T_{S,-\infty}}<Y_{T_{S,-\infty}}^\ell \text{ for all $\ell\leq 0$}\right\}\right)=0.
				\label{sto:eq_114}
		\end{align}			
		
		\textbf{Proof of \eqref{sto:eq_111}:}
			By taking expectations on both sides in (\ref{sto:useful_gl_5})
			we get,
			\begin{align}\label{sto:eq_110}
				\EW[X_S]\leq \EW[M_{S}^X]+\EW\left[\int_{[S,T_{S,\ell})}
				g_t({\ell})\mu(\md t)\right]+\EW\left[g_{T_{S,\ell}}({\ell}) \mu(\{T_{S,\ell}\})
				\mathbb{1}_{H_{S,\ell}^+}\right],
			\end{align}
			where we have used \citing{BB18_2}{Lemma 3.16}{13} for 
			the $\Lambda$-martingale $M^X$ of Lemma \ref{sto:tools_lem}.
			Furthermore, for $\ell<0$, \eqref{sto:eq_110} leads to
			\begin{align*}
				\EW[X_S]\leq \EW[M_{S}^X]+\EW\left[\int_{[S,T_{S,\ell})}
				g_t({\ell})\mu(\md t)\right],
			\end{align*}			
			by $g_t(\ell)\leq 0$ (see (\ref{sto:sep_gl})) for any $t\in [0,\infty)$. 
			This implies that we get by Fatou's Lemma for any $\ell_0<0$ that
			\begin{align*}
				-\infty<\limsup_{\mathbb{Q}\ni \ell\downarrow  -\infty} \EW\left[\int_{[S,T_{S,\ell})}
				g_t({\ell})\mu(\md t)\right]
				&\leq \EW\left[\int_{[S,T_{S,-\infty})} g_t(\ell_0)
							\mu(\md t)\right]\\
			&\overset{\ell_0\downarrow -\infty}{\longrightarrow}
							\EW\left[(-\infty)
							\mathbb{1}_{\{\mu([S,T_{S,-\infty}))>0\}}
							\right].\nonumber
			\end{align*}
			Hence, 
			\begin{align*}
				\mu([S,T_{S,-\infty}))=0, \quad \text{$\WM$-almost surely}.			
			\end{align*}	
		
		\textbf{Proof of \eqref{sto:eq_113}, \eqref{sto:eq_114}:}			
		Define
		$\Gamma:=\bigcap_{n=1}^\infty  \left\{T_{S,{-n}}>T_{S,-\infty}\right\}$.
		Plugging \eqref{sto:eq_111} into \eqref{sto:eq_110} gives us
		again for $\ell_0<0$ and therefore $g_t(\ell_0)<0$ for all $t\in [0,\infty)$ with Fatou's Lemma
		\begin{align}\label{sto:eq_112}
			-\infty&< \limsup_{\mathbb{Q}\ni \ell \downarrow -\infty}
			 \EW\left[g_{T_{S,\ell}}({\ell}) \mu(\{T_{S,\ell}\})\nonumber
			\mathbb{1}_{H_{S,\ell}^+}\mathbb{1}_{\Gamma^c}\right.\\
			&\hspace{15ex}\left.+\left(\int_{[T_{S,-\infty},T_{S,\ell})}
			g_t({\ell})\mu(\md t)+g_{T_{S,\ell}}({\ell}) \mu(\{T_{S,\ell}\})
			\mathbb{1}_{H_{S,\ell}^+}\right)\mathbb{1}_{\Gamma}
			\right]\nonumber\\
			&\leq \limsup_{\mathbb{Q}\ni \ell\downarrow -\infty}
			 \EW\left[g_{T_{S,\ell}}({\ell_0}) \mu(\{T_{S,\ell}\})\nonumber
			\mathbb{1}_{H_{S,\ell}^+}\mathbb{1}_{\Gamma^c}
			+\left(\int_{[T_{S,-\infty},T_{S,\ell})}
			g_t({\ell_0})\mu(\md t)\right)\mathbb{1}_{\Gamma}
			\right]\\
			&\leq 
			 \EW\left[\mu(\{T_{S,-\infty}\})g_{T_{S,-\infty}}({\ell_0})\left(
			\liminf_{\mathbb{Q}\ni \ell\downarrow -\infty}\mathbb{1}_{H_{S,\ell}^+\cap \Gamma^c}
				+\mathbb{1}_{\Gamma}\right)\right]\nonumber\\
			&\overset{l_0\downarrow -\infty}{\longrightarrow}
			\EW\left[(-\infty)\mu(\{T_{S,-\infty}\})\left(
			\liminf_{\mathbb{Q}\ni \ell\downarrow -\infty}\mathbb{1}_{H_{S,\ell}^+\cap \Gamma^c}
			+\mathbb{1}_{\Gamma}\right)\right].
		\end{align}		
		One can see that the limes inferior inside of the latter
		expectation can actually be replaced by a limes and is a.s. given by
		\begin{align}\label{sto:eq_121}
			\lim_{\mathbb{Q}\ni \ell \downarrow -\infty}\mathbb{1}_{H_{S,\ell}^+\cap \Gamma^c}
			=\mathbb{1}_{\left\{X_{T_{S,-\infty}}<Y_{T_{S,-\infty}}^\ell \text{ for all $\ell\leq 0$}\right\}\cap \Gamma^c},
		\end{align}
		which follows from $H_{S,\ell}^+=\{X_{T_{S,\ell}}<Y_{T_{S,\ell}}^\ell\}$ up to a $\WM$-null set 
		(see \eqref{sto:eq_77}) 
		because by definition $\Gamma^c=\{T_{S,\ell}=T_{S,-\infty} \text{ for some $\ell<0$}\}$. 
		Passing to complements yields a.s.
		\begin{align}\label{sto:eq_211}
			\lim_{\mathbb{Q}\ni l\downarrow -\infty}\mathbb{1}_{(H_{S,\ell}^-\cup H_{S,\ell})\cap \Gamma^c}
			=\mathbb{1}_{\left\{X_{T_{S,-\infty}}=Y_{T_{S,-\infty}}^l \text{ for some $\ell\leq 0$}\right\}\cap \Gamma^c},
		\end{align}
		which we will use later.
		With \eqref{sto:eq_121} result \eqref{sto:eq_112} leads to \eqref{sto:eq_113} and \eqref{sto:eq_114}.

		With the help of \eqref{sto:eq_111}, \eqref{sto:eq_113} and \eqref{sto:eq_114}
		 we will now prove $X_S=Y_S^{-\infty}$:			
			
	\emph{\textbf{Establishing $X_S=Y_S^{-\infty}$:}}			
		First we have that $(T_{S,\ell})_{H_{S,\ell}^-}$ is a predictable stopping time (see 
		Lemma \ref{Lem:sp3} (iv)) and therefore we get by Lemma \ref{sto:tools_lem} (ii) that
		\begin{align}\label{sto:eq_119}
			\EW\left[\lsl{X}_{(T_{S,\ell})_{H_{S,\ell}^-}}\right]
			\leq \EW\left[{^\mathcal{P} X}_{(T_{S,\ell})_{H_{S,\ell}^-}}\right]
			=\EW\left[{X}_{T_{S,\ell}}\mathbb{1}_{H_{S,\ell}^-}\right].
		\end{align}
		Now we obtain by \eqref{sto:useful_gl_5}, \eqref{sto:eq_119}
		and $g_t(\ell)\leq 0$ for $\ell \leq 0$ (see (\ref{sto:sep_gl})) 
		for any $t\in [0,\infty)$ with the help of Fatou's Lemma
		that
		\begin{align}\label{sto:eq_125}
			\EW[X_S]&\overset{}{\leq} 
			\limsup_{\mathbb{Q}\ni \ell\downarrow -\infty}\EW[Y_S^\ell]\nonumber\\
			&\overset{\eqref{sto:useful_gl_5}}{=}	\limsup_{\mathbb{Q}\ni \ell\downarrow -\infty}	
			\EW\left[\lsl{X}_{T_{S,\ell}}  \mathbb{1}_{H_{S,\ell}^-}
									+X_{T_{S,\ell}}\mathbb{1}_{H_{S,\ell}}
									+\lsr{X}{T_{S,\ell}}\mathbb{1}_{H_{S,\ell}^+}
									+\int_{[S,\tau_{S,\ell})}
			g_t({\ell})\mu(\md t)	\right]\nonumber\\	
			&\overset{\eqref{sto:eq_119}}{\leq}	\limsup_{\mathbb{Q}\ni \ell\downarrow -\infty}	
			\EW\left[X_{T_{S,\ell}}\mathbb{1}_{H_{S,\ell}^-\cup H_{S,\ell}}
									+\lsr{X}{T_{S,\ell}}\mathbb{1}_{H_{S,\ell}^+}
									+\int_{[S,\tau_{S,\ell})}
			g_t({\ell})\mu(\md t)	\right]\nonumber\\	
			&\leq		\EW\left[	\limsup_{\mathbb{Q}\ni \ell\downarrow -\infty}\left(
			X_{T_{S,\ell}}\mathbb{1}_{H_{S,\ell}^-\cup H_{S,\ell}}
									+\lsr{X}{T_{S,\ell}}\mathbb{1}_{H_{S,\ell}^+}
									\right)	\right]\nonumber\\
				&\overset{\eqref{sto:eq_121},\eqref{sto:eq_211}}{=}
						\EW\left[	\limsup_{\mathbb{Q}\ni \ell\downarrow -\infty}\left(
			X_{T_{S,\ell}}\mathbb{1}_{H_{S,\ell}^-\cup H_{S,\ell}}
									+\lsr{X}{T_{S,\ell}}\mathbb{1}_{H_{S,\ell}^+}
									\right)\mathbb{1}_{\Gamma}	\right]\nonumber\\
			&\hspace{10ex}+\EW\left[X_{T_{S,-\infty}}
					\left(\lim_{\mathbb{Q}\ni \ell\downarrow -\infty}\mathbb{1}_{(H_{S,\ell}^-\cup H_{S,\ell})\cap \Gamma^c}\right)
					+\lsr{X}{T_{S,-\infty}}\left(\lim_{\mathbb{Q}\ni \ell\downarrow -\infty}
					\mathbb{1}_{H_{S,\ell}^+\cap \Gamma^c}\right)\right]\nonumber\\
			&\overset{}{\leq} \EW\left[\lsr{X}{T_{S,-\infty}}\mathbb{1}_{\Gamma}+
			 X_{T_{S,-\infty}}
					\left(\lim_{\mathbb{Q}\ni \ell\downarrow -\infty}\mathbb{1}_{(H_{S,\ell}^-\cup H_{S,\ell})\cap \Gamma^c}\right)
					+\lsr{X}{T_{S,-\infty}}\left(\lim_{\mathbb{Q}\ni \ell\downarrow -\infty}
					\mathbb{1}_{H_{S,\ell}^+\cap \Gamma^c}\right)\right],
		\end{align}	
		where in the last step we have used that $T_{S,\ell}$ is converging strictly from above
		 to $T_{S,-\infty}$ on $\Gamma$ and therefore $X_{T_{S,\ell}}$ and $\lsr{X}{T_{S,\ell}}$ converge both
		 to $\lsr{X}{T_{S,-\infty}}$.
					
        By \eqref{sto:eq_121}, \eqref{sto:eq_211} and  \eqref{sto:eq_125} it remains to show that
        \begin{align}\label{sto:eq_123}
        	\EW\left[\lsr{X}{T_{S,-\infty}}\mathbb{1}_{E}+
        	 X_{T_{S,-\infty}}
        			\mathbb{1}_{E^c}\right]\leq \EW[X_S],
        \end{align}
        where we use for ease of notation
        \[
        	E:=\Gamma\cup
        	\left(\left\{X_{T_{S,-\infty}}<Y_{T_{S,-\infty}}^\ell \text{ for all $\ell\leq 0$}\right\}
        	\cap \Gamma^c\right)
        \]
        with
        \[
        	E^c=\left\{X_{T_{S,-\infty}}=Y_{T_{S,-\infty}}^\ell
        	 \text{ for some $\ell\leq 0$}\right\}\cap \Gamma^c.
        \]

        For the next steps we will need the following claim, which we will prove below:
        
        \emph{\textbf{Claim:}} The $\FA$-stopping time $(T_{S,-\infty})_{E^c}$ 
        		is actually a $\Lambda$-stopping time.
        
        Now we get with the help of  \citing{BB18_2}{Proposition 4.2}{19}
        a	non-increasing sequence $\tilde{T}_n\in 
        \stmg{(T_{S,-\infty})_{E}}$ 
        with limit $(T_{S,-\infty})_{E}$ and $\tilde{T}_n>
        (T_{S,-\infty})_{E}$ on $\{(T_{S,-\infty})_{E}<\infty\}
        =E$ such that
        \[
        					\lsr{X}{(T_{S,-\infty})_{E}}=
        					\lim_{n\rightarrow \infty} X_{\tilde{T}_n}\quad 
        					\text{ $\WM$-almost surely.}
        \]
        By the claim we can define the sequence
        of $\Lambda$-stopping times
        $T_n:=\tilde{T}_n\wedge (T_{S,-\infty})_{E^c}$.
        Now we see that $\lim_{n\rightarrow \infty}\mu([S,T_n))=0$
        a.s.. Indeed, 
        by \eqref{sto:eq_111}, \eqref{sto:eq_113}
        and \eqref{sto:eq_114} we have 
        $\mu([S,T_{S,-\infty}))=0$ a.s. and
        $E\subset \{\mu(\{T_{S,-\infty}\})=0\}$. As $T_n$ converges strictly from above to $T_{S,-\infty}$ on 
        $E$ and is equal $T_{S,-\infty}$ on $E^c$ we get
        \[
        	\lim_{n\rightarrow \infty}\mu([S,T_n))=
        	\mu([S,T_{S,-\infty}])\mathbb{1}_E+\mu([S,T_{S,-\infty}))\mathbb{1}_{E^c}=0.
        \]
        Therefore by $\mu$-right-upper-semicontinuity in expectation of $X$ and 
        because $X$ is of class($D^\Lambda$) we get \eqref{sto:eq_123}
        by
        \begin{align*}
        	\EW\left[\lsr{X}{T_{S,-\infty}}\mathbb{1}_{E}+
        	 X_{T_{S,-\infty}}\mathbb{1}_{E^c}\right]
        			= \EW[\lim_{n\rightarrow \infty} X_{T_n}]
        			= \lim_{n\rightarrow \infty} \EW[X_{T_n}]
        			\leq \EW[X_S].
        \end{align*}				
        \emph{\textbf{Proof of the above claim:}}
			First we can calculate that
			\begin{align*}
				E^c&=\left\{T_{S,-\infty}=T_{S,-n} \text{ and } X_{T_{S,-\infty}}
							=Y_{T_{S,-\infty}}^{-n} \text{ for some $n\in \NZ$}\right\}\\
				&=\bigcup_{n=1}^\infty \left(\left\{T_{S,-\infty}=T_{S,-n}\right\}
					\cap \left\{X_{T_{S,-\infty}}=Y_{T_{S,-\infty}}^{-n}\right\}\right)\\
				&=\bigcup_{n=1}^\infty \left(\left\{T_{S,-\infty}=T_{S,-n}\right\}
					\cap (H_{S,-n}^-\cup H_{S,-n})\right).
			\end{align*}
			The first equality follows by monotonicity of $\ell\mapsto Y^\ell$, $\ell \mapsto T_{S,\ell}$ and the
			third by \eqref{sto:eq_12}.
			Now we define for $n\in \NZ$ the sets
			 \[
			 	A^{-n}:=\left\{T_{S,-\infty}=T_{S,-n}\right\}\cap (H_{S,-n}^-\cup H_{S,-n})
			 \]
			 and we claim that $(T_{S,-n})_{A^{-n}}$ is a $\Lambda$-stopping time, i.e.
			 $\stsetRO{(T_{S,-n})_{A^{-n}}}{\infty}\in \Lambda$.
			Indeed, we have by Lemma \ref{Lem:sp3} (iv), for $n\in \NZ$ that
			$(T_{S,-n})_{H_{S,-n}^-\cup H_{S,-n}}$ is a $\Lambda$-stopping time
			 and as
			 \begin{align*}
			 	A^{-n}
			 	 	\in \aFA_{(T_{S,-n})_{(H_{S,-n}^-\cup H_{S,-n})}}
			 \end{align*}
			 also $(T_{S,-n})_{A^{-n}}$ is a $\Lambda$-stopping time. Finally we get
			\begin{align*}
			 	\stsetRO{(T_{S,-\infty})_{E^c}}{\infty}
			 		&=\bigcup_{n=1}^\infty 
			 	\left\{(\omega,t)\midG T_{S,-\infty}(\omega)\leq t,\ \omega \in 
			 	 A^{-n}\right\}\\
			 	 &=\bigcup_{n=1}^\infty 
			 	\left\{(\omega,t)\midG T_{S,-n}(\omega)\leq t,\ \omega \in 
			 	 A^{-n}\right\}\\
			 &=	\bigcup_{n=1}^\infty
			 				\stsetRO{(T_{S,-n})_{A^{-n}}}{\infty}\in \Lambda.
			 \end{align*}
				
			\emph{Proof of (ix)}: By Proposition \ref{sto:Pro_46} below we have, up to a $\WM$-null set,
			\[
				E:=\{\WM\left(\mu([S,\infty))>0 \midG \aFA_S\right)=0\}
				\subset \{\mu([S,\infty))=0\}.
			\]			
			Hence we have for $S_E$ that $\mu([S_E,\infty))=0$ a.s.
			and for any $T\in \stmg{S}$ that $\mu([T_E,\infty))=0$ a.s..
			Therefore we have by the properties of $X$ that $X_{T_E}
			=0$ and consequently we have $Y_{S_E}^\ell=0=X_{S_E}$
			for all $\ell$.
	    
		
		\subsection{Proof of Lemma \ref{Lem:s4}}\label{proof:4.2}
			We first note that $L$ is $\Lambda$-measurable and $L_S$ for $S\in \stm$
			takes the value $\infty$ at most on the set 
			$\{\WM\left(\mu([S,\infty))>0\midG \aFA_S\right)=0\})$ almost surely, 
			which can be verified readily as in \citing{BK04}{Proof of Lemma 4.14}{1066}.
			
			Next fix $S\in \stmr$ and assume by way of contradiction
			 that $\WM(\{L_S=-\infty\}\cap \{\WM\left(\mu(\{S\})>0
				\midG \aFA_S\right)>0\})>0$.
			Then we can use the following claim, to be proven at the end:
			
			\textbf{Claim:}
			 There exists a sequence
			 $(T_k)_{k\in \NZ}\subset \stmsg{S}$ and a 
			corresponding non-increasing sequence $(E_k)_{k\in \NZ}\in \aFA_S$ such that
			\begin{align}\label{sto:eq_337}
				X_{S}<  \EW\left[X_{T_k}+\int_{[S,T_k)} 
				g_t\left(-k\right)\mu(\md t)\midG \aFA_S\right]
				\quad \text{ on $E_k$}
			\end{align}
			with
			\begin{align}\label{sto:eq_338}
				E_k\subset \{S<\infty\}\cap\{\WM\left(\mu(\{S\})>0
				\midG \aFA_S\right)>0\}
			\end{align}
			and
			\[
				\WM\left(\bigcap_{k=1}^\infty E_k \right)>0.
			\]					
			Using this claim and recalling that $g_S(-n)<0$ for all $n\in \NZ$ 	
		 we get on $\bigcap_{k=1}^\infty E_k$ 
					that for all $n\in \NZ$
			\begin{align}
				-\infty<-M_{S}^X\leq X_S<
				&\ \EW\left[X_{T_n}+\int_{[S,T_n)} 
				g_t\left(-n\right)\mu(\md t)\midG \aFA_S\right]\\
				&\leq  M^X_S+\EW\left[\int_{[S,T_n)} 
				g_t\left(-n\right)\mu(\md t)\midG \aFA_S\right]\\
				&\leq M^X_S+\EW\left[
				g_S\left(-n\right)\mu(\{S\})\midG \aFA_S\right]\label{sto:sol_gl_4}
			\end{align}
			with $M^X$ the martingale from Lemma \ref{sto:tools_lem}. 
			Observing that $0>g_S(-n)\downarrow -\infty$ for $n\rightarrow
			\infty$ we deduce from (\ref{sto:sol_gl_4}) that 
			$\WM\left(\mu(\{S\})>0
				\midG \aFA_S\right)=0$ a.s. on 
			$\bigcap_{k=1}^\infty E_k$. This contradicts the properties of the sets $(E_k)_{k\in \NZ}$
			 and finishes our proof once the above claim is proven.	
				
			\textbf{Proof of the above Claim:}					
			The family
			\[
				\left\{\EW\left[X_{T}+\int_{[S,T)} g_t\left(-k\right)
									\mu(\md t)\midG \aFA_S\right]\midG T\in \stmg{S}\right\}
			\]
			is upwards directed for any fixed $k\in \NZ$. Hence
			there exists by \citing{NE75}{Proposition VI-1-I}{121} for every $k\in \NZ$ a sequence  
			$(R_m^k)_{m\in \NZ}\subset \stmg{S}$, such that
			\[
				Y_S^{-k}=\lim_{m\rightarrow \infty} 
				\EW\left[X_{R_m^k}+\int_{[S,R_m^k)} g_t\left(-k\right)
									\mu(\md t)\midG \aFA_S\right],
			\]					
			where the limit on the right hand side is non-decreasing. Note furthermore that
			\begin{align}\label{sto:eq_91}
				\left\{L_S=-\infty\right\}=
				\left\{X_S<Y^{-k}_S \text{ for all } k\in \NZ\right\}.
			\end{align}
			Now we can construct $T_k$ and $E_k$:
			\begin{compactitem}
				\item Construction of $T_1$ and $E_1$. Consider for $m=1,2,\dots$ the sets
					\begin{align*}
						\tilde{E}_m^1&:=\left\{
						X_S< \EW\left[X_{R_m^1}+\int_{[S,R_m^1)} g_t\left(-1\right)\mu(\md t)
						\midG \aFA_S\right]\right\}\\
						&\hspace{20ex}\cap \{L_S=-\infty\}\cap 
						\left\{\WM\left(\mu(\{S\})>0
				\midG \aFA_S\right)>0\right\}\in \aFA_S,
					\end{align*}
					which, up to a $\WM$-null set, grow to 
					\[\{L_S=-\infty\}\cap 
					\left\{\WM\left(\mu(\{S\})>0	\midG \aFA_S\right)>0\right\},\]
					because of \eqref{sto:eq_91}. Now
					just choose $m$ large enough  to ensure that 
					\[
						\WM\left(\tilde{E}_m^1\right)>\frac12\left(\WM
						\left[\left\{L_S=-\infty\right\}\cap
						\left\{\WM\left(\mu(\{S\})>0
				\midG \aFA_S\right)>0\right\}\right]\right),
					\]
					and set $E_1:=\tilde{E}_m^1$ and $T_1:=R_m^1$.
				\item Fix $k\in \NZ$, $k\geq 1$ and assume $E_{k}$ and $T_{k}$ have been constructed already.
					Now we construct $T_{k+1}$ and $E_{k+1}$. 
					Similarly to above, consider the sequence of sets
					\begin{align*}
						\tilde{E}_m^{k+1}&:=\left\{
						X_S< \EW\left[X_{R_m^{k+1}}+\int_{[S,R_m^{k+1})} g_t\left(-(k+1)\right) \mu( \md 
						t)\midG \aFA_S\right]\right\}\cap E_k,
					\end{align*}
					which, up to a $\WM$-null set, grows to $E_k$ 
					again by \eqref{sto:eq_91} and because by construction
					\[
						E_k\subset \{L_S=-\infty\}\cap 
					\{\WM\left(\mu(\{S\})>0
				\midG \aFA_S\right)>0\}.
					\]
					For $m$ large enough with
					\[
						\WM\left(\tilde{E}_m^{k+1}\right)>
						\frac12\WM(\{L_S=-\infty\}\cap\{\WM\left(\mu(\{S\})>0
				\midG \aFA_S\right)>0\}),
					\]
					 we set $E_{k+1}:=\tilde{E}_{m}^{k+1}$
					and $T_{k+1}:=R_{m}^{k+1}$.
				\item The stopping times $T_k$ and sets $E_k$, $k\in \NZ$,
				are as requiered by \eqref{sto:eq_337} and \eqref{sto:eq_338}
					by construction. We also obtain
				\begin{align*}
					\WM\left(\bigcap_{k=1}^\infty E_k \right)
					&= \lim_{k\rightarrow \infty}\WM\left(E_k\right)\\
					&\geq \frac12 \left(\WM(\{L_S=-\infty\}\cap\{\WM\left(\mu(\{S\})>0
				\midG \aFA_S\right)>0\})\right)>0,
				\end{align*}
				completing the proof of our claim.
			\end{compactitem}

		
		\subsection{Proof of Lemma \ref{Lem:s6}}\label{proof:4.3}
	    	
			Fix $S\in \stmr$. On the set $\{\WM\left(\mu([S,\infty))>0
			\midG \aFA_S\right)=0\}$
			we have by Lemma \ref{Lem:sp3} (ix) that $X_S=Y_S^\ell$ and $S=T_{S,\ell}$ for all $\ell\in \RZ$
			and hence $L_S=\infty$. On the other hand, we have by definition of $\ell_{S,T}$ that also
			$\ell_{S,T}=\infty$ for all $T\in \stmg{S}$, which shows 
			\eqref{sto:sol_gl_1} on $\{\WM\left(\mu([S,\infty))>0
			\midG \aFA_S\right)=0\}$.
			From now on we focus on the set $\{\WM\left(\mu([S,\infty))>0
			\midG \aFA_S\right)>0\}$. 	
			We start by noting that for an $\aFA_S$-measurable  $R:\Omega \rightarrow 
			\RZ\cup \{-\infty\}$ 
			we have
			\begin{align}\label{sto:gl_77}
				Y^R_S-X_S=\esssup_{T\in \stmg{S}} \EW\left[X_T-X_S
				+\int_{[S,T)}g_t(R)\mu(\md t)\midG \aFA_S\right].
			\end{align}					
			
			Next we state the following first claim:
			
			\textbf{Claim 1:} Fix $n\in \NZ$ and define the $\aFA_S$-measurable random variable
			\[
				K^n:=\left(L_S+\frac{1}{n}\right)
				\mathbb{1}_{\{L_S>-\infty\}}-n\mathbb{1}_{\{L_S=-\infty\}}
				>L_S \quad \text{ on } \{L_S<\infty\}
			\]
			and $K^n:=\infty$ on $\{L_S=\infty\}$ with $K^n>-\infty$. Then
			there exists for $n=1,2,\dots$ 
			a sequence $(T_m^n)_{m\in \NZ} \subset \stmg{S}$ such that
			\begin{align*}
				Y^{K^n}_S-X_S
						&=\lim_{m\rightarrow\infty} \EW\left[X_{T_m^n}-X_S+\int_{[S,
						T_m^n)}g_t(K^n)\mu(\md t)\midG \aFA_S\right],
			\end{align*}
			where the preceding limit is non-decreasing.
			
			\textbf{Proof of Claim 1:}
			Indeed, by \eqref{sto:gl_77} we have					
			\begin{align*}
				Y^{K^n}_S-X_S
				=\esssup_{T\in \stmg{S}} \EW\left[X_T-X_S
				+\int_{[S,T)}
						g_t(K^n)\mu(\md t)\midG \aFA_S\right].
			\end{align*}					
			As the family
			\[
				\left\{\EW\left[X_T-X_S+\int_{[S,T)}g_t(K^n)\mu(\md t)\midG \aFA_S\right]
				\midG T\in \stmg{S}\right\}
			\]					
			is upwards directed the rest follows from \citing{NE75}{Proposition VI-1-I}{121}.
			
			Before concluding the proof let us state one more claim,
			which we will prove at the end of this section:
			
		\textbf{Claim 2:}
				We have 
				\begin{align}
					K^n\geq \essinf_{T\in \stmsg{S}} \ell_{S,T} \label{sto:gl_70}
				\end{align}			
				for all $n\in \NZ$.
			
			\textbf{Conclusion of the proof:}			
				First we prove \eqref{sto:eq_1}. Indeed, on $\{L_S>-\infty\}$
			 we obtain by continuity of $\ell\mapsto Y^\ell_S$
			(Lemma \ref{Lem:sp3} (viii)) that
			$X_S=Y_S^{L_S}$. On the other hand we get by 
			Lemma \ref{Lem:sp3} (viii) that $X_S=Y_S^{-\infty}=Y_S^{L_S}$ a.s.
			on $\{L_S=-\infty\}$, which proves the claimed equation. 			
			Second we get ``$\geq$'' in \eqref{sto:sol_gl_1} by letting 
			$n\uparrow \infty$ in \eqref{sto:gl_70}.			
			Next we get for any $T'\in \stmsg{S}$ by \eqref{sto:gl_77} that
				\begin{align*}
					X_S&=Y^{L_S}_S
						=\esssup_{T\in \stmg{S}}\EW\left[X_T+\int_{[S,T)}g_t(L_S)\mu(\md t)\midG
						 \aFA_S\right]\\
						&\geq 
						\EW\left[X_{T'}+\int_{[S,T')}g_t(L_S)\mu(\md t)\midG
						 \aFA_S\right].
				\end{align*}
			 As $L_S\in \aFA_S$ we obtain $L_S\leq \ell_{S,T'}$
			for any such $T'$,
			 which shows  ``$\leq$'' in \eqref{sto:sol_gl_1}
			and finishes the proof once we have proven Claim 2.					
				
			\textbf{Proof of Claim 2:}
			By $K^n:=\infty$ on $\{L_S=\infty\}$ the inequality 
			\eqref{sto:gl_70} is only to be proven on 
				$\{L_S<\infty\}\subset \{\WM\left(\mu([S,\infty))>0
			\midG \aFA_S\right)>0\}$.
			On this set we have $\ell_{S,\infty}<\infty$  and hence also $\essinf_{T\in \stmsg{S}} \ell_{S,T}<\infty$.					
			 For the proof of Claim 2 we fix $n\in \NZ$. Then we set $(T_m^n)_{m\in \NZ}$ to be the sequence in $\stmg{S}$
			  given by Claim 1 and we define $A_m:=\left\{\mathcal{E}_{T_m^n}>0\right\}$
			  with
			  \[
			  	\mathcal{E}_{T_m^n}:=\EW\left[X_{T_m^n}-X_S+\int_{[S,{T_m^n})}g_t(K^n)\mu(\md t)\midG \aFA_S\right].
			  \]
				Now we have that 
				\begin{align*}\label{sto:eq_13}
						A_m\subset 
				\{\WM\left(\mu([S,T_m^n))>0
			\midG \aFA_S\right)>0\}=:E_m,
			\end{align*}
			which follows by $\mu$-right-upper-semicontinuity in expectation of $X$, Proposition \ref{eq_rusc} and Proposition \ref{sto:Pro_46} below.
			
			By definition of $K$ and $L$ we have 
			\[
				\{L_S<\infty\}\subset \{0<\lim_{m\rightarrow \infty} \mathcal{E}_{T_m^n}\}
				\]
			and with the monotonicity of $\mathcal{E}_{T_m^n}$ in $m$ we
			also have 
			$A_m\subset A_{m+1}$ for any $m\in \NZ$. As we have by Lemma \ref{Lem:s6} that
			\[
				\WM(\{L_S=\infty\}\cap \WM(\mu([S,\infty)>0|\aFA_S)>0\})=0
			\]
			we get
			\begin{align*}
				\WM(\{\WM(\mu([S,\infty)>0|\aFA_S)>0\})&\leq 
				\WM(\{L_S<\infty\}
				\leq  \WM(\{0<\lim_{m\rightarrow \infty} \mathcal{E}_{T_m^n}\})\\
				&=\lim_{m\rightarrow \infty} \WM(A_m)
				\leq \WM(\{\WM(\mu([S,\infty)>0|\aFA_S)>0\}),
			\end{align*}
			which actually shows
			\begin{align}\label{sto:gl_71}
				\lim_{m\rightarrow \infty} \WM(A_m)=
				\WM(\{\WM(\mu([S,\infty)>0|\aFA_S)>0\}).
			\end{align}
			Furthermore we have $A_m\subset \{T_m^n>S\}$ and therefore 
			we get for any $m\in \NZ$ on $A_m$
			\begin{align*}
				\EW\left[\int_{[S,T_m^n)} g_t(\ell_{S,T_m^n})\mu(\md t)\midG \aFA_S\right]
				&=\EW\left[X_S-X_{T_m^n}\midG \aFA_S \right]\\
				&<\EW\left[\int_{[S,T_m^n)} g_t(K^n)\mu(\md t)\midG \aFA_S\right].
			\end{align*}		
			 As $K^n$ is $\aFA_S$-measurable this leads for any $m\in \NZ$
			to
			\[
				K^n\geq \ell_{S,T_m^n}\geq \essinf_{T\in \stmsg{S}} \ell_{S,T} \quad \text{ on $A_m$},
			\]
			and, so, by (\ref{sto:gl_71}), the inequality (\ref{sto:gl_70}) holds almost surely on
			$\{\WM\left(\mu([S,\infty))>0\midG \aFA_S\right)>0\}$,
			which finishes the proof as we already proven that we only
			have to focus on the latter set.
	    
    \subsection{Proof of Lemma \ref{Lem:s5}}\label{proof:4.4}
    
        \subsubsection{Preliminary results for the Proof of Lemma \ref{Lem:s5}}
            
            We start with the following simple observation:
	
    		\begin{Pro}\label{sto:Pro_46}
    			Let $\mathcal{G}_1\subset\mathcal{G}_2$ be nested $\sigma$-fields on $\Omega$.
    			Then for any $\mathcal{G}_2$-measurable random variable $X\geq 0$ and $B\in \mathcal{G}_1$
    			with $\{X>0\}\subset B$ we have
    			\begin{align}\label{sto:eq_33}
    				\left\{X>0\right\} \subset \left\{\EW\left[X\midG \mathcal{G}_1\right]>0\right\}
    				\subset B
    			\end{align}
    			up to a $\WM$-null set.
    			In particular, for $A\in \mathcal{G}_2$, $B\in \mathcal{G}_1$ and $A\subset B$ we have
    			\begin{align}\label{sto:eq_44}
    				A\subset \left\{\WM\left(A\midG \mathcal{G}_1\right)
    								>0\right\}\subset B
    			\end{align}
    			up to a $\WM$-null set.
    		\end{Pro}
    		
    		\begin{proof}
    		    For the first inclusion in \eqref{sto:eq_33}, note that
    				\[
    				\EW\left[X\mathbb{1}_{\left\{\EW\left[X\,|\, \mathcal{G}_1\right]
    								=0\right\}}\right]
    				=\EW\left[\EW\left[X\midG \mathcal{G}_1\right]\mathbb{1}_{\left\{	\EW\left[X\, |\, \mathcal{G}_1\right]
    								=0\right\}}\right]
    				=0,
    			\]
    			which proves 	$\left\{X>0\right\} \subset \left\{\EW\left[X\midG \mathcal{G}_1\right]
    								>0\right\}$ up to a $\WM$-null set. 
    								
    			For the second inclusion in \eqref{sto:eq_33}, observe that $\{X>0\}
    			\subset B$ and $B\in \mathcal{G}_1$ yields 
    			\[
    				\EW\left[\EW\left[X\midG \mathcal{G}_1\right]\mathbb{1}_{B^c}\right]
    				=\EW\left[X\mathbb{1}_{B^c}\right]=0,
    			\]
    			which implies by $X\geq 0$ that 
    			$\EW\left[X\midG \mathcal{G}_1\right]\mathbb{1}_{B^c}=0$
    			almost surely and therefore 
    			\begin{align}
    			    \left\{\EW\left[X\midG \mathcal{G}_1\right]
    								>0\right\}\subset B \quad \text{ up to a $\WM$-null set.}
    			\end{align}
    		\end{proof}

		Next we need to prove some technical result, which is used in the proof of Lemma 
		\ref{Lem:s5}.
	
		\begin{Pro}\label{sto:pro_70}
			Consider $S\in \stm$, $\ell\in \RZ$ and let $T\in \stmg{S}$ with $\WM(T<\infty)>0$
			  be such that $\bar{L}_{S,T}<\ell$
					on $\{T<\infty\}$ where
			\begin{align}\label{sto:eq_2}
						\bar{L}_{S,t}(\omega):=\begin{cases}
													\sup\limits_{v\in [S(\omega),t]} L_v(\omega)
														&\text{ for } t\geq S(\omega),\\
													-\infty 									&\text{ for } t< S(\omega),
												\end{cases}
					\end{align}		
			with $L$ as in Lemma \ref{Lem:s4}.
			
			Then for any $U\in \stmgs{S}{T}$ 
			the set $E^U :=\{\WM(T<\infty|\aFA_U)>0\}$ has strictly positive
			probability and there exists $R^U\in \stmg{U}$ such that $R^U=\infty$
			on $(E^U)^c$ and $R^U>T$, $\WM(\mu([U,R^U))>0|\aFA_U)>0$, 
			$X_U\leq  \EW\left[X_{R^U}+\int_{[U,R^U)} 
								g_t\left(\ell\right)\mu(\md t)\midG \aFA_U\right]$
 on $E^U$.
			\end{Pro}

			\begin{Rem}
			In the stochastic setting we have to manage two difficulties. First we have no total order relation  on the set
			of stopping times, which makes it difficult to choose a \emph{largest} stopping time satisfying a specific property. We will overcome this by a partial order and an application of the Hausdorff Maximality Theorem.
			Second everything is conditioned on the information known up to time $U$, which makes it more difficult to prove
			that the obtained maximal element is larger than $T$ as the objects we want to compare are not known at $U$.
		\end{Rem}		
	
			\begin{proof}[Proof of Proposition \ref{sto:pro_70}]
					\emph{Constructing a sequence of sets $E_m^U$ exhausting $E^U$:}
						As the family
					\[
						\left\{\EW\left[X_R+\int_{[U,R)}g_t(\ell)\mu(\md t)\midG 
						\aFA_U\right]
						\midG R\in \stmg{U}\right\}
					\]					
					is upwards directed 
					there exists by \citing{NE75}{Proposition VI-1-I}{121}
					 a sequence of stopping times
					$(\tilde{R}_m^U)_{m\in \NZ}\subset \stmg{U}$ such that
					with $Y_U^\ell$ from Lemma \ref{Lem:sp3} we have
					\begin{align}\label{sto:eq_23}
					\EW\left[X_{\tilde{R}_m^U}+\int_{[U,\tilde{R}_m^U)} g_t\left(\ell\right)
											\mu(\md t)\midG \aFA_U\right]\nearrow Y_U^\ell
											\quad \text{ as $m\uparrow \infty$.}
					\end{align}			
					 Let us define for $m\in \NZ$
					\begin{align}\label{sto:eq_22}
						E_m^U:=\ \left\{
						X_U< \EW\left[X_{\tilde{R}_m^U}+\int_{[U,\tilde{R}_m^U)}
						g_t\left(\ell\right)\mu( \md t)
						\midG \aFA_U\right]\right\}
						\cap E^U \in \aFA_U.
					\end{align}			
					This gives us a non-decreasing sequence of sets whose union
					is $\{X_U<Y_U^{\ell}\}\cap E^U$, up to a $\WM$-null set. Indeed,
					as $U\leq  T$,  we know that $X_U<Y_U^{\ell}$ 
					at least on the set $\{T<\infty\}$, because by assumption on $T$,
					\begin{align*}
						\left\{T<\infty\right\}
						\subset \left\{\bar{L}_{S,T}<\ell\right\}
						\subset \left\{\bar{L}_{S,U}<\ell\right\}
						\subset \left\{X_v<Y_v^\ell \text{ for } v\in [S,U] \right\}.
					\end{align*}
					Hence, $\left\{T<\infty\right\}
						\subset \left\{X_U<Y_U^\ell\right\}\in \aFA_U$
					and so by Proposition \ref{sto:Pro_46} we have					$\left\{T<\infty\right\}\subset
						E^U \subset \left\{X_U<Y_U^\ell\right\}$.
					Thus, $(E_m^U)_{m\in \NZ}$ grows to $E^U$ and 
					\begin{align}\label{sto:eq_24}
						0<\WM(T<\infty)\leq \WM\left(E^U\right)
						=\WM\left(\bigcup_{m=1}^\infty E_m^U\right).
					\end{align}
					Therefore, there exists $M\in \NZ$ such that $\WM(E_m^U)>0$
					for $m\geq M$
					and we assume without loss of generality
					that in fact $\WM(E_m^U)>0$ holds for all $m\in \NZ$.

					\emph{Fix $m\in \NZ$ and construct 
					$R_m^U$ corresponding to $E_m^U$ such that $R_m^U$ satisfies the desired conditions for
					$R^U$ on $E_m^U$:} 
					Define
					\begin{align*}
						\Theta_m^U :=\left\{R\in \stmg{U} \midG 
						X_U\leq \EW\left[X_R+\int_{[U,R)} g_t\left(\ell
						\right)\mu(\md t)\midG \aFA_U\right]\text{ on } E_m^U,\right.\\ 
						\WM\left(\mu([U,R))>0\,\middle|\, \aFA_U\right)>0 \text{ on } E_m^U 
						\left.\text{ and } R=\infty \text{ on } (E_m^U)^c\right\}.
					\end{align*}
					For $T_1,T_2\in \stm$ we define $
					T_1\leq T_2:\Leftrightarrow
					T_1(\omega)\leq T_2(\omega)$ for a.e. $\omega\in \Omega$.
					 This defines a partial order (see, e.g.,
					\citing{RU64}{4.20}{87}).					
					By definition of $E_m^U$ the set $\Theta_m^U$ is 
					a partially ordered set, which is nonempty as it contains 
					$(\tilde{R}_m^U)_{E_m^U}$ from \eqref{sto:eq_23}. Indeed, by \eqref{sto:eq_22} we just have to show 
					$\WM (\mu([U,\tilde{R}^U_m))>0 | \aFA_U)>0$
					on $E_m^U$. On the
					set $\Psi:=\{\WM(\mu([U,\tilde{R}_m^U))>0 | \aFA_U)=0\}
					\cap E_m^U\in \aFA_U$
					we have by Proposition \ref{sto:Pro_46} that
					$\mu([U,\tilde{R}_m^U))=0$. Hence we get on $\Psi$ by definition
					of $E_m^U$ that 
					\[
						X_U< \EW\left[X_{\tilde{R}_m^U}+\int_{[U,\tilde{R}_m^U)}
						g_t\left(\ell\right)\mu( \md t)
						\midG \aFA_U\right]
						=\EW\left[X_{\tilde{R}_m^U}
						\midG \aFA_U\right]\leq X_U,
					\]
					which is a contradiction. Here the first inequality
										came from the definition of $E_m^U$,
										the first equality follows from 
										$\mu([U,\tilde{R}_m^U))=0$ and the last inequality
									is due to the	$\mu$-right-upper-semicontinuity 
									in expectation of $X$ and Proposition 
									 \ref{eq_rusc}.
										
					Now we have seen that $\Theta_m^U$ is a partially ordered 
					non-empty set and hence we obtain by the Hausdorff Maximality
					Theorem
					(e.g. \citing{RU64}{4.21}{87}),
					that there exists a maximal totally ordered
					subset $\tilde{\Theta}_m^U$:
					 For any two elements $\bar{R}_1,\bar{R}_2$ of $\tilde{\Theta}_m^U$, we have
					 $\bar{R}_1\leq \bar{R}_2$ or $\bar{R}_2\leq \bar{R}_1$ and, if we add an element
					of $\Theta_m^U\backslash \tilde{\Theta}_m^U$, then the resulting set is not totally ordered any more.
					
					Now set $R_m^U:=\esssup_{\bar{R}\in \tilde{\Theta}_m^U}
						 \bar{R}$. 
						 As the set $\tilde{\Theta}_m^U$ is totally ordered, 
						 it is in particular upwards directed and hence
						 by \citing{NE75}{Proposition VI-1-I}{121}
						there exists a non-decreasing
						  sequence $(\bar{R}_{m,k}^{U})_{k\in \NZ}$
							in $\tilde{\Theta}_m^U$ with
						 $R_m^U=\lim_{k\rightarrow \infty}\bar{R}_{m,k}^U$.
						 This shows $R_m^U=\infty$ on $(E_m^U)^c$. 					  
						  Observing
						\[						 
						 	\stsetRO{R_m^U}{\infty}\ =\bigcap_{k=1}^{\infty} \stsetRO{\bar{R}_{m,k}^{U}}{\infty}
						 \]
						 we see that $R_m^U\in \stmg{U}$. 
						In the following we will omit 
						 for notational simplicity to emphasize the dependence of $R_m^U$, $\bar{R}_{m,k}^U$,
						 $E_m^U$, $\Theta_m^U$, $\tilde{\Theta}_m^U$
						 on $m$, $U$ and will instead work with $R$, $\bar{R}_k$, $E$, $\Theta$, $\tilde{\Theta}$.	
						 Now define  for $k\in \NZ$
						\begin{align}\label{sto:eq_40}
							I_k&:=\left\{\bar{R}_k<R\right\}\in \aFA_{R-},\nonumber\\
							I&:=\left\{\bar{R}_k<R \text{ for all $k$}\right\}
							=\bigcap_{k=1}^\infty I_k							\in \aFA_{R-}
						\end{align}
						with $I_k\subset I_{k-1}$ for all $k\in \NZ$.
						The sequence $\hat{R}_k:=(\bar{R}_k)_{I_k}\wedge k$, 
						$k\in \NZ$, announces $R_{I}$ 
						and $\{R_I=0\}=\emptyset\in \aFA_{0-}$, which shows
						by \citing{DM82}{Theorem 71}{128}, that $R_I$ is an $\FA$-predictable 
						stopping time.
						 As $R\leq R_I$ we also get $I\in \aFA_{R_I-}$ and by
						 Lemma \ref{sto:tools_lem} (ii) we have $\limsup_{k\rightarrow \infty} X_{\bar{R}_k}
						 \leq X_{R}$ on $\{R=\infty\}$. Hence, with Fatou's Lemma  this gives us 	on $E$ that
						\begin{align}\label{sto:gl_73}
							X_U&\leq \limsup_{k\rightarrow \infty}
							\EW\left[X_{\bar{R}_k}+\int_{[U,\bar{R}_k)} \nonumber
								g_t\left(\ell\right)\mu(\md t)\midG \aFA_U\right]\\
								&\overset{}{\leq} \EW\left[\lsl{X}_{R_I}\mathbb{1}_{
								\{R_{I}<\infty \}}
								+X_{R}\mathbb{1}_{I^c\cup\{R=\infty\}}
								+\int_{[U,R)} \nonumber
								 g_t\left(\ell\right)\mu(\md t)\midG \aFA_U\right]\\
								&\overset{\text{Lemma } \ref{sto:tools_lem} (ii)}{\leq} 
								\EW\left[{^\mathcal{P} X_{R_{I}}}\mathbb{1}_{	\{R_{I}<\infty \}}+X_{R}
								\mathbb{1}_{{I^c\cup\{R=\infty\}}}
								+\int_{[U,R)} \nonumber
								 g_t\left(\ell\right)\mu(\md t)\midG \aFA_U\right]\\
								&\overset{\text{Thm. } \ref{app:meyer_thm_3}}{=}
								\EW\left[\EW\left[{ X_{R_I}}\nonumber
								\mathbb{1}_{\{R_{I}<\infty \}}
								\midG \aFA_{R_{I}-}\right]								
								+X_{R}\mathbb{1}_{{I^c\cup\{R=\infty\}}}
								+\int_{[U,R)}\nonumber
								 g_t\left(\ell\right)\mu(\md t)\midG \aFA_U\right]\\
								&= \EW\left[X_{R}+\int_{[U,R)}
								 g_t\left(\ell\right)\mu(\md t)\midG \aFA_U\right]
						\end{align}
						  and hence $R \in \Theta$ as 
							$\WM(\mu([U,R))>0|
							\aFA_U)\geq \WM(\mu([U,\bar{R}_k))>0|\aFA_U)>0$
							already for $k=1$. 
						Here we have used in the last equality in \eqref{sto:gl_73} 
						that $\aFA_U \subset \aFA_{R_I-}$ holds.
						Indeed, let $A
						\in \aFA_U$. By \citing{DM78}{Theorem 56 (c), (56.2)}{118}, we have
						\begin{align}\label{sto:eq_3}
							A\cap \{U<R_I\}\in \aFA_{R_I-}.
						\end{align}			
						Since $R_I=R$ on $I$ and by \eqref{sto:eq_40} also $R>\bar{R}_k\geq U$ 
						for an arbitrary $k\in \NZ$ on $I$ we get
						\begin{align*}
							\left\{U<R_I\right\}
							=\left(\left\{U<\infty\right\}\cap I^c\right)
							\cup I.
						\end{align*}
						Hence,  $\left\{U<R_I\right\}^c
							=\{U=\infty\} \cap I^c$
						and so, in view of \eqref{sto:eq_3},
						it remains to show that we have
						$A\cap \{U=\infty\}\cap I^c\in \aFA_{R_I-}$.
						Actually we even have
						\[
							A\cap \{U=\infty\}\cap I^c\in \aFA_{U-}.
						\]
						This follows by \citing{DM78}{Theorem 56 (e)}{118},
						 if $A\cap I^c\in \aFA_{\infty-}$, which is clear by
						 $\aFA_\infty=\aFA_{\infty-}$.					
						Therefore by (\ref{sto:gl_73}) and $R\geq \bar{R}_k$
						  for all $k$ we get
						that $R \in \tilde{\Theta}$ by maximality of $\tilde{\Theta}$.
												
					By way of contradiction suppose now that $\WM(\{R\leq  T\}\cap E)>0$. 
					 As, by the properties of $T$,
					  \[
						\{R\leq T\}\cap \{T<\infty\}\subset \{R\leq T\}\cap\{X_{R}<Y^{l}_{R}\}\in \aFA_{R}
						\]
					  we obtain from Proposition \ref{sto:Pro_46} that even
						\[
							\{R\leq T\}\cap \left\{\WM\left(T<\infty\midG \aFA_{R}\right)
								>0\right\}\subset \{R\leq T\}\cap\{X_{R}<Y^{l}_{R}\}.
								\]
					Now we
					can construct analogously to the set $E$ 
					a set $\Gamma \in \aFA_{R}$ with $\WM(\Gamma)>0$,
					$\Gamma \subset \{R\leq T\}\cap 
					\{X_{R}<Y^{\ell}_{R}\}$
					and a stopping time $R_2 \geq R$ with
					\begin{align*}
					X_{R}< \EW\left[X_{R_2}+\int_{[R,R_2)}
					g_t\left(\ell\right)\mu(\md t)\midG 
					\aFA_{R}\right]\quad \text{on}\quad \Gamma,
					\end{align*}
					 which also implies $R_2>R$ there. 
					We set 
					\[
						\hat{R}:=R_{\Gamma^c}\wedge
									(R_2)_{\Gamma}=	\begin{cases}
										R & \text{ on } \quad\Gamma ^c,\\
										R_2 & \text{ on }\quad \Gamma,
										\end{cases}
					\]
					which gives us
					\begin{align*}
						X_{R}\leq \EW\left[X_{\hat{R}}+\int_{[R,\hat{R})}
						g_t\left(\ell\right)\mu(\md t)\midG 
						\aFA_{R}\right]\quad \text{ on $\Omega$.}
					\end{align*}
					This leads on $E$ to
					\begin{align*}
					X_U&\leq \EW\left[X_{R}+\int_{[U,R)} g_t\left(\ell\right)\mu(\md t)\midG 
					\aFA_{U}\right]\\
					&\leq \EW\left[\EW\left[X_{\hat{R}}
					+\int_{[R,\hat{R})} g_t\left(\ell\right)\mu(\md t)\midG 
					\aFA_{R}\right]+\int_{[U,R)} g_t\left(\ell\right)\mu(\md t)\midG 
					\aFA_{U}\right]\\
					&=\EW\left[X_{\hat{R}}+\int_{[U,\hat{R})} g_t\left(\ell\right)\mu(\md t)\midG 
					\aFA_{U}\right].
					\end{align*}										
					Hence $\hat{R}\in \Theta$, but as $\hat{R}\geq  R
					$ and $\hat{R}>R$ on $\Gamma$
					we could extend $\tilde{\Theta}$. This is a contradiction
					to the maximality of the totally ordered
					set $\tilde{\Theta}$. It follows that $\WM(\{R \leq T\}
					\cap E)=0$.
					
					\emph{Construct $R^U$ with the help of $R_m^U$ and $E_m^U$, $m\in \NZ$:} 
					From the previous steps we get a sequence of $\Lambda$-stopping times 
					$(R_m^U)_{m\in \NZ}$ and an increasing sequence of $\aFA_U$-measurable 
					sets $(E_m^U)_{m\in \NZ}$ with $R^U_m> T$ on $E_m^U$, 
					$\WM\left(\mu([U,R^U_m))>0\,\middle|\, \aFA_U\right)>0$ on $E^U_m$,					
						\[
						X_U\leq  \EW\left[X_{R_m^U}+\int_{[U,R_m^U)} 
						g_t\left(\ell\right)\mu(\md t)\midG \aFA_U\right]
						\quad \text{ on } \quad E^U_m,
					\]
					$R^U_m=\infty$ on $(E^U_m)^c$ and 
					\[
						0< \WM\left(E^U\right)
						=\WM\left(\bigcup_{m=1}^\infty E_m^U\right).
					\]
					Now define 
					\[
						R_U:=\bigwedge_{m=1}^\infty \left(R_m^U\right)_{E_m^U\backslash E_{m-1}^U}.
					\]
					As for any $m\in \NZ$ we have $E_m^U\backslash E_{m-1}^U\in \aFA_U\subset \aFA_{R_m^U}$ 
					the random variable $R_U$ is a $\Lambda$-stopping time as a countable minimum of 
					$\Lambda$-stopping times. Furthermore as for any $m\in \NZ$ we have $R_m^U>T$ on $E_m^U$ we have
					$R_U>T$ on $\cup_{m\in \NZ}E_m^U=E^U$
					and $R_U=\infty$ on $(E^U)^c$. Additionally we have for any $m\in \NZ$ by $E_m^U\backslash E_{m-1}^U\in \aFA_U$
					that
					\[
						X_U\leq  \EW\left[X_{R_m^U}+\int_{[U,R_m^U)} 
						g_t\left(\ell\right)\mu(\md t)\midG \aFA_U\right]=
						\EW\left[X_{R^U}+\int_{[U,R^U)} 
						g_t\left(\ell\right)\mu(\md t)\midG \aFA_U\right]
						\quad 
					\]
					on $E^U_m$. This implies	
					\[
						X_U\leq
						\EW\left[X_{R^U}+\int_{[U,R^U)} 
						g_t\left(\ell\right)\mu(\md t)\midG \aFA_U\right]
						\quad \text{ on } E^U
					\]			
					and analogously we get $\WM(\mu([U,R^U))>0|\aFA_U)>0$ on $E^U$, which shows that $R^U$ has the desired properties 
					and this finishes the proof of our Proposition.				
		    \end{proof}
		    
		   \subsubsection{Proof of Lemma \ref{Lem:s5}}
		   
				First we get from \eqref{sto:useful_gl_8} that for \emph{any} choice of $\WM$-null set $\mathcal{N}$
					\begin{align*}
						B
							\subset \ \left\{(\omega,t,\ell)\in \bar{\Omega}_S^\mathcal{N}
						\midG
						Y^\ell_v(\omega)>X_v(\omega)
							\text{ for all } v\in [S(\omega),t)\right\}
							\subset C,
					\end{align*}
					and analogously $\tilde{B}\subset \tilde{C}$.
					To see that $\tilde{A}\subset \tilde{B}$, note first that again for \emph{any} choice of $\WM$-null set
					$\mathcal{N}$
					\begin{align}\label{sto:sol_gl_2}
						A
					\subset  &\left\{(\omega,t,\ell)\in \bar{\Omega}_S^\mathcal{N}
						\midG Y^\ell_v(\omega)>X_v(\omega)
					\text{ for all } v\in [S(\omega),t]\right\}.
					\end{align}
					Hence, for $(\omega,t,\ell)\in A$ with $X_{T_{S,\ell}(\omega)}
					(\omega)=Y_{T_{S,\ell}(\omega)}^\ell
						(\omega)$, i.e. for $(\omega,t,\ell)\in \tilde{A}$ ,
						we have $t<T_{S,\ell}(\omega)$, i.e. 
						$(\omega,t,\ell)\in \tilde{B}$.
						For the proof of $A\subset B$ we need the following
						auxiliary result, which is proven below:

					\textbf{Claim 1:}  For $\ell\in \RZ$ 
					we have outside an evanescent set, possibly depending 
					on $\ell$, that
					\begin{align}\label{sto:sol_gl_5}
						\stsetRO{S}{\infty}\ \cap\left\{\bar{L}_{S,\cdot}<\ell\right\}
						\subset \stsetC{S}{T_{S,\ell}},
					\end{align}
					where $\bar{L}_{S,\cdot}$ is defined in \eqref{sto:eq_2}.
					
					One can see that the left hand side (respectively right hand side)
					 of \eqref{sto:sol_gl_5} is the section of $A$ 
					(respectively $B$) for fixed $\ell\in \RZ$. 
					Therefore we obtain by Claim 1 a set $\mathcal{N}$ with $\WM(\mathcal{N})=0$ such that
					for $\omega\in \mathcal{N}^c$, $t\geq S(\omega)$ we have for all rational $\ell$
					that $\bar{L}_{S,t}(\omega)\leq \ell$ implies $t\leq T_{S,\ell}(\omega)$.
					Hence, for this choice of $\mathcal{N}$, 
					\begin{align}\label{sto:eq_133}
						A\cap
					(\Omega\times [0,\infty)\times \mathbb{Q})
						&\subset B\cap
					(\Omega\times [0,\infty)\times \mathbb{Q}).
					\end{align}
					Let us argue that in fact even $A\subset B$ holds for this choise
					of $\mathcal{N}^c$. Fix 
					$(\omega,t,\ell)\in A\subset \bar{\Omega}_S^\mathcal{N}$.
					Consider $(q_n)_{n\in \NZ} \subset \mathbb{Q}$
					 a sequence which increases
					strictly to $\ell$. Without loss of generality 
					$\bar{L}_{S,t}(\omega)<q_n$ for all $n\in \NZ$. 
					Therefore $(\omega,t,q_n)\in A$ and thus, by the choice of $\mathcal{N}$,
					we obtain $(\omega,t,q_n)\in B$, i.e. $t\leq T_{S,q_n}(\omega)$.
					As the sequence $(T_{S,q_n}(\omega))_{n\in \NZ}$
					is non-decreasing we obtain
					\[
						t\leq T_{S,q_n}(\omega)\leq T_{S,\ell}(\omega),
					\]
					which shows $(\omega,t,\ell)\in B$. 
					Hence, $A\subset B$ is proven once
					we have established Claim 1.					
									
					\textbf{Proof of Claim 1:}				
					Assume by way of contradiction that \eqref{sto:sol_gl_5}
					is not true. By Lemma \ref{Lem:sp3} (iii), $T_{S,\ell}$ is an 
					$\mathcal{F}_+^\Lambda$-stopping time and by
					\citing{EL80}{Theorem 2}{503}, we have $\stsetC{S}{T_{S,\ell}}\in 
					\Lambda$. Hence if the Claim 1 fails for some $\ell\in \RZ$
					 then 
					the Meyer Section Theorem 
					(see Theorem \ref{Main:4}) yields
					a stopping time $T\in \stmg{S}$ with $\WM(T<\infty)>0$,
					such that  $\bar{L}_{S,T}<\ell$ and $T_{S,\ell}<T$
					on $\{T<\infty\}$.				
					As we have
					\[
						\left\{\bar{L}_{S,T}
						<\ell\right\}
						=\bigcup\limits_{\ell'\in \mathbb{Q},\, \ell'<\ell} 
						\left\{  
						\bar{L}_{S,T} <\ell' \right\}
					\]
					 we can assume without loss of generality that there is an $\ell'\in \mathbb{Q}$ such that
					  $\bar{L}_{S,T}<\ell'$ 
					on $\{T<\infty\}$ for a fixed $\ell'\in \mathbb{Q}$ with $\ell'<\ell$.

					We will prove that the existence of such a $T$
					leads to a contradiction. 
					Define $(\tilde{U}_n)_{n\in \NZ}\subset \stmg{S}$ with $\tilde{U}_n\geq T_{S,\ell}$
					 as the non-increasing sequence from \citing{BB18_2}{Proposition 4.2 (i)}{19} such that
					\begin{align}\label{sto:eq_340}
						\tilde{U}_n\downarrow T_{S,l}\quad \text{ and }
						\quad 
						\lsr{X}{T_{S,\ell}}=\lim_{n\rightarrow \infty} X_{\tilde{U}_n}.
					\end{align}										
					With the help of this sequence we define the sequence of $\Lambda$-stopping times
					$(U_n)_{n\in \NZ}$, $U_n:=\tilde{U}_n\wedge T$ for $n\in \NZ$.
					Recall, that by assumption $\bar{L}_{S,T}<\ell$ and $T>T_{S,\ell}$ 
					on $\{T<\infty\}$, and by (\ref{sto:sol_gl_2})  this implies 
					$X_{T_{S,\ell}}<Y_{T_{S,\ell}}^\ell$ on $\{T<\infty\}$. Hence
					Proposition \ref{sto:Pro_46} and \eqref{sto:eq_77} in
					conjunction with \eqref{sto:eq_12}
					implies up to $\WM$-null sets the inclusions
					\begin{align}\label{sto:gl_78}
						\{T<\infty\}&\subset \left\{
						\WM\left(T<\infty\midG \aFA_{T_{S,\ell}+}\right)>0\right\}
						\subset 
						\left\{	
						X_{T_{S,\ell}}
						<Y^\ell_{T_{S,\ell}}\right\}\nonumber\\
						&\subset 
						\left\{	Y^\ell_{T_{S,\ell}+}
						=\lsr{X}{T_{S,\ell}}\right\}
						=\left\{	\lim_{n\rightarrow \infty} Y_{U_n}^\ell
						= \lim_{n\rightarrow \infty} X_{U_n}\right\}.
					\end{align}
					The desired contradiction will be deduced using the following
					result, which we
						will prove at the end:

				\textbf{Claim $2$:} For all $U\in \stmgs{S}{T}$, we have 
					\begin{align*}
						X_U
						\leq\ Y_U^\ell- \EW\left[
						\int_{[U,T]} \left(g_t(\ell)-g_t\left(\ell'\right)\right)
						\mu(\md t)\midG \aFA_U\right]
						\quad \text{on}\quad  \{\WM(T<\infty\, |\, \aFA_U)>0\}.
					\end{align*}

					Now define
					\[
						\Gamma:=\bigcap_{n=1}^\infty 
					\{\WM(T<\infty|\aFA_{U_n})>0\}						
					\]
					and by Proposition \ref{sto:Pro_46} we have
					\[
						\{T<\infty\}\subset \Gamma\subset \left\{\WM(T<\infty| \aFA_{T_{S,\ell}+})>0\right\}.
					\]					
					By Claim $2$, 
					for $U=U_n$, $n=1,2,\dots$, we get on $\Gamma$ that
					\begin{align}\label{sto:gl_79}
						\lsr{X}{T_{S,\ell}}
						&\leq\ Y^\ell_{T_{S,\ell}+}- \liminf_{n\rightarrow \infty}
						\EW\left[
						\int_{[U_n,T]} \left(g_t(\ell)-g_t\left(\ell'\right)\right)
						\mu(\md t)\midG \aFA_{U_n}\right]\nonumber\\
						&\overset{\eqref{sto:gl_78}}{=}\ \lsr{X}{T_{S,\ell}}-
						\liminf_{n\rightarrow \infty}
						\EW\left[
						\int_{[U_n,T]} \left(g_t(\ell)-g_t\left(\ell'\right)\right)
						\mu(\md t)\midG \aFA_{U_n}\right].
					\end{align}	
					The latter sequence of conditional expectations defines a backward supermartingale, which 
					converges almost surely to a random variable 
					$Z$ with
					\[
						Z\geq						
						\EW\left[
						\int_{(T_{S,\ell},T]} \left(g_t(\ell)-g_t\left(\ell'\right)\right)
						\mu (\md t)\midG \cap_{n=1}^{\infty} \aFA_{U_n}\right]\geq 0
					\]
						by \citing{DM82}{Theorem 30}{24}. In detail
					the family $(Z_{\tilde{n}})_{\tilde{n}\in \mathbb{Z}_{\leq -1}}$
					 adapted to $\mathcal{G}_{\tilde{n}}:=\aFA_{U_{-\tilde{n}}}$ for $\tilde{n}
					 \in \mathbb{Z}_{\leq -1}$,
					 defined by 
					\[
						Z_{\tilde{n}}:=\EW\left[
						\int_{[U_{-\tilde{n}},T]} \left(g_t(\ell)-g_t\left(\ell'\right)\right)
						\mu(\md t)\midG \mathcal{G}_{\tilde{n}}\right]\geq 0
					\]
					defines a backward supermartingale with
					\[
						\sup_{\tilde{n}\in \mathbb{Z}_{\leq -1}}\EW[|Z_{\tilde{n}}|]<\infty,
					\] 
					i.e. $\EW[Z_{\tilde{n}}|\mathcal{G}_{\tilde{n}-1}]\leq Z_{\tilde{n}-1}$ for
					$\tilde{n}\in \mathbb{Z}_{\leq -1}$.
					Therefore \citing{DM82}{Theorem 30}{24} implies 
					that $Z:=\lim_{\tilde{n}\rightarrow -\infty} Z_{\tilde{n}}$ almost surely
					exists and
					\begin{align}\label{sto:gl_80}
						Z\geq \EW\left[Z_{\tilde{n}}\midG \mathcal{G}_{-\infty}\right]
					\end{align}
					for all $\tilde{n}\in \mathbb{Z}_{\leq -1}$ with
					$\mathcal{G}_{-\infty}:=\cap_{\tilde{n}=-1}^{-\infty} 
					\mathcal{G}_{\tilde{n}}
					=\cap_{n=1}^{\infty} \aFA_{U_n}$.
					Estimate (\ref{sto:gl_80}) leads with Fatou's Lemma to
					\begin{align}\label{sto:gl_81}
						Z\geq 
						\liminf_{\tilde{n}\rightarrow -\infty} \EW\left[Z_{\tilde{n}}
						\midG \mathcal{G}_{-\infty}\right]
						\geq						
						\EW\left[
						\int_{(T_{S,\ell},T]} \left(g_t(\ell)-g_t\left(\ell'\right)\right)
						\mu (\md t)\midG \mathcal{G}_{-\infty}\right]\geq 0	 .
					\end{align}
					From \eqref{sto:gl_79} we thus infer that $Z=0$ on $\Gamma$
					and, hence, by \eqref{sto:gl_81} and $g_t(\ell)-g_t(\ell')>0$,
					\begin{align}\label{eq_200}
						\Gamma \subset \left\{\WM\left(\mu((T_{S,\ell},T])>0\midG
						\mathcal{G}_{-\infty}\right)=0\right\}
					\end{align}
					up to $\WM$-null sets.
					Finally this leads by Proposition \ref{sto:Pro_46} for any $n\in \NZ$ to
					\begin{align}\label{eq_201}
						\Gamma &\subset \left\{\WM\left(\mu((T_{S,\ell},T])>0\midG
						\mathcal{G}_{-\infty}\right)=0\right\}\nonumber\\
						&\subset \left\{\WM\left(\mu((T_{S,\ell},T])>0\midG
						\aFA_{U_n}\right)=0\right\}						
						\subset 
						\left\{\mu((T_{S,\ell},T])=0\right\}.
					\end{align}
					
					Let us now show that $\Gamma \subset \left\{\mu((T_{S,\ell},T])=0\right\}$
					yields a contradiction. For that we have to use another sequence of
					$\Lambda$-stopping times given by the following claim
					whose proof is deferred until the end:
							
					\textbf{Claim 3:}					
					There exists a sequence $(R_m)_{m\in \NZ}
					\subset \stmg{U_1}$ such that we have for
					all $n\in \NZ$
					\begin{align}\label{eq_202}
						Y_{U_n}^{\ell'}=\lim_{m\rightarrow \infty}
						\EW\left[X_{R_m}+\int_{(T,R_m)}g_t(\ell')\mu(\md t)
						\midG \aFA_{U_n}\right]\quad 
						\text{ on }\quad \Gamma.
					\end{align}
					With 
					$(R_m)_{m\in \NZ}$ from Claim 3 we get on $\Gamma$
					again by \eqref{sto:eq_340} and \eqref{sto:gl_78} that
					\begin{align}\label{eq_204}
						\lsr{X}{T_{S,\ell}}&=\lim_{n\rightarrow \infty} X_{U_n}
						\leq \lim_{n\rightarrow \infty} Y_{U_n}^{\ell'}\nonumber\\
						&=\lim_{n\rightarrow\infty}
						\left(\lim_{m\rightarrow \infty}
						\EW\left[X_{R_m}+\int_{(T,R_m)}g_t(\ell')\mu(\md t)
						\midG \aFA_{U_n}\right] \right)\nonumber\\
						&\leq\ Y^\ell_{T_{S,\ell}+}- \liminf_{n\rightarrow \infty}
						\liminf_{m\rightarrow \infty}
						\EW\left[
						\int_{(T,R_m)} \left(g_t(\ell)-g_t\left(\ell'\right)\right)
						\mu(\md t)\midG \aFA_{U_n}\right]\nonumber\\
						&=\ \lsr{X}{T_{S,\ell}}- \liminf_{n\rightarrow \infty}
						\liminf_{m\rightarrow \infty}
						\EW\left[
						\int_{(T,R_m)} \left(g_t(\ell)-g_t\left(\ell'\right)\right)
						\mu(\md t)\midG \aFA_{U_n}\right].
					\end{align}						
					Define for $\tilde{n}\in \mathbb{Z}_{\leq -1}$
					\[
						\tilde{Z}_{\tilde{n}}:=\liminf_{m\rightarrow \infty}
						\EW\left[
						\int_{(T,R_m)} \left(g_t(\ell)-g_t\left(\ell'\right)\right)
						\mu(\md t)\midG \aFA_{U_{-\tilde{n}}}\right].
					\]
					Then this defines again a backward supermartingale by 
					Fatou's Lemma with
					\[
						\sup_{\tilde{n}\in \mathbb{Z}_{\leq -1}}
						\EW\left[|\tilde{Z}_{\tilde{n}}|\right]
						\leq 2\EW\left[\int_{[0,\infty)}|g_t(\ell)|\mu(\md t)\right]
						<\infty.
					\]
					Hence again by \citing{DM82}{Theorem 30}{24} the
					limit $\tilde{Z}:=\lim_{\tilde{n}\rightarrow -\infty} \tilde{Z}_{\tilde{n}}$
					exists almost surely and
					\begin{align}\label{eq_205}
						\tilde{Z}\geq \EW\left[\tilde{Z}_{\tilde{n}}\midG 
						\bigcap_{n=1}^{\infty} \aFA_{U_n}\right]
					\end{align}
					for all $\tilde{n}\in \mathbb{Z}_{\leq -1}$.
					Finally we need the following result, which we also prove 
					at the end:
					
					\textbf{Claim 4}: There exists $\tilde{\Gamma}
					\subset \Gamma$ with $\WM(\tilde{\Gamma})>0$
					such that on $\tilde{\Gamma}$
					\[
						\EW\left[\tilde{Z}_{-1}\midG 
						\bigcap_{n=1}^{\infty} \aFA_{U_n}\right]>0.
					\]
					Combining Claim 4 with \eqref{eq_204} and \eqref{eq_205}
					gives us on $\tilde{\Gamma}$
					\begin{align*}
						\lsr{X}{T_{S,\ell}}
						\leq  \lsr{X}{T_{S,\ell}}- \tilde{Z}<\lsr{X}{T_{S,\ell}},
					\end{align*}						
					the	contradiction	needed to establish Claim 1.
					It remains to prove Claims 2,3 and 4.
					
					\textbf{Proof of Claim 2:} By Proposition \ref{sto:pro_70},
					there is a 	
					stopping time $R^U\in \stmg{T}$ such that on $E^U:=\{\WM(T<\infty|\aFA_U)>0\}$ we have
					\begin{align*}
						X_U
						&\leq \EW\left[X_{R^U}
						+\int_{[U,R^U)} g_t\left(\ell'\right)\mu(\md t)\midG \aFA_U\right]\\
						&\leq\ Y_U^\ell - \EW\left[\int_{[U,T]} 
						\left(g_t(\ell)-g_t\left(\ell'\right)\right)\mu(\md t)\midG
						\aFA_U\right].
					\end{align*}
					Here we have used $E^U\in \aFA_U$, $R^U> T$ on $E^U$
					and $g_t(\ell)-g_t(\ell')>0$ by $\ell'<\ell$.

				\textbf{Proof of Claim 3:}					
					As the family
					\[
						\left\{\EW\left[X_R+\int_{[U_1,R)}g_t(\ell')\mu(\md t)\midG 
						\aFA_{U_1}\right]
						\midG R\in \stmg{U_1}\right\}
					\]					
					is upwards directed 
					there exists by \citing{NE75}{Proposition VI-1-I}{121} a sequence of
					stopping times
					$(R_m)_{m\in \NZ}\subset \stmg{U_1}$, such that
					\begin{align}\label{sto:eq_25} 
					\EW\left[X_{R_m}+\int_{[U_1,R_m)} g_t\left(\ell'\right)
											\mu(\md t)\midG \aFA_{U_1}\right]\nearrow 
											Y_{U_1}^\ell
											\quad \text{ as $m\uparrow \infty$.}
					\end{align}
					As by \eqref{eq_201} 
					\[
						\Gamma \subset \left\{\WM\left(\mu((U_1,T])>0\midG
						\aFA_{U_1}\right)=0\right\}
						\subset \left\{\mu((U_1,T])=0\right\}
					\]
					equation \eqref{eq_202} holds for $n=1$.
					Let us now  show that
					the sequence $(R_m)_{m\in \NZ}$ from the previous
					step will also fulfil
					\eqref{eq_202} for arbitrary $n$. On $\Gamma$ we have by \eqref{eq_201} that
					$\mu([U_n,U_1])=0$ and therefore also
						\begin{align}\label{eq_203}
							Y_{U_n}^{\ell'}
							&=\esssup_{R\in \stmg{U_n}}
							\EW\left[X_R+\int_{[U_n,R)}g_t(\ell')\mu(\md t)
							\midG \aFA_{U_n}\right]\nonumber\\
							&=\esssup_{R\in \stmg{U_n}}
							\EW\left[
							\EW\left[X_R+\int_{[U_1,R)}g_t(\ell')\mu(\md t)
							\midG \aFA_{U_1}\right]
							\midG \aFA_{U_n}\right]\nonumber\\
							&\leq \EW\left[
							Y_{U_1}^{\ell'}
							\midG \aFA_{U_n}\right]=\EW\left[
						 \lim_{m\rightarrow \infty}
						\EW\left[X_{R_m}+\int_{(T,R_m)} g_t\left(\ell'\right)
											\mu(\md t)\midG \aFA_{U_1}\right]
							\midG \aFA_{U_n}\right]\nonumber\\
							&\overset{\eqref{eq_201}}{=}
							\lim_{m\rightarrow \infty}\EW\left[
						X_{R_m}+\int_{[U_1,R_m)} g_t\left(\ell'\right)
											\mu(\md t)
							\midG \aFA_{U_n}\right]
							\leq Y_{U_n}^{\ell'} \quad \text{ on }\quad \Gamma.
						\end{align}
						Here we have used additionally 
						dominated convergence in the fifth step, which
						is possible as 
						\[
							\left|\EW\left[X_{R_m}+\int_{(T,R_m)} g_t\left(\ell'\right)
											\mu(\md t)\midG \aFA_{U_1}\right]\right|
								\leq M_{U_1}^X+\EW\left[
											\int_{[0,\infty)}|g_t(\ell')|\mu(\md t)
											\midG \aFA_{U_1}\right]
						\]
						with $M^X$ the $\Lambda$-martingale of Lemma 
						\ref{sto:tools_lem}. Hence equation \eqref{eq_203}
						leads to
						\[
							Y_{U_n}^{\ell'}=
							\lim_{m\rightarrow \infty}\EW\left[
						X_{R_m}+\int_{(T,R_m)} g_t\left(\ell\right)
											\mu(\md t)
							\midG \aFA_{U_n}\right],
						\]
						which we wanted to prove.
						
		\textbf{Proof of Claim 4:} 
		We assume by way of contradiction that we have on $\Gamma$
		\[
			\EW\left[\tilde{Z}_{-1}\midG 
						\bigcap_{n=1}^{\infty} \aFA_{U_n}\right]=0,
		\]
		which leads by Proposition \ref{sto:Pro_46} to
		$\tilde{Z}_{-1}=0$ on $\Gamma$. With the short hand notation
		\[
			Q_m:=	\EW\left[
						\int_{(T,R_m)} \left(g_t(\ell)-g_t\left(\ell'\right)\right)
						\mu(\md t)\midG \aFA_{U_{1}}\right]
		\]
		this is equivalent to
		\begin{align}\label{eq_300}
			\liminf_{m\rightarrow \infty} Q_m = 0 \quad \text{ on } \quad \Gamma.
		\end{align}
		The rest of the proof is structured in the following way:
		\begin{compactenum}[(i)]
			\item First we construct a sequence
						$(\ubar{Q}_m)_{m\in \NZ}$ connected
						to some new constructed $\Lambda$-stopping	times $(\ubar{R}_m)_{m\in \NZ}$
						such that $\ubar{Q}_m=\min_{n=1,\dots,m} Q_n$
						and $\ubar{Q}_m>0$ on some set $\Gamma_2\in \aFA_{U_1}$
						with $\WM(\Gamma_2)>0$ and $\Gamma_2\subset \Gamma$.
			\item We show that there exists a subsequence 
						$(\ubar{R}_{m_k})_{k\in \NZ}$ of 
						$(\ubar{R}_m)_{m\in \NZ}$ such that on $\Gamma_2$
						we have
						\begin{compactenum}[(1)]
							\item $\lim_{k\rightarrow \infty} \mu([U_1,\ubar{R}_{m_k}))=0$,
							\item \begin{align}\label{sto:eq_344}
										Y_{U_1}^{\ell'}= \lim_{k\rightarrow \infty}
										\EW\left[X_{\ubar{R}_{m_k}}+\int_{(T,\ubar{R}_{m_k})}
										g_t(\ell')\mu(\md t)\midG
										\aFA_{U_1}\right]=\lim_{k\rightarrow \infty} \EW[X_{\ubar{R}_k}
									|\aFA_{U_1}].
									\end{align}
						\end{compactenum}
				\item We combine the previous points to obtain by
				$\mu$-right-upper-semicontinuity in expectation of $X$
				our desired contradiction.
		\end{compactenum}
		
		\textbf{Construction of $(\ubar{Q}_m)_{m\in \NZ}$ and $(\ubar{R}_m)_{m\in \NZ}$:} We define
		as in the proof of Proposition \ref{sto:pro_70} the following increasing sequence of sets
		\[
			E_m:=\left\{
						X_{U_1}< \EW\left[X_{R_m}+\int_{[U_1,R_m)}
						g_t\left(\ell'\right)\mu( \md t)
						\midG \aFA_{U_1}\right]\right\}\\
						\cap \left\{\WM\left(T<\infty\midG \aFA_{U_1}\right)
								>0\right\},
		\]
		where $(R_m)_{m\in\NZ}$ is the sequence of $\Lambda$-stopping times
		constructed in Claim 3.
		Here we can assume
		without loss of generality that for 
		\[
			\Gamma_2:=\Gamma \cap E_1\in \aFA_{U_1}
		\]		
		we have $\WM(\Gamma_2)>0$. Indeed, as on $\{T<\infty\}$ we 
		have $X_{U_1}<Y_{U_1}^{\ell'}$ and the convergence property of the 
		sequence $(R_m)_{m\in \NZ}$ we see that 
		\[
			\bigcup_{m\in \NZ} E_m=\left\{\WM\left(T<\infty\midG \aFA_{U_1}\right)
								>0\right\}\supset \{T<\infty\}.
		\]
		Next we argue that on $\Gamma_2$ we have $Q_m>0$ for all $m\in \NZ$.
		Indeed, by Proposition \ref{sto:Pro_46}  the equation $Q_m=0$ for some $m\in \NZ$
		implies
		\[
			\int_{(T,R_m)}\left(g_t(\ell)-g_t(\ell')\right)\mu(\md t)=0
		\]		
		and by $g_t(\ell)-g_t(\ell')>0$ we get $\mu((T,R_m))=0$. 
		As $\Gamma_2\cap \{Q_m=0\}\subset E_1\subset E_m$  we get 
		by $\mu$-right-upper-semicontinuity in expectation of $X$ 
		combined with Proposition \ref{eq_rusc} 
		  the contradiction
		\begin{align*}
			X_{U_1}<\EW\left[X_{R_m}+\int_{[U_1,R_m)}
						g_t\left(\ell'\right)\mu( \md t)
						\midG \aFA_{U_1}\right]
						&\overset{\eqref{eq_201}}{=}\EW\left[X_{R_m}
						\midG \aFA_{U_1}\right]\\&\leq X_{U_1} \quad 
						\text{ on } \quad \Gamma_2\cap \{Q_m=0\}.
		\end{align*}
		Hence, $Q_m>0$ for all $m\in \NZ$ on $\Gamma_2$.

		Let us now define the sequence $(\ubar{R}_m)_{m\in \NZ}
		\subset \stmg{U_1}$
		by
		\begin{align*}
			\ubar{R}_m:=\sum_{p=1}^m (R_p)_{
			\{Q_p=\min_{k\in\{1,\dots, m\}} Q_k\}
			\cap \bigcap_{r=1}^{p-1}\{Q_r>\min_{k\in\{1,\dots, m\}} Q_k\}}
			\in \stmg{U_1},
		\end{align*}
		which means $\ubar{R}_m$ is equal to $R_j$, where $j$ is the first
		index for which $Q_j$ attains the	minimal value $\min_{k\in \{1,\dots,m\}} Q_k$.
		One can see that for $m\in \NZ$
		\[
			\ubar{Q}_m:=\EW\left[
						\int_{(T,\ubar{R}_m)} \left(g_t(\ell)-g_t\left(\ell'\right)\right)
						\mu(\md t)\midG \aFA_{U_{1}}\right]
			=\min_{k\in\{1,\dots, m\}}Q_k
		\]
			is non-increasing and hence $
			\lim_{m\rightarrow \infty} \ubar{Q}_m$
							exists. By \eqref{eq_300} we have
			\[
				\Gamma_2\subset \left\{\liminf_{m\rightarrow \infty}
				Q_m=0\right\}
				=\left\{\lim_{m\rightarrow \infty}
				\ubar{Q}_m=0\right\}=:E\in \aFA_{U_1}.
			\]
			We can assume without loss of generality
			$\WM(E)=1$, because we can
		replace $U_1$, $(R_m)_{m\in \NZ}$, $(\ubar{R}_m)_{m\in \NZ}$
		and $T$
		by the $\Lambda$-stopping times $(U_1)_E$, $((R_m)_E)_{m\in \NZ}$,
		$((\ubar{R}_m)_E)_{m\in \NZ}$ and $T_E$.
		
		\textbf{There exists a subsequence of $(\ubar{R}_m)_{m\in \NZ}$
		with the desired conditions:} 
		The values $\ubar{Q}_m$ are decreasing in $m\in \NZ$ to zero and
		therefore
		we get by the monotone convergence theorem that
		$\int_{(T,\ubar{R}_m)} \left(g_t(\ell)-g_t\left(\ell'\right)\right)
						\mu(\md t)$ converges 
		in $\mathrm{L}^1(\WM)$ to zero. By possibly passing to
		a subsequence we can assume that also the sequence
		$(\int_{(T,\ubar{R}_m)} \left(g_t(\ell)-g_t\left(\ell'\right)\right)
						\mu(\md t))_{m\in \NZ}$ converges to zero almost surely.
		As $g_t(\ell)-g_t(\ell')>0$ we also get that 
		\begin{align}\label{sto:eq_26}
			\lim_{m\rightarrow \infty} \mu([U_1,\ubar{R}_m))=0\quad 
			\text{a.s..}
		\end{align}	
		
		Next we get for $m,p\in \NZ$ and $1\leq p\leq m$
		\begin{align}\label{eq_269}
			\left\{Q_p=\ubar{Q}_{m}\right\}
			\cap \bigcap_{r=1}^{p-1}
			\left\{Q_r>\ubar{Q}_{m}\right\}
			 \supset \left\{Q_p=\ubar{Q}_{m+1}\right\}
			\cap \bigcap_{r=1}^{p-1}
			\left\{Q_r>\ubar{Q}_{m+1}\right\}.
		\end{align}
		Furthermore we want to remind that $(R_m)_{m\in \NZ}$ 
		satisfies by \eqref{sto:eq_25}
		\begin{align}\label{eq_231}
			\EW\left[X_{R_m}+\int_{(T,R_m)}g_t(\ell')\mu(\md t)\midG
			\aFA_{U_1}\right]
			\leq \EW\left[X_{R_{n}}+\int_{(T,R_{n})}g_t(\ell')\mu(\md t)\midG
			\aFA_{U_1}\right]
		\end{align}		
		for any $m,n\in \NZ$ with $m\leq n$. Hence combing \eqref{eq_269} and \eqref{eq_231}   
		gives us
		\begin{align}\label{eq_208}
			\EW\left[X_{\ubar{R}_m}+\int_{(T,\ubar{R}_m)}g_t(\ell')\mu(\md t)\midG
			\aFA_{U_1}\right]
			\leq \EW\left[X_{\ubar{R}_{m+1}}+\int_{(T,\ubar{R}_{m+1})}
			g_t(\ell')\mu(\md t)\midG
			\aFA_{U_1}\right].
		\end{align}				
		We will show next the following result:
		
		\textbf{Claim 5:} 
		For fixed $m\in \NZ$ we have  on $\Gamma_2$ that
		\[
			\EW\left[X_{R_m}+\int_{(T,R_m)}g_t(\ell')\mu(\md t)\midG
			\aFA_{U_1}\right]\nonumber
			\leq \sup_{p\in \NZ}
			\EW\left[X_{\ubar{R}_p}+\int_{(T,\ubar{R}_p)}
			g_t(\ell')\mu(\md t)\midG
			\aFA_{U_1}\right]
		\]
		
		\textbf{Proof of Claim 5:}		
			As we have shown $Q_m>0$ for all $m\in \NZ$ on $\Gamma_2$ we
			also have $\ubar{Q}_m>0$ for all $m\in \NZ$ on $\Gamma_2$.
		On the other hand we know that $\lim_{m\rightarrow \infty}
		\ubar{Q}_m=0$ on $\Gamma_2$. Fix now $m\in \NZ$. Then we have
		\begin{align}\label{sto:eq_341}
			\Gamma_2\subset \bigcup_{p=m+1}^\infty
			\left\{Q_p=\ubar{Q}_p<\ubar{Q}_m\right\}.
		\end{align}
		and actually we can rewrite \eqref{sto:eq_341} as a disjoint union of 
		sets by restricting the right-hand-side to the first 
		time point being strictly smaller than $\ubar{Q}_m$, i.e.
		\begin{align*}
			\Gamma_2\subset \bigcup_{p=m+1}^\infty\left(
			\left\{Q_p=\ubar{Q}_p<\ubar{Q}_m\right\}
			\cap \bigcap_{s=m+1}^{p-1} 
			\left\{\ubar{Q}_s=\ubar{Q}_m\right\}\right)
			=:\bigcup_{p=m+1}^\infty H^{(p)}.
		\end{align*}
		But on $H^{(p)}$ we have $\ubar{R}_p=R_p$ and hence by
		\eqref{eq_231}
		\begin{align*}
			\EW\left[X_{R_m}+\int_{(T,R_m)}g_t(\ell')\mu(\md t)\midG
			\aFA_{U_1}\right]
			\leq \EW\left[X_{\ubar{R}_p}+\int_{(T,\ubar{R}_p)}
			g_t(\ell')\mu(\md t)\midG
			\aFA_{U_1}\right] \text{ on } H^{(p)},
		\end{align*}
		which finishes the proof of Claim 5.
		
		\textbf{Continuation of the Proof of Claim 4:}
		Now we have by Claim $3$ and $5$ on $\Gamma_2$
		\begin{align*}
			Y_{U_1}^{\ell'}&=
			\lim_{m\rightarrow \infty}
			\EW\left[X_{R_m}+\int_{(T,R_m)}g_t(\ell')\mu(\md t)\midG
			\aFA_{U_1}\right]\nonumber\\
			&\leq \sup_{p\in \NZ}
			\EW\left[X_{\ubar{R}_p}+\int_{(T,\ubar{R}_p)}
			g_t(\ell')\mu(\md t)\midG
			\aFA_{U_1}\right]\nonumber\\
			&=\lim_{p\rightarrow \infty}
			\EW\left[X_{\ubar{R}_p}+\int_{(T,\ubar{R}_p)}
			g_t(\ell')\mu(\md t)\midG
			\aFA_{U_1}\right]
			\leq Y_{U_1}^{\ell'}, 
		\end{align*}
		where we have used that the supremum is equal to the limit
		by  \eqref{eq_208}.		In particular we have on $\Gamma_2$ that
		\begin{align}\label{eq_210}
			Y_{U_1}^{\ell'}=\lim_{p\rightarrow \infty}
			\EW\left[X_{\ubar{R}_p}+\int_{(T,\ubar{R}_p)}
			g_t(\ell')\mu(\md t)\midG
			\aFA_{U_1}\right].		
		\end{align}
		Furthermore we have by \eqref{sto:eq_26}
		\begin{align*}
			\lim_{p\rightarrow \infty}
			\EW\left[\int_{(T,\ubar{R}_p)}
			g_t(\ell')\mu(\md t)\midG
			\aFA_{U_1}\right]=0	\quad \text{ on } \quad \Gamma_2	
		\end{align*}
		 that
		\begin{align}\label{sto:eq_28}
			\lim_{p\rightarrow \infty}
			\EW\left[X_{\ubar{R}_p}\midG
			\aFA_{U_1}\right]=Y_U^{\ell_1}
			\quad \text{ on } \quad \Gamma_2.		
		\end{align}
		
		\textbf{Combining the previous results:}
		We have on $\Gamma_2$
		\begin{align*}
					X_{U_1}&<Y_{U_1}^{\ell'}
			= \lim_{p\rightarrow \infty}
			\EW\left[X_{\ubar{R}_p}\midG \aFA_{U_1}\right]
			\leq X_{U_1},
		\end{align*}
		which is a contradiction. Here we have used in the first inequality
		\eqref{sto:sol_gl_2}, in the first
		equality \eqref{sto:eq_344} and in the second inequality
		$\mu$-right-upper-semicontinuity
		 in expectation of $X$ and Proposition
		 \ref{eq_rusc}.

		
	    \subsection{Proof of Proposition \ref{sto:dis_pro}}\label{proof:4.5}
			We follow the idea of the proof in \citing{BK04}{Lemma 4.12 (iv)}{1050}.
				Fix $S\in \stm$. It suffices to consider, for $\ubar{\ell}
				<\bar{\ell}$, $\phi=\mathbb{1}_{[\ubar{\ell},\bar{\ell})}$ in 
				(\ref{sto:dis_gl_1}) and prove accordingly
				\begin{align}\label{sto:eq_37}
					Y^{\bar{\ell}}_S-Y^{\ubar{\ell}}_S=
					\EW\left[\int_{[S,\infty)}\left\{\int_{[\ubar{\ell},\bar{\ell})} 
									\mathbb{1}_{[S,\tau_{S,\ell})}(t)\
									g_t(\md \ell)\right\}
									\mu(\md t)\midG \aFA_S\right].
				\end{align}
				To this end, fix a set $A\in \aFA_S$ and consider a rational partition
				$\pi_n=\{\ubar{\ell}=\ell_0<\ell_1<\dots<\ell_{n+1}=\bar{\ell}\}$ of 
				the interval $[\ubar{\ell},\bar{\ell}]$.
				First we get
				\begin{align}\label{sto:gl_75}
					\EW\left[\left(Y^{\bar{\ell}}_S-Y^{\ubar{\ell}}_S\right)
					\mathbb{1}_A\right]	=\sum\limits_{i=0}^n
					\EW\left[\left(Y^{\ell_{i+1}}_S-Y^{\ell_{i}}_S\right)\mathbb{1}_A\right].
				\end{align}
				Now we have by Lemma \ref{Lem:sp3} (vi), that
				\begin{align}
					Y^{\ell_i}_S
					=\EW\left[X_{\tau_{S,\ell_i}}
							+\int_{[S,\tau_{S,\ell_i})} g_t({\ell_i})\mu(\md t) \midG \aFA_S\right]
				\end{align}			
				and by Lemma \ref{Lem:sp3} (vii) we get
				\begin{align*}
					Y^{\ell_{i+1}}_S
					\geq\EW\left[X_{\tau_{S,\ell_i}}
							+\int_{[S,\tau_{S,\ell_i})} g_t({\ell_{i+1}})\mu(\md t) \midG \aFA_S\right].
				\end{align*}
				This implies with (\ref{sto:gl_75})
				\begin{align}\label{sto:eq_54}
					\EW\left[\left(Y^{\bar{\ell}}_S
					-Y^{\ubar{\ell}}_S\right)\mathbb{1}_A\right]
				\geq 	\sum\limits_{i=0}^n\EW\left[\int_{[S,\tau_{S,\ell_i})}
				(g_t(\ell_{i+1})-g_t(\ell_i))\mu(\md t)
					\ \mathbb{1}_A\right]=:\RM{1}^{\pi_n}_1.
				\end{align}
				Similarly we obtain 
				\begin{align}\label{sto:eq_53}
					\EW\left[\left(Y^{\bar{\ell}}_S-Y^{\ubar{\ell}}_S\right)\mathbb{1}_A\right] 
					\leq \sum\limits_{i=0}^n\EW\left[\int_{[S,\tau_{S,\ell_{i+1}})}
					(g_t(\ell_{i+1})-g_t(\ell_i))\mu(\md t)
					\ \mathbb{1}_A\right]=:\RM{1}^{\pi_n}_2
				\end{align}
				and we see $\RM{1}^{\pi_n}_2\geq \RM{1}^{\pi_n}_1$. 
				Now we consider a refining sequence of such partitions $\pi_n$ with mesh
				$\left\|\pi_n\right\|\rightarrow 0$ as $n\rightarrow \infty$ and we will prove
				\begin{align}\label{sto:eq_55}
					\limsup_{n\rightarrow \infty} \RM{1}^{\pi_n}_2
					\leq \EW\left[\int_{[S,\infty)} 
								\left\{\int_{[\ubar{\ell},\bar{\ell})}  \mathbb{1}_{[S,\tau_{S,\ell})}(t)
								g_t(\md \ell)\right\}\mu(\md t)\mathbb{1}_A\right]
								\leq \liminf_{n\rightarrow \infty} \RM{1}^{\pi_n}_1,
				\end{align}	
				which proves \eqref{sto:eq_37} as $A\in \aFA_S$ was arbitrary.
				
				First we will consider $\RM{1}^{\pi_n}_1$ from \eqref{sto:eq_54}.
				By
				\begin{align*}
					g_t(\ell_{i+1})-g_t(\ell_{i})
					&= \int_{\RZ} \mathbb{1}_{[\ell_i,\ell_{i+1})}(\ell)  g_t(\md \ell)
				\end{align*}
				we may rewrite $\RM{1}^{\pi_n}_1$  as
				\begin{align*}
					\RM{1}^{\pi_n}_1
						=\EW\left[\int_{[S,\infty)} \left\{\int_{[\ubar{\ell},\bar{\ell})} 
									\mathbb{1}_{[S,\tau_{S,\ell_n(\ell)})}(t)		
									  g_t(\md \ell)\right\}\mu(\md t)\mathbb{1}_A\right]			 
				\end{align*}
				with
				\begin{align}\label{sto:eq_331}
					\ell_n(\ell):=\max\left\{\ell_i\in \pi_n\midG \ell_i\leq \ell\right\}
					\quad (n=1,2,\dots).
				\end{align}	
				
				The liminf estimate in \eqref{sto:eq_55} follows thus from 
				Fatou's Lemma and the following claim,
				which we will prove at the end:
				 				
				\textbf{Claim 1:} For $\WM$-a.e. $\omega\in \Omega$ we have
					\begin{align}\label{sto:dis_gl_2}
						&\int_{[S(\omega),\infty)}\int_{[\ubar{\ell},\bar{\ell})}
					\liminf_{n\rightarrow \infty} \mathbb{1}_{[S,\tau_{S,\ell_n(\ell)})}(\omega,t)
					g_t(\omega,\md \ell)	\mu(\omega,\md t)\nonumber\\
									\geq&\ \int_{[S(\omega),\infty)}
									\int_{[\ubar{\ell},\bar{\ell})} 
									\mathbb{1}_{[S,\tau_{S,\ell})}(\omega,t) 
								g_t(\omega,\md \ell)	\mu(\omega,\md t),
					\end{align}
					where $(\ell_n(\ell))_{n\in \NZ}$ are defined in \eqref{sto:eq_331}. 
				
				For the limsup estimate, we similarly write
				$\RM{1}^{\pi_n}_2$ from \eqref{sto:eq_53} as
				\begin{align*}
					\RM{1}^{\pi_n}_2
						&=\EW\left[\int_{[S,\infty)} \left\{\int_{[\ubar{\ell},\bar{\ell})} 
									\mathbb{1}_{[S,\tau_{S,r_n(\ell)})}(t)		
									  g_t(\md \ell)\right\}\mu(\md t)\mathbb{1}_A\right]				 
				\end{align*}
				with
				\begin{align}\label{sto:eq_332}
					r_n(\ell):=\min\left\{\ell_i\in \pi_n\midG \ell_i> \ell\right\}.
				\end{align}
				
				Again we use Fatou's Lemma 
				to estimate $\limsup_{n\rightarrow \infty} \RM{1}^{\pi_n}_2$ in \eqref{sto:eq_55}
				from above.	Here we are allowed to use Fatou's Lemma because the 
				latter integrand is bounded
				by $\mathbb{1}_{[S,\infty)}\mathbb{1}_{[\ubar{\ell},\bar{\ell})}$,
				 which is $\WM\otimes \mu \otimes \md g$ - integrable
				 by Assumption \ref{sto:frame_ass}.									
				With the help of the following claim, which we will prove at the end, 
				this proves the first inequality in \eqref{sto:eq_55}.

				\textbf{Claim 2:} For $\WM$-a.e. $\omega\in \Omega$ we have 
				\begin{align}\label{sto:dis_gl_4}					
					&\int_{[S(\omega),\infty)}\int_{[\ubar{\ell},\bar{\ell})}
					\limsup_{n\rightarrow \infty} \mathbb{1}_{[S,\tau_{S,r_n(\ell)})}(\omega,t)
					g_t(\omega,\md \ell)	\mu(\omega,\md t)\nonumber\\
									\leq&\ \int_{[S(\omega),\infty)}
									\int_{[\ubar{\ell},\bar{\ell})} 
									\mathbb{1}_{[S,\tau_{S,\ell})}(\omega,t) 
								g_t(\omega,\md \ell)	\mu(\omega,\md t),
				\end{align}
				where $(r_n(\ell))_{n\in \NZ}$ is defined in \eqref{sto:eq_332}.
						
				 It remains to prove Claim 1 and 2.
				
				\textbf{Proof of Claim 1:} For reasons that will become clear later, we will
				establish \eqref{sto:dis_gl_2} only for $\omega\in \tilde{\Omega}$, 
				where $\tilde{\Omega}\subset \Omega$
				 with $\WM(\tilde{\Omega})=1$ such that	for all $q\in \mathbb{Q}$
				\begin{align}\label{sto:gl_64_2}
					H_{S,q}^+\cap \tilde{\Omega} = 
					\left\{X_{T_{S,q}}<Y_{T_{S,q}}^q\right\}\cap \tilde{\Omega}.
				\end{align}
				Notice that such an $\tilde{\Omega}$ can be found by \eqref{sto:eq_12}.
				
				Now fix $\omega\in \tilde{\Omega}$. 
				For $(t,\ell)\in [0,\infty)\times [\ubar{\ell},\bar{\ell})$ with 
				$t\neq T_{S,\ell}(\omega)$
				and $T_{S,\ell-}(\omega)=T_{S,\ell}(\omega)$ we have by 
				$\lim_{n\rightarrow \infty} \ell_n(\ell)=\ell$ 
				that $\lim_{n\rightarrow \infty} T_{S,\ell_n(\ell)}(\omega)=T_{S,\ell}(\omega)$ 
				and
				\[
					\liminf_{n\rightarrow \infty} \mathbb{1}_{[S,\tau_{S,\ell_n(\ell)})}(\omega,t)
					= \mathbb{1}_{[S,\tau_{S,\ell})}(\omega,t).
				\]			
				As for fixed $\omega \in \tilde{\Omega}$ the set 
				$\{\ell\in \RZ\, |\, T_{S,\ell-}(\omega)<T_{S,\ell}(\omega)\}$ is countable
				it is for every $t\in [0,\infty)$ a $g_t(\omega,\md \ell)$-null set. Hence we get
				\begin{align}\label{sto:eq_333}
					&\int_{[S(\omega),\infty)}\int_{[\ubar{\ell},\bar{\ell})}
					\liminf_{n\rightarrow \infty} \mathbb{1}_{[S,\tau_{S,\ell_n(\ell)})}(\omega,t)
					g_t(\omega,\md \ell)	\mu(\omega,\md t)\nonumber\\
									\geq &\int_{[S(\omega),\infty)}\int_{[\ubar{\ell},\bar{\ell})}
						\left(\mathbb{1}_{[S,\tau_{S,\ell})}(\omega,t) 
						\mathbb{1}_{\{T_{S,\ell}(\omega)\neq t\}}\right)
					g_t(\omega,\md \ell)	\mu(\omega,\md t)\nonumber\\
					&\quad +\int_{[S(\omega),\infty)}\int_{[\ubar{\ell},\bar{\ell})}
						\left(\liminf_{n\rightarrow \infty} 
						\mathbb{1}_{[S,\tau_{S,\ell_n(\ell)})}(\omega,t)
						\mathbb{1}_{\{T_{S,\ell}(\omega)=t\}}
					\right)
					g_t(\omega,\md \ell)	\mu(\omega,\md t).
				\end{align}					
				Therefore it remains to show for 
				any fixed $t\in [S(\omega),\infty)$ 
				that
				\begin{align}\label{sto:eq_335}					
					\liminf_{n\rightarrow \infty} \mathbb{1}_{[S,\tau_{S,\ell_n(\ell)})}(\omega,t)
							\geq		\mathbb{1}_{[S,\tau_{S,\ell})}(\omega,t)
				\end{align}
				for $ g_t(\omega,\md \ell)$-a.e. $\ell\in J$ with
				\begin{align*}
					J&:=\left\{\ell\in [\ubar{\ell},\bar{\ell})\midG T_{S,\ell}(\omega)=t\right\}.
				\end{align*}
				As $\ell\mapsto T_{S,\ell}(\omega)$ is non-decreasing,
				$J$ is an interval and since $g_t(\omega,\md \ell)$
				is an atomless measure, we can focus without loss of 
				generality on the interior of $J$ and we assume that 
				$\mathrm{int}J$ is non-empty. Fix $\ell\in \mathrm{int}J$.
				Now there exists some $N_{\omega,\ell}\in \NZ$
				such that $\ell_n(\ell)\in \mathrm{int}J$ for
				$n\geq N_{\omega,\ell}$ and thus
				$T_{S,\ell}(\omega)=T_{S,\ell-}(\omega)=T_{S,\ell+}(\omega)=T_{S,\ell_n(\ell)}(\omega)$ 
				for $n\geq N_{\omega,\ell}$. 
				This implies that for any fixed $\ell$, inequality \eqref{sto:eq_335}
				is equivalent to
				\begin{align}\label{sto:eq_336}					
					\liminf_{n\rightarrow \infty} \mathbb{1}_{H_{S,\ell_n(\ell)}^+}(\omega)
							\geq		\mathbb{1}_{H_{S,\ell}^+}(\omega).
				\end{align}
				Now we get by the property of $\tilde{\Omega}$ 
				and $(\ell_n(\ell))_{n\in \NZ}\subset \mathbb{Q}$
				that
				\begin{align*}
					\liminf_{n\rightarrow \infty}
							\mathbb{1}_{H_{S,\ell_n(\ell)}^+}(\omega)&=
					\liminf_{n\rightarrow \infty}
					\mathbb{1}_{\{
								Y_{T_{S,\ell_n(\ell)}}^{\ell_n(\ell)}>X_{T_{S,\ell_n(\ell)}}\}}(\omega)\\
							&=
									\liminf_{n\rightarrow \infty} \mathbb{1}_{\{
								Y_{T_{S,\ell}}^{\ell_n(\ell)}>X_{T_{S,\ell}}\}}(\omega)\\
							&=\mathbb{1}_{\bigcup_{m=1}^\infty \bigcap_{n=m}^\infty
								\{
								Y_{T_{S,\ell}}^{\ell_n(\ell)}>X_{T_{S,\ell}}\}}(\omega).
				\end{align*}						
				Moreover we have for $\omega\in 
				\{Y_{T_{S,\ell}}^{\ell}>X_{T_{S,\ell}}\}$ that there exists
				by continuity of $\tilde{\ell}\mapsto 
				Y^{\tilde{\ell}}_{T_{S,\ell}}$ some
				$\bar{N}(\omega)\geq N_{\omega,\ell}$ such that
				\[
					Y_{T_{S,\ell}(\omega)}^{\ell_m(\ell)}(\omega)>X_{T_{S,\ell}(\omega)}(\omega)
				\]
				for $m\geq \bar{N}(\omega)$ and so $\omega\in 
				\bigcap_{m=\bar{N}(\omega)}^\infty \{
								Y_{T_{S,\ell}}^{\ell_m(\ell)}>X_{T_{S,\ell}}\}$. 
				As $\mathbb{1}_{H_{S,{\ell}}^+}(\omega)
				\leq 
				\mathbb{1}_{\{X_{T_{S,\ell}}<Y_{T_{S,\ell}}^{\ell}\}}(\omega)$
				by definition of $H_{S,{\ell}}^+$ we finally obtain
				\begin{align*}
					\liminf_{n\rightarrow \infty}
							\mathbb{1}_{H_{S,\ell_n(\ell)}^+}(\omega)
							=\mathbb{1}_{\bigcup_{n=1}^\infty \bigcap_{m=n}^\infty
								\{
								Y_{T_{S,\ell}}^{\ell_m(\ell)}>X_{T_{S,\ell}}\}}(\omega)
							\geq		\mathbb{1}_{\{X_{T_{S,\ell}}<Y_{T_{S,\ell}}^{\ell}\}}(\omega)
							\geq \mathbb{1}_{H_{S,l}^+}(\omega),
				\end{align*}			
				 which shows \eqref{sto:eq_336}	and
				 finishes the proof of Claim 1.
			
				\textbf{Proof of Claim 2:}
				Let $\tilde{\Omega}\subset \Omega$ with $\WM(\tilde{\Omega})=1$ be
				such that the relation in \eqref{sto:eq_12} holds for any 
				$\omega\in \tilde{\Omega}$ and all $\ell\in 
				\mathbb{Q}$. Analogously to the proof of Claim 1 we get
				\begin{align}\label{sto:gl_97}
					&\int_{[S(\omega),\infty)}\int_{[\ubar{\ell},\bar{\ell})}
					\limsup_{n\rightarrow \infty} \mathbb{1}_{[S,\tau_{S,r_n(\ell)})}(\omega,t)
					g_t(\omega,\md \ell)	\mu(\omega,\md t)\nonumber\\
									\leq &\int_{[S(\omega),\infty)}\int_{[\ubar{\ell},\bar{\ell})}
						\left(\mathbb{1}_{[S,\tau_{S,\ell})}(\omega,t) 
						\mathbb{1}_{\{T_{S,\ell}(\omega)\neq t\}}\right)
					g_t(\omega,\md \ell)	\mu(\omega,\md t)\nonumber\\
					&\quad +\int_{[S(\omega),\infty)}\int_{[\ubar{\ell},\bar{\ell})}
						\left(\limsup_{n\rightarrow \infty} \mathbb{1}_{[S,\tau_{S,r_n(\ell)})}(\omega,t)
						\mathbb{1}_{\{T_{S,\ell}(\omega)=t\}}
					\right)
					g_t(\omega,\md \ell)	\mu(\omega,\md t).
				\end{align}					
				Therefore it remains to show for 
				any $t\in [S(\omega),\infty)$ and $\omega\in \tilde{\Omega}$
				that
				\begin{align}\label{sto:gl_96}					
					\limsup_{n\rightarrow \infty} \mathbb{1}_{[S,\tau_{S,r_n(\ell)})}(\omega,t)
							\leq		\mathbb{1}_{[S,\tau_{S,\ell})}(\omega,t)
				\end{align}
				$ g_t(\omega,\md \ell)$-a.e. on
				\begin{align*}
					J&:=\left\{\ell\in [\ubar{\ell},\bar{\ell})\midG T_{S,\ell}(\omega)=t\right\}.
				\end{align*}
				So fix $t\in [S(\omega),\infty)$. 
				As $\ell\mapsto T_{S,\ell}(\omega)$ is non-decreasing,
				$J$ is an interval and since $g_t(\omega,\md \ell)$
				is an atomless measure, we can focus without loss of 
				generality on the interior of $J$ and we can assume that 
				$\mathrm{int}J$ is non-empty. Now
				we get for $\ell\in \mathrm{int}J$ that also $r_n(\ell)\in \mathrm{int}J$ for
				$n$ large enough and thus
				$T_{S,\ell}(\omega)=T_{S,\ell-}(\omega)=T_{S,\ell+}(\omega)=T_{S,r_n(\ell)}(\omega)=t$ 
				for sufficiently large $n$. This implies, analalogously to the
				proof of Claim 1, that for $\ell\in \mathrm{int}J$, inequality \eqref{sto:gl_96}
				is equivalent to
				\begin{align}\label{sto:eq_84}
					\limsup_{n\rightarrow \infty} 
									\mathbb{1}_{\{X_{T_{S,\ell}}<Y_{T_{S,\ell}}^{r_n(\ell)}\}}(\omega)
									\leq 
									\mathbb{1}_{H_{S,\ell}^+}(\omega).
				\end{align}				
				We claim that $H_{S,\ell}^+={\{X_{T_{S,\ell}}<Y^{\ell}_{T_{S,\ell}}\}}$. 
				Indeed, as
				 $H_{S,\ell}^+\subset \{X_{T_{S,\ell}}<Y^{\ell}_{T_{S,\ell}}\}$
				  is clear we can assume 			
				$\omega \in \{X_{T_{S,\ell}}<Y^{\ell}_{T_{S,\ell}}\}$ and we will show 
				$\omega \in H_{S,\ell}^+$.
				By continuity and monotonicity of $\ell\mapsto Y^\ell_t(\omega)$ and 
				$\ell \in \mathrm{int}J$ there
				exists $q\in \mathrm{int}J\cap \mathbb{Q}$ with $q<\ell$ and 
				$\omega\in \{X_{T_{S,q}}<Y^{q}_{T_{S,q}}\}$.
				As $\omega \in \tilde{\Omega}$ this implies $\omega \in H_{S,q}^+$.
				By Lemma \ref{Lem:sp3} (vi) the mapping $\ell\mapsto \tau_{S,\ell}$
				 is increasing. In particular
				we get by $T_{S,\ell}(\omega)=T_{S,q}(\omega)=t$ and
				$\omega\in H_{S,q}^+$ that $\omega \in H_{S,\ell}^+$.
							
				Next we see that the inequality \eqref{sto:eq_84} is trivially fulfilled
				 if the left-hand side is zero or if $\ell=r_n(\ell)$ for sufficiently large
				 $n$, 
				we just have to analyse $\ell\in J_2$, where
				\begin{align*}
					J_2:=\left\{\ell\in \mathrm{int} J\midG \right.
					\omega \in &\{X_{T_{S,\ell}}<Y_{T_{S,\ell}}^{r_n(\ell)}\} 
					\text{ for infinitely many $n$ and}\\
											&\text{$(r_n(\ell))_{n\in \NZ}$ converges strictly
											from above to $\ell$} \left.\right\}.
				\end{align*}
				So let us show $\omega \in \{X_{T_{S,\ell}}<Y^{\ell}_{T_{S,\ell}}\}$ 
				for $\ell\in J_2$. For that 
				we will use the following, which we will prove at the end:
				
				\textbf{Claim 3:} There
				exists at most one $\tilde{\ell}\in J_2$ with		
				\begin{align}\label{sto:eq_82}
					Y_{T_{S,\tilde{\ell}}(\omega)}^{\tilde{\ell}}(\omega)
					<Y_{T_{S,\tilde{\ell}}(\omega)}^{r_{n}(\tilde{\ell})}(\omega)
					 \text{\quad  for all } n\in \NZ
				\end{align}
				and
				\begin{align}\label{sto:eq_83}
					X_{T_{S,\tilde{\ell}}(\omega)}(\omega)
					=Y^{\tilde{\ell}}_{T_{S,\tilde{\ell}}(\omega)}(\omega).
				\end{align}

				As $g_{t}(\omega,\md \ell)$ is a continuous measure we can now focus 
				 by Claim 3 on
				$\ell\in J_2\backslash \{\tilde{\ell}\}$, which do not satisfy \eqref{sto:eq_82}
				 and \eqref{sto:eq_83} at once. If $\ell$
				does not satisfy \eqref{sto:eq_83} we have  
				\begin{align*}
					X_{T_{S,\ell}(\omega)}(\omega)<Y^{\ell}_{T_{S,\ell}(\omega)}(\omega),
				\end{align*}		
				which is exactly what we want to show. Assume $\ell$ does 
				not satisfy \eqref{sto:eq_82}, i.e.
				$Y_{T_{S,\ell}(\omega)}^\ell(\omega)=Y_{T_{S,\ell}(\omega)}^{r_n(\ell)}
				(\omega)$ for $n\in \NZ$ large enough.
				Then there will be $\tilde{n}\geq n$ with $\omega \in
				 \{X_{T_{S,\ell}}<Y_{T_{S,\ell}}^{r_{\tilde{n}}(\ell)}\}$
				and by monotonicity of $r\mapsto Y^r$ we have again
				$Y_{T_{S,\ell}(\omega)}^\ell(\omega)
				=Y_{T_{S,\ell}(\omega)}^{r_{\tilde{n}}(\ell)}(\omega)$.
				This leads to 
				\[
					Y_{T_{S,\ell}(\omega)}^{\ell}(\omega)
					=Y_{T_{S,\ell}(\omega)}^{r_{\tilde{n}}(\ell)}(\omega)
					>X_{T_{S,\ell}(\omega)}(\omega)
				\]
				and hence $\omega \in \{X_{T_{S,\ell}}<Y_{T_{S,\ell}}^{\ell}\}$,
				which proves Claim 2 once Claim 3 is established.
				
				\textbf{Proof of Claim 3:} Assume $\tilde{\ell}$ fulfills 
				\eqref{sto:eq_82} and \eqref{sto:eq_83} and  $u\in J_2$.

				\emph{Case $u> \tilde{\ell}$:} As $r\mapsto Y^r_t(\omega)$ is
				non-decreasing we get by $\tilde{\ell}$ satisfying \eqref{sto:eq_82}
				and \eqref{sto:eq_83} that
				\[
					Y^{u}_{T_{S,u}(\omega)}(\omega)
					=Y^{u}_{T_{S,\tilde{\ell}}(\omega)}(\omega)> 
					Y^{\tilde{\ell}}_{T_{S,\tilde{l}}(\omega)}(\omega)= 
					X_{T_{S,\tilde{\ell}}(\omega)}(\omega)=X_{T_{S,u}(\omega)}(\omega),
				\]
				where we have used  $T_{S,\tilde{\ell}}(\omega)=t=T_{S,u}(\omega)$ 
				by $\tilde{\ell},u\in J_2$.
				Hence $u$ does not fulfill \eqref{sto:eq_83}.
				
				\emph{Case $u< \tilde{\ell}$:} Again as $r\mapsto Y^r_t(\omega)$ is
				non-decreasing we get  that   
				\[
					X_{T_{S,\tilde{\ell}}(\omega)}(\omega)
					=Y^{\tilde{\ell}}_{T_{S,\tilde{\ell}}(\omega)}(\omega)
					=Y^u_{T_{S,\tilde{\ell}}(\omega)}(\omega).  
				\]
				Furthermore we see by $u<\tilde{\ell}$ that the corresponding sequence
				 $(\ell_n(u))_{n\in \NZ}$ will fulfill  
				$\ell_n(u)\leq {\tilde{\ell}}$ for $n$ large enough and therefore 
				\[
					Y^{\ell_n(u)}_{T_{S,\tilde{\ell}}(\omega)}(\omega)=
						Y^{{\tilde{\ell}}}_{T_{S,\tilde{\ell}}(\omega)}(\omega)
						=Y^{u}_{T_{S,\tilde{\ell}}(\omega)}(\omega).
				\]		 
				Therefore $u$ does not satisfy \eqref{sto:eq_82}, which  proves our claim.
		
		
			\subsection{Proof of Lemma \ref{sto:Lem_12}}\label{proof:4.6}
    		 First we have $X_S=Y_S^{L_S}$ by Lemma \ref{Lem:s6}.			
    		 Fix now $\ell_0\in\mathbb{\RZ}$. As we 
    		 have $X_S=Y_S^{\ell_0}$ on $\{\ell_0\leq L_S\}$ and as $\ell\mapsto Y_S^\ell$ is 
    		 non-decreasing
    		(Lemma \ref{Lem:sp3} (viii)), we get
    		\begin{align}\label{sto:eq_41}
    			X_S=Y^{L_S}_S=Y^{\ell_0}_S-\int_{\RZ} 
    								\mathbb{1}_{[L_S\wedge \ell_0,\ell_0)}(\ell)  Y_S(\md \ell).
    		\end{align}
    		Denote by $I$ the integral on the right hand side of this expression.
    		Due to our disintegration
    		formula (see Proposition \ref{sto:dis_pro})
    		for the random measure $Y_S(\md \ell)$, we can rewrite
    		\begin{align}\label{sto:eq_17}
    			I=\EW\left[\int_{[S,\infty)} \left\{\int_{\RZ} 
    						\mathbb{1}_{[L_S\wedge \ell_0,\ell_0)}(\ell)
    						\mathbb{1}_{[S,\tau_{S,\ell})}(t) g_t(\md \ell)\right\}
    										\mu(\md t)\midG \aFA_{S}\right].
    		\end{align}
    		
    		Next we state a claim, which uses the notation $\bar{L}_{S,t}$
    		from \eqref{sto:eq_2}. The claim will be proven at the end.
    
    		\textbf{Claim:}\label{sto:claim} 
    		Let $\tilde\Omega:=\bar{\Omega}\backslash \mathcal{N}$, with
    		$\mathcal{N}$ from Lemma \ref{Lem:s5}
    		and $\bar{\Omega}\subset\Omega$, $\WM(\bar{\Omega})=1$ such that on
    		$\bar{\Omega}$ relation \eqref{sto:eq_12} holds for all 
    		$\ell\in \mathbb{Q}$ and $\ell_0$. Then
    		we have the following three equations: 			
    		\begin{compactenum}[(a)]
    			\item  For $\omega\in \tilde\Omega$,
    			$t\in [0,\infty)$ we have			
    					\begin{align}\label{sto:eq_611}
    					\mathbb{1}_{H_{S,\ell}^+}(\omega)
    					\mathbb{1}_{[L_{S}(\omega)\wedge \ell_0,\ell_0)}(\ell)
    					\mathbb{1}_{[S(\omega),T_{S,\ell}(\omega)]}(t)\nonumber\\
    					=
    					\mathbb{1}_{H_{S,\ell}^+}(\omega)
    				\mathbb{1}_{[\bar{L}_{S,t}(\omega),\ell_0)}(\ell)\mathbb{1}_{[S(\omega),T_{S,\ell_0}(\omega)]}(t)
    				\end{align}	
    				for $g_t(\omega,\md \ell)$-a.e. $\ell\in \RZ$.
    			\item  For $\omega\in \tilde\Omega$
    			and $t\in [0,\infty)$, we have
    			\begin{align}\label{sto:eq_8}
    				\mathbb{1}_{H_{S,\ell}^-\cup H_{S,\ell}}(\omega)
    				\mathbb{1}_{[L_S(\omega)\wedge \ell_0,\ell_0)}(\ell)\mathbb{1}_{[S(\omega),T_{S,\ell}(\omega))}(t)\nonumber\\
    					=\mathbb{1}_{H_{S,\ell}^-\cup H_{S,\ell}}(\omega)
    					\mathbb{1}_{[\bar{L}_{S,t}(\omega),\ell_0)}(\ell)\mathbb{1}_{[S(\omega),T_{S,\ell_0}(\omega))}(t)
    			\end{align}
    			for $g_t(\omega,\md \ell)$-a.e. $\ell\in \RZ$.
    			\item For $\omega\in\tilde\Omega$ 
    		and
    		$t=T_{S,\ell_0}(\omega)$, we have  
    		\begin{align}\label{eq_301}
    			\mathbb{1}_{[\bar{L}_{S,t},\ell_0)}(\ell)
    						\mathbb{1}_{H_{S,\ell}^+}(\omega)
    				=\mathbb{1}_{[\bar{L}_{S,t},\ell_0)}(\ell)
    						\mathbb{1}_{H_{S,\ell_0}^+}(\omega)
    		\end{align}
    		for $g_t(\omega,\md \ell)$-a.e. $\ell\in \RZ$.
    		\end{compactenum}
    
    		Combining (a) and (b) from the above Claim with \eqref{sto:eq_17} leads to
    		\begin{align*}
    			I&=\EW\left[\int_{[S,\infty)} \left\{\int_{\RZ} 
    						\left(\mathbb{1}_{[\bar{L}_{S,t},\ell_0)}(\ell)
    						\mathbb{1}_{[S,T_{S,\ell_0})}(t)\right.\right.\right.\\
    						&\left.\left.\left.\hspace{25ex} +\mathbb{1}_{[\bar{L}_{S,t},\ell_0)}(\ell)
    						\mathbb{1}_{H_{S,\ell}^+}\mathbb{1}_{\{T_{S,\ell_0}\}}(t)\right) g_t(\md \ell)\right\}
    										\mu(\md t)\midG \aFA_{S}\right]
    		\end{align*}
    		and by (c) of the above Claim
    		\begin{align*}
    			I&=\EW\left[\int_{[S,\tau_{S,\ell_0})} \left(g_t(\ell_0)-g_t(\bar{L}_{S,t})
    			\right)\mu(\md t)
    			 	\midG \aFA_S\right].
    		\end{align*}						
    		By \eqref{sto:eq_41} and Lemma \ref{Lem:sp3} (vi), we see that 
    		\begin{align*}
    			X_S= Y^{\ell_0}_S-I
    				= \EW\left[X_{\tau_{S,\ell_0}}
    									+\int_{[S,\tau_{S,\ell_0})} g_t(\bar{L}_{S,t})
    					\mu(\md t)
    						\midG \aFA_S\right],
    		\end{align*}
    		which is equivalent to
    		\begin{align}\label{sto:eq_51}
    			X_S	-\EW\left[X_{\tau_{S,\ell_0}}
    									\midG \aFA_S\right]
    						=\EW\left[\int_{[S,\tau_{S,\ell_0})} g_t(\bar{L}_{S,t})
    					\mu(\md t)
    						\midG \aFA_S\right].		
    		\end{align}
    		Now it stays to show the integrability of $\mathbb{1}_{[S,\tau_{S,\ell})}
    		g(\sup_{v\in [S,t]}L_v)$. First, we have for$\omega\in \Omega$ and $t<T_{S,\ell_0}(\omega)$
    		by Lemma \ref{Lem:s5} that $\bar{L}_{S,t}(\omega)<\ell_0$. On the other hand we
    		get for $\omega\in H_{S,l_0}^+$ and $t=T_{S,\ell_0}(\omega)$ that $X_t(\omega)<Y_t^{l_0}(\omega)$
    		which show also in this case $L_t(\omega)<l_0$. Hence we have for $t \in [S,\tau_{S,\ell_0})(\omega)$
    		that $L_t(\omega)<l_0$, which shows by monotonicity of $l\mapsto g_t(l)$,
    		$g_t(0)=0$ (see \eqref{sto:sep_gl}) and $g(l_0)\in \mathrm{L}^1$ 
    		by Assumption \ref{sto:frame_ass}  \ref{sto:ass_bullet_2} that
    		\[
    			\EW\left[\int_{[S,\tau_{S,\ell_0})} \left(g_t(\bar{L}_{S,t})\vee 0\right)
    					\mu(\md t)			\right]
    			\leq 
    			\EW\left[\int_{[0,\infty)} g_t(l_0\vee 0)
    					\mu(\md t)\right]<\infty,
    		\]
    		which shows that the positive part of $\mathbb{1}_{[S,\tau_{S,\ell})}	g(\sup_{v\in [S,t]}L_v)$	
    		is integrable. Furthermore we have by Lemma \ref{sto:tools_lem} that there exists a $\Lambda$-martingale $M^X$ 
    		of class($\text{D}^\Lambda$) with
    		$-M^X\leq X\leq M^X$ and by Lemma \citing{BB18_2}{Lemma 3.16}{13} we obtain with \eqref{sto:eq_51}
    		\begin{align*}
    			-\infty<\EW[X_S-M^X_S]\leq 
    						\EW\left[\int_{[S,\tau_{S,\ell_0})} g_t(\bar{L}_{S,t})
    					\mu(\md t) \right],	
    		\end{align*}
    		which shows that also the negative part of $\mathbb{1}_{[S,\tau_{S,\ell})}	g(\sup_{v\in [S,t]}L_v)$	
    		is integrable. This
    		completes the proof of our Lemma once we have proven the above Claim.
    				
    		\textbf{Proof of Part (a) of the above claim:} Fix $\omega\in \tilde\Omega$.
    		
    		``$\geq$'' in \eqref{sto:eq_611}:
    			 Assume $t\in [S(\omega),
    			T_{S,\ell_0}(\omega)]$, $\ell_0>\ell>
    			\bar{L}_{S,t}(\omega)\geq L_S(\omega)$ and $\omega\in H_{S,\ell}^+$.
    			Then we get by $A\subset B$ in 
    			 Lemma \ref{Lem:s5} (\eqref{sto:set_a} and \eqref{sto:set_b})
    			  that $t\leq T_{S,\ell}(\omega)$. We can focus on $\ell>\bar{L}_{S,t}(\omega)$
    			as $\{\bar{L}_{S,t}(\omega))\}$ is a $g_t(\omega,\md \ell)$-null set.

    		``$\leq$'' in \eqref{sto:eq_611}: 
    			Assume $t\in [S(\omega),
    			T_{S,\ell}(\omega)]$, $\ell_0>\ell\geq L_S(\omega)$ and $\omega\in H_{S,\ell}^+$.
    			 In the case $t<T_{S,\ell}(\omega)$
    			 we have by the relation $\tilde{B}\subset \tilde{C}$ in 
    			 Lemma \ref{Lem:s5} (\eqref{sto:set_tb} and \eqref{sto:set_tc}) that
    			  $\ell\geq \bar{L}_{S,t}(\omega)$. For
    			$t=T_{S,\ell}(\omega)$ we get by $B\subset C$ in 
    			Lemma \ref{Lem:s5} (\eqref{sto:set_b} and \eqref{sto:set_c})
    			that
    			 \[
    				\sup_{v\in [S(\omega),{T_{S,\ell}(\omega)})} L_v(\omega)\leq \ell
    			\]
    			and  $\omega \in H_{S,\ell}^+\subset\{X_{T_{S,\ell}}<Y_{T_{S,\ell}}^\ell\}$ (see \eqref{sto:eq_77})
    			shows by the definition of $L$ that $\bar{L}_{S,T_{S,\ell}(\omega)}(\omega)\leq \ell$. This finishes our proof as
    			 ${T_{S,\ell}(\omega)}\leq T_{S,\ell_0}(\omega)$ follows by monotonicity of $\ell\mapsto T_{S(\omega),\ell}(\omega)$. 			
    
    		\textbf{Proof of Part (b) of the above claim:}
    		Fix $\omega\in \tilde{\Omega}$. 
    		
    		``$\leq$'' in \eqref{sto:eq_8}: 
    		Assume $L_{S(\omega)}(\omega)\leq \ell<\ell_0$, $S(\omega)
    		\leq t <T_{S,\ell}(\omega)\leq T_{S,\ell_0}(\omega)$
    		 and $\omega\in H_{S,\ell}^-\cup H_{S,\ell}$.
    		 From $\tilde{B}\subset \tilde{C}$ in Lemma \ref{Lem:s5} 
    		(\eqref{sto:set_tb} and \eqref{sto:set_tc})
    		 we get $\bar{L}_{S,t}(\omega)\leq \ell<\ell_0$. 
    		
    		``$\geq$'' in \eqref{sto:eq_8}: Let $t\in 
    		[S(\omega),T_{S,\ell_0}(\omega))$, 
    		$\bar{L}_{S,t}(\omega)\leq \ell<\ell_0$  and 
    		$\omega\in H_{S,\ell}^-\cup H_{S,\ell}$. As $\{\bar{L}_{S,t}(\omega)\}$ is a 
    		$g_t(\omega,\md \ell)$-null set we can
    		focus on $\bar{L}_{S,t}(\omega)< \ell$. 
    		From $\bar{L}_{S,t}(\omega)< \ell$ we obtain by 
    		$A\subset B$ in Lemma \ref{Lem:s5} (\eqref{sto:set_a}
    		and \eqref{sto:set_b}) that $t\leq T_{S,\ell}(\omega)$.
    		Now we have to prove that for fixed 
    		$t\in [S(\omega),T_{S,\ell_0}(\omega))$ the set
    		\[
    			J:=(\bar{L}_{S,t}(\omega),l_0)\cap 
    			\left\{\ell\in \RZ\midG  
    			\omega \in H_{S,\ell}^-\cup H_{S,\ell} \text{ and } 
    			T_{S,\ell}(\omega)=t \right\}
    		\]
    		is a $g_t(\omega,\md \ell)$-null set. 
    		By Lemma  \ref{Lem:sp3} (vi) the mapping $\ell\mapsto 
    		\tau_{S,\ell}$ is increasing and therefore $J$ is an interval. 
    		Assume $J$ contains more than one point.
    		Then take $\ell_1,\ell_2\in J$ and some $q\in \mathbb{Q}$
    		with $\ell_1<q<\ell_2$. As $J$ is an interval also $q\in J$.
    		From 
    		$\omega\in \tilde{\Omega}$ we get 
    		$q\in H_{S,\ell}^-\cup H_{S,\ell}=\{X_{T_{S,q}}=Y_{T_{S,q}}^q\}$,
    		which implies by $\tilde{A}\subset \tilde{B}$ in Lemma 
    		\ref{Lem:s5} (\eqref{sto:set_ta}
    		and \eqref{sto:set_tb}) that $t<T_{S,q}(\omega)$,
    		which contradicts $q\in J$. Hence $J$ contains
    		at most one point, which shows $J$ is a 
    		$g_t(\omega,\md \ell)$-null set.
    		
    		\textbf{Proof of Part (c) of the above claim:}
    		Fix $\omega\in \tilde\Omega$, $t=T_{S,\ell_0}(\omega)$ and
    		$\ell\in [\bar{L}_{S,T_{S,\ell_0}}(\omega),\ell_0)$. We do not have
    		to consider the case 
    		$\ell=\bar{L}_{S,T_{S,\ell_0}}(\omega)$ as 
    		for fixed $\omega$ the set $\{\bar{L}_{S,T_{S,\ell_0}}(\omega)\}$
    		is a $g_{T_{S,\ell_0}(\omega)}(\omega,\md \ell)$-null set.			
    		Now we get from $\bar{L}_{S,T_{S,\ell_0}}(\omega)<\ell$
    		that 
    		\begin{align}\label{sto:eq_85}
    			X_v(\omega)<Y_v^\ell(\omega) 
    		\text{ for all } v\in [S(\omega),T_{S,\ell_0}(\omega)].
    		\end{align}			
    		Furthermore 
    		$\bar{L}_{S,T_{S,\ell_0}}(\omega)<\ell$ implies by 
    		$A\subset B$ in Lemma \ref{Lem:s5} (\eqref{sto:set_a}
    		and \eqref{sto:set_b}) that $T_{S,\ell_0}(\omega)
    		\leq T_{S,\ell}(\omega)$ and therefore by monotonicity of $\ell\mapsto
    		T_{S,\ell}(\omega)$ that $t=T_{S,\ell_0}(\omega)= T_{S,\ell}(\omega)$.
    		Hence if $\ell \in (\bar{L}_{S,T_{S,\ell_0}}(\omega),\ell_0]$
    		we have $X_t(\omega)<Y_t^\ell(\omega)$ and there exists
    		$q\in (\bar{L}_{S,T_{S,\ell_0}}(\omega),\ell]\cap \mathbb{Q}$
    		with $X_t(\omega)<Y_t^q(\omega)$. By $\omega\in \tilde{\Omega}$
    		this implies $\omega\in H_{S,q}^+$ and monotonicity of $\ell\mapsto
    		\tau_{S,\ell}$ we get $\omega \in H_{S,\ell}^+$,
    		which proves part (c).
		
		
	    \subsection{Proof of Lemma \ref{sto:Lem_13}}\label{proof:4.7}
			Note first that by Lemma \ref{Lem:sp3} (iii) $\ell\mapsto T_{S,\ell}$ is non-decreasing. Hence $T_{S,\infty}$ exists
			as a monotone limit of stopping times. Moreover, by Lemma \ref{Lem:sp3} (vi) and the definition of the essential
			supremum, we have
			\begin{align*}
				\EW[Y^\ell_S]=\EW\left[X_{\tau_{S,\ell}}+
										\int_{[S,\tau_{S,\ell})} g_t(\ell)\mu(\md t)\right]
							\geq \EW\left[X_{\infty}+\int_{[S,\infty)} g_t(\ell)\mu(\md t)
							\right]
			\end{align*}
			or, equivalently, as $X_{\infty}=0$ by assumption,
			\begin{align*}
				\EW\left[X_{\tau_{S,\ell}}\right]
							\geq   
				\EW\left[\int_{[S,\tau_{S,\ell})^c} g_t(\ell)\mu(\md t) \right].
			\end{align*}
			Hence, for any $\mathbb{Q}\ni \ell_0> 0$, we can, by monotonicity of $\ell\mapsto g_t(\ell)$,
			 and normalization to $g_t(0)=0$ (see (\ref{sto:sep_gl})), use
			 monotone convergence to conclude
			\begin{align}
				\EW[M_{S}^X] &\geq \liminf_{\mathbb{Q}\ni \ell\uparrow \infty} \EW\left[X_{\tau_{S,\ell}}\right]
				\geq   
				\liminf_{\mathbb{Q}\ni \ell\uparrow \infty} \EW\left[\int_{[S,\tau_{S,\ell})^c} g_t(\ell_0)\mu(\md t) \right]\nonumber\\
					&\geq \liminf_{\mathbb{Q}\ni \ell\uparrow \infty} \EW\left[\int_{(T_{S,\ell},\infty)}
							 g_t(\ell_0)\mu(\md t) \right]
							\geq \EW\left[\int_{(T_{S,\infty},\infty)} g_t(\ell_0)\mu(\md t)\right]\geq 0,\label{Gl:27}
			\end{align}
			where $M^X\geq X$ with $M^X$ of Lemma \ref{sto:tools_lem} (i) and the first
			inequality follows with the help of \citing{BB18_2}{Lemma 3.16}{13}
			applied to $M^X$. 
			
			For $\ell_0\uparrow
			\infty$, the right-hand side in \eqref{Gl:27} tends to $\infty$ on the set
			$\{\mu((T_{S,\infty},\infty))>0\}$ while the left-hand side
			yields a finite upper bound. Hence, 
			$\WM(\mu((T_{S,\infty},\infty))>0)=0$, establishing \eqref{sto:eq_36}.
			
			Now we want to analyse more precisely the set 
			$\{\mu(\{T_{S,\infty}\})>0\}$. By repeating the arguments in \eqref{Gl:27}
			and using \eqref{sto:eq_36} we obtain
			\begin{align}
				\EW[M_{S}^X ]
					&\geq   
				\liminf_{\mathbb{Q}\ni \ell\uparrow \infty}\EW\left[\int_{[S,\tau_{S,\ell})^c} g_t(\ell_0)\mu(\md t) \right]\nonumber\\
					&\geq \EW\left[g_{T_{S,\infty}}(\ell_0)\mu(\{T_{S,\infty}\})\left(\mathbb{1}_{\Gamma}
					+\liminf_{\mathbb{Q}\ni \ell\uparrow \infty}
					\mathbb{1}_{(H_{S,\ell}^-\cup H_{S,\ell})\cap \Gamma^c}\right)\right]\nonumber\\
					&\overset{\eqref{sto:eq_12}}{=} \EW\left[g_{T_{S,\infty}}(\ell_0)\mu(\{T_{S,\infty}\})\left(\mathbb{1}_{\Gamma}
					+\liminf_{\mathbb{Q}\ni \ell\uparrow \infty}
					\mathbb{1}_{\{X_{T_{S,l}}=Y_{T_{S,l}}^l\}\cap \Gamma^c}\right)\right].\label{sto:eq_199}
			\end{align}			
			Hence again by letting $\ell_0$ tend to $\infty$ we obtain
			\eqref{sto:eq_101} and \eqref{sto:eq_104} if we can show
			\eqref{sto:eq_105}.
			But actually \eqref{sto:eq_105} follows immediately by \eqref{sto:eq_12} and monotonicity of $\ell\mapsto Y^\ell$.
			
			Let us now show \eqref{sto:eq_38}.
			Repeating the arguments in \eqref{Gl:27} replacing the
			expectation by the conditional expectation with respect to 
			$\aFA_S$ gives us
			\[
				\liminf_{\mathbb{Q}\ni \ell\uparrow \infty}
			\EW\left[X_{\tau_{S,\ell}}\midG \aFA_S\right]\geq 0
			\]
			 almost surely.	
			 On the other hand, it remains to prove 
			\begin{align}\label{sto:eq_102}
					\limsup_{\mathbb{Q}\ni \ell\uparrow \infty}
			\EW\left[X_{\tau_{S,\ell}}\midG \aFA_S\right]\leq 0.
			\end{align}
			Actually we will prove
			\begin{align}\label{sto:eq_146}
					\limsup_{\mathbb{Q}\ni \ell\uparrow \infty}
			\EW\left[X_{\tau_{S,\ell}}\mathbb{1}_{\Gamma}\midG \aFA_S\right]\leq 0.
			\end{align}
			and 
			\begin{align}\label{sto:eq_147}
					\limsup_{\mathbb{Q}\ni \ell\uparrow \infty}
			\EW\left[X_{\tau_{S,\ell}}\mathbb{1}_{\Gamma^c}\midG \aFA_S\right]\leq 0,
			\end{align}
			which will lead to \eqref{sto:eq_102}. Remind that the set $\Gamma$ 
			is not necessarily $\aFA_S$-measurable.
			
			\textbf{Proving \eqref{sto:eq_146}:} First 
			$(T_{S,\infty})_{\Gamma}$ is a predictable stopping time
			with announcing sequence $(T_{S,n})_{\{T_{S,n}<T_{S,\infty}\}}\wedge n$.
			By \eqref{sto:eq_36} and \eqref{sto:eq_101} we get
			$\mu([(T_{S,\infty})_{\Gamma},\infty))=0$ almost surely. Hence
			we have by assumption on $X$ that  $X_{(T_{S,\infty})_{\Gamma}}=0$.
			Therefore we obtain by Fatou's Lemma and
			then Lemma \ref{sto:tools_lem} (ii) that
			\begin{align*}
			  \limsup_{\mathbb{Q}\ni \ell\uparrow \infty}	\EW\left[
				X_{\tau_{S,\ell}}
										\mathbb{1}_{\Gamma}\midG \aFA_{S}\right]
			 \leq   \EW\left[ \lsl{X}_{(T_{S,\infty})_{\Gamma}}
			\midG \aFA_{S}\right]
			\leq\EW\left[ X_{(T_{S,\infty})_{\Gamma}}		
				\midG \aFA_{S}\right]
			=0.
			\end{align*}			
			Before proving \eqref{sto:eq_147}
			we need as an intermediate result the following claim:
			
			\textbf{Claim:}
			\begin{align*}
			  \limsup_{\mathbb{Q}\ni \ell\uparrow \infty}	\EW\left[
				X_{\tau_{S,\ell}}			\mathbb{1}_{\Gamma^c}\midG \aFA_{S}\right]
			 \leq   \EW\left[ {X}_{T_{S,\infty}}\mathbb{1}_{\{X_{T_{S,\infty}}
				=Y_{T_{S,\infty}}^\ell
					\text{ for all $\ell$}\}}
			\mathbb{1}_{\Gamma^c}
				\midG \aFA_{S}\right].
			\end{align*}
			
			\textbf{Proof of the claim:} We define for $k\in \NZ$
			\begin{align*}
				A_0&:=\{T_{S,p}=T_{S,\infty}	 \text{ for some $p\leq 0$}\},\\
				A_k&:=\{T_{S,k-1}<T_{S,k}=T_{S,\infty}\}
			\end{align*}
			such that $\Gamma^c$ is the disjoint union of the sets $(A_k)_{k\in \NZ}$.
			Now we get
			\begin{align*}
			  \limsup_{\mathbb{Q}\ni \ell\uparrow \infty}	\EW\left[
				X_{\tau_{S,l}}
										\mathbb{1}_{\Gamma^c}\midG \aFA_{S}\right]
			 \leq   \limsup_{\mathbb{Q}\ni \ell\uparrow \infty}	\EW\left[
				X_{\tau_{S,l}}
										\mathbb{1}_{\cup_{k=0}^\ell A_k}\midG \aFA_{S}\right]
			\end{align*}
			By \eqref{sto:eq_36} we get for any $\epsilon>0$
			that the predictable 
			stopping time $T_{S,\infty}+\epsilon$ satisfies $\mu([T_{S,\infty}+\epsilon,
			\infty))=0$ almost surely. Hence by assumptions on 
			$X$ this gives us $X_{T_{S,\infty}+\epsilon}=0$ and therefore
			$\lsr{X}{T_{S,\infty}}=0$. As by Lemma \ref{Lem:sp3} (iv)
			$\tilde{T}_\ell:=(T_{S,\ell})_{H_{S,\ell}^-}$ is a predictable stopping time
			we obtain by Lemma \ref{sto:tools_lem} (ii) that
			\[
				\lsl{X}_{\tilde{T}_\ell}\leq
				{^\mathcal{P} X}_{\tilde{T}_\ell}.
			\]
			Combining this inequality and $\lsr{X}{T_{S,\infty}}=0$
			leads for any $l\in \mathbb{Q}$ to
			\begin{align}\label{sto:eq_271}
				\EW\left[
						X_{\tau_{S,l}}
										\mathbb{1}_{\cup_{k=0}^\ell A_k}\midG \aFA_{S}\right]\leq &
					\EW\left[\left({^\mathcal{P} X}_{T_{S,\ell}}\mathbb{1}_{H_{S,\ell}^-}
												+X_{T_{S,\ell}}\mathbb{1}_{H_{S,\ell}}\right)
								\mathbb{1}_{\cup_{k=0}^\ell A_k}
									\midG \aFA_S\right].
			\end{align}
			For $k\in \RZ$ we get by \citing{DM78}{Theorem 56 (c)}{p.118}, that 
			$A_k\in \aFA_{T_{S,k}-}\subset \aFA_{\tilde{T}_k-}$.
			Hence we see that for any $l\in \RZ$ and all $k\leq \ell$ 
			\begin{align*}
					\EW\left[{^p X}_{\tilde{T}_\ell}
								\mathbb{1}_{A_k}
									\midG \aFA_S\right]&=
					\EW\left[\EW\left[X_{\tilde{T}_\ell}\midG 
					\aFA_{\tilde{T}_\ell-}\right]
								\mathbb{1}_{A_k}
									\midG \aFA_S\right]\\
						&=\EW\left[\EW\left[X_{\tilde{T}_\ell}
						\mathbb{1}_{A_k}\midG \aFA_{\tilde{T}_\ell-}\right]
								\midG \aFA_S\right]
								=\EW\left[X_{\tilde{T}_\ell}\mathbb{1}_{A_k}	\midG \aFA_S\right].
			\end{align*}
			Plugging this into \eqref{sto:eq_271} gives us
			\begin{align*}
				\limsup_{\mathbb{Q}\ni \ell\uparrow \infty}	\EW\left[
				X_{\tau_{S,l}}
										\mathbb{1}_{\Gamma^c}\midG \aFA_{S}\right] \leq &\limsup_{\mathbb{Q}\ni \ell\uparrow \infty}
					\EW\left[\mathbb{1}_{\cup_{k=0}^\ell A_k}
					X_{T_{S,\infty}}\mathbb{1}_{H_{S,\ell}^-\cup H_{S,\ell}}
														\midG \aFA_S\right]
			\end{align*}
			By Fatou's Lemma and \eqref{sto:eq_105} this finally proves our
			claim.

		\textbf{Proving \eqref{sto:eq_147}:}
			The previous claim leads to \eqref{sto:eq_102} if we can show
			\begin{align}\label{sto:eq_103}
				{X}_{T_{S,\infty}}\mathbb{1}_{\{X_{T_{S,\infty}}
				=Y_{T_{S,\infty}}^\ell
					\text{ for all $\ell$}\}}
			\mathbb{1}_{\Gamma^c}=0.
			\end{align}
			For that we will show that
			\[
				\tilde{T}:=(T_{S,\infty})_{\{X_{T_{S,\infty}}
				=Y_{T_{S,\infty}}^\ell
					\text{ for all $\ell$}\}\cap \Gamma^c}
			\]
			is a $\Lambda$-stopping time. If this is true
			we obtain
			by \eqref{sto:eq_36} and \eqref{sto:eq_104}
			that $\mu([\tilde{T},
			\infty))=0$ almost surely and by assumptions on $X$
			that $X_{\tilde{T}}=0$ establishing \eqref{sto:eq_103}.
			
			\textbf{Showing that $\tilde{T}$ is a $\Lambda$-stopping time:}
			First we have
			\[
				\{X_{T_{S,\infty}}
				=Y_{T_{S,\infty}}^\ell
					\text{ for all $\ell$}\}\cap \Gamma^c
					=\bigcup_{k=0}^{\infty} B_k
			\]			
			with disjoint sets
			\begin{align*}
				B_0&:=\{X_{T_{S,0}}
				=Y_{T_{S,0}}^\ell
					\text{ for all $\ell$}\}\cap \{T_{S,0}=T_{S,\infty}\},\\
					B_k&:=\{X_{T_{S,k}}
				=Y_{T_{S,k}}^\ell
					\text{ for all $\ell$}\}\cap \{T_{S,k-1}<T_{S,k}=T_{S,\infty}\}.
			\end{align*}		
			Next we have by Lemma \ref{Lem:sp3} (iv), that for every $k\in \NZ$,
			$(T_{S,k})_{H_{S,k}^-}$ is a predictable stopping
			time and $(T_{S,k})_{H_{S,k}}$ is a $\Lambda$-stopping
			time. Combined $(T_{S,k})_{H_{S,k}^-\cup H_{S,k}}=
			(T_{S,k})_{H_{S,k}^-}\wedge (T_{S,k})_{H_{S,k}}$ is a 
			$\Lambda$-stopping time. Now we define 
			\[
				T_k:=(T_{S,k})_{B_k},\text{ for } k\in \NZ,
			\]			
			which is again a
			$\Lambda$-stopping time. Indeed, one can see by \eqref{sto:eq_12}
			that
			 $B_k\subset \{X_{T_{S,k}}=Y_{T_{S,k}}^k\}=H_{S,k}^-\cup H_{S,k}$
			 up to a $\WM$-null set and we assume without loss of generality that
			 this actually holds true for all $\omega\in \Omega$.
			 Then we can rewrite $B_k$ by
			\begin{align*}
				B_k
				&=B_k\cap (H_{S,k}^-\cup H_{S,k})&&\\
				   &=\left\{X_{(T_{S,k})_{H_{S,k}^-\cup H_{S,k}}}
				=Y_{(T_{S,k})_{H_{S,k}^-\cup H_{S,k}}}^\ell
				\text{ for all $\ell\in \mathbb{Q}$}\right\}&&\cap (H_{S,k}^-\cup H_{S,k})\\
				&&& \cap  
				\{T_{S,k-1}<T_{S,k}=T_{S,\infty}\},
			\end{align*}
			where we have used that $\ell\mapsto Y^\ell$ is non-decreasing and
			we can therefore restrict to $\ell\in \mathbb{Q}$.		 
			 This shows then by the $\Lambda$-measurability
				of $X$ and $Y^k$ that $B_k\in \aFA_{(T_{S,k})_{H_{S,k}^-
				\cup H_{S,k}}}$ and hence
				$T_k\in \stm$.		
				As for every $\omega\in \Omega$ there exists 
				at most one $k\in \NZ$ with $T_k(\omega)<\infty$
				we can see that
			\[
				\tilde{T}=\bigwedge_{k=1}^\infty T_k
			\]		
			and
			\[
				\stsetRO{\tilde{T}}{\infty}=\bigcup_{k=1}^\infty \stsetRO{T_k}{\infty}\in \Lambda.
			\]
			This implies that $\tilde{T}$ is a $\Lambda$-stopping time and finishes Step 2.	
	
			
		\section{Proof of Theorem \ref{thm:opt_stop}}\label{sec:opt_stop}

		    As a first step we get by \citing{BB18_2}{Lemma 4.4 (ii)}{21}, that left-upper-semicontinuity in expectation is equivalent to $\lsl{X}\leq {^\mathcal{P}X}$ up to indistinguishability. Hence, we can obtain for any $\tau=(T,H^-,H,H^+)\in \stmd$ an alternative divided stopping time $\tilde{\tau}=(T,\emptyset,H^-\cup H,H^+)$ which yields
			at least as high a value in \eqref{eq:0} as $\tau$ does. Hence, we can restrict ourselves
			to divided stopping times with $H^-=\emptyset$. 
			These divided stopping times
			can be approximated by $\Lambda$-stopping times as follows. Define for $\tau=(T,\emptyset,H,H^+)\in \stmd$ the times
			\[
				\tilde{T}_n:=T_{H^-\cup H} \wedge (T_n)_{H^+}, \quad n \in \NZ,
			\]
			where $T_n> T$ on $\{T<\infty\}$ is a sequence of
			$\Lambda$-stopping times such that	$\lim_{n\rightarrow \infty} X_{T_n}=X^\ast_{T_{H^+}}$ as given by \citing{BB18_2}{Proposition 4.2 (i)}{19}.
			Recalling the definition $X_\tau$ from \eqref{sto:eq_9}, we then have
			\[
				\lim_{n\rightarrow \infty} 
				\EW\left[X_{T_n}+\int_{[0,T_n)}g_t(\ell)\mu(\md t)\right]
				=
				\EW\left[X_\tau+\int_{[0,\tau)}g_t(\ell)\mu(\md t)\right].
			\]
			It follows that our optimal stopping problem \eqref{eq:0} 
			attains the
			same value as the optimal stopping problem over 
			$\Lambda$-stopping times, i.e.
			\begin{align}\label{eq:2}
				\sup_{\tau\in \stmd}
				\EW\left[X_\tau+\int_{[0,\tau)}g_t(\ell)\mu(\md t)\right]
				=\sup_{T\in \stm}
				\EW\left[X_T+\int_{[0,T)}g_t(\ell)\mu(\md t)\right].
			\end{align}	
			
            Next we consider 
			\[
					\tilde{\tau}_\ell:=(T_{0,\ell},\emptyset,\{L_{T_{0,\ell}}\geq \ell\},
				\{L_{T_{0,\ell}}<\ell\}),
			\]
			where $T_{0,\ell}$ is the stopping time defined (with $S=0$) in \eqref{sto:useful_gl_8}.
			We claim that $\tilde{\tau}_\ell$ is a divided stopping time. Indeed, as one can see from Proposition \ref{Lem:s4} the process $L$ is actually obtained as 
			\begin{align}\label{Main:3}
			    L_t=\sup\left\{\ell\in \RZ\midG X_t=Y_t^\ell\right\},\quad t\in[0,\infty),
			\end{align}
    		with $Y$ from Lemma \ref{Lem:sp3}. Therefore, 
			\[
				\{L_{T_{0,\ell}}<\ell\}=\{X_{T_{0,\ell}}<Y_{T_{0,\ell}}^\ell\}
			\]
			and, hence, Lemma \ref{Lem:sp3} (iv) yields that, up to a $\WM$-nullset,
			\[
				\{L_{T_{0,\ell}}<\ell\}=H_{0,\ell}^+,\quad
				\{L_{T_{0,\ell}}\geq \ell\}=H_{0,\ell}^-\cup H_{0,\ell}.
			\]
			This shows that $(T_{0,\ell})_{\{L_{T_{0,\ell}}\geq \ell\}}$ is a.s.
			equal to the $\Lambda$-stopping time 
			$(T_{0,\ell})_{H_{0,\ell}^-\cup H_{0,\ell}}$ and, as $\Lambda$ is 
			a $\WM$-complete Meyer-$\sigma$-field, it is thus also a $\Lambda$-stopping 
			time. This finally yields that $\tilde{\tau}_L$ is indeed a divided stopping time.
			By Lemma \ref{Lem:sp3} (vi), we actually see that $\tilde{\tau}_L$ is an optimal
			divided stopping time for \eqref{eq:0} as we have
			\begin{align*}
				\sup_{\tau\in \stmd}
				\EW\left[X_\tau+\int_{[0,\tau)}g_t(\ell)\mu(\md t)\right]
				&\overset{\eqref{eq:2}}{=}\sup_{T\in \stm}
				\EW\left[X_T+\int_{[0,T)}g_t(\ell)\mu(\md t)\right]\\
				&=Y_0^\ell=\EW\left[X_{\tau_L}+\int_{[0,\tau_L)}g_t(\ell)\mu(\md t)\right].
			\end{align*}
			
		    Finally, let us show that for a.e. $\omega\in \Omega$ we have $T_{0,\ell}(\omega)=T_\ell(\omega)$, which will then show that also $\tau_\ell$ of Theorem \ref{thm:opt_stop} is an 
		    optimal divided stopping time as claimed. 
			The inclusion $B^c\subset A^c$ from Lemma \ref{Lem:s5} reveals 
			that $T_{0,\ell} \geq T_\ell$ almost surely. 
			Hence, it just remains to show that also $T_{0,\ell}\leq T_L$
			almost surely. By Lemma \ref{Lem:sp3} (vii) there exists $\tilde{\Omega}\subset\Omega$
			with $\WM(\tilde{\Omega})=1$ such that for $\omega\in \tilde{\Omega}$ we have
			\begin{align}\label{Main:5}
										\lim_{\delta \downarrow 0}\sup_{\overset{\ell,\ell'\in C}{|\ell'-\ell|
		\leq \delta}}\sup_{t\in I}
										 \left|Y^\ell_t(\omega)-Y^{\ell'}_t(\omega)\right|=0
			\end{align}
			for all compact sets $C\subset \RZ$, $I\subset [0,\infty)$.
			Let now $\omega\in \tilde{\Omega}$ and assume there exists $t<T_{0,\ell}(\omega)$ with $\sup_{v\in [0,t]}L_v(\omega)\geq \ell$. If there exists $\tilde{t}\in [0,t]$
			with $L_{\tilde{t}}(\omega)\geq \ell$ we obtain by \eqref{Main:3} $X_{\tilde{t}}(\omega)=Y_{\tilde{t}}^{\ell}(\omega)$, which would lead by \eqref{sto:useful_gl_8} 
			to $T_{0,\ell}(\omega)\leq {\tilde{t}}\leq t$ contradicting $t<T_{0,\ell}(\omega)$. Hence we have $\sup_{v\in[0,t]}L_{t}(\omega)= \ell$ and 
			$L_{\tilde{t}}(\omega)< \ell$ for all $\tilde{t}\in [0,t]$. Now there has to exist a sequence $(t_n)_{n\in \NZ}\subset [0,t]$ such that
			$(\ell_n)_{n\in \NZ}$ defined by $\ell_n:=L_{t_n}(\omega)$, $n\in \NZ$, satisfies
			$\lim_{n\rightarrow \infty} \ell_n =\sup_{v\in [0,t]}L_v(\omega)= \ell$
			and $\ell_n<\ell$ for all $n\in \NZ$. As $(t_n)_{n\in \NZ}\subset [0,t]$,
			there exists a monotone subsequence, again denoted by $(t_n)_{n\in \NZ}$,
			with limit $\bar{t}$. We will only consider the case of an 
			increasing sequence as the other case follows analogously. Now we obtain
			by $X\leq Y^{\tilde{\ell}}$ for all $\tilde{\ell}\in \RZ$, \eqref{Main:3} and  \eqref{Main:5} that
			\begin{align}
			    \lsl{X}_{\bar{t}}(\omega)\leq Y_{\bar{t}-}^\ell(\omega)
			    =\lim_{n\rightarrow \infty} Y_{t_n}^{\ell_n}(\omega)
			    =\lim_{n\rightarrow \infty} X_{t_n}(\omega)\leq \lsl{X}_{\bar{t}}(\omega),
			\end{align}
			which shows $\lsl{X}_{\bar{t}}(\omega)= Y_{\bar{t}-}^\ell(\omega)$.
			Here the first equality follows by
			\begin{align}
			    \left|Y^\ell_{t_n}(\omega)-Y^{\ell_n}_{t_n}(\omega)\right|
			    \leq \sup_{\ell'\in [\ell_n,\ell]} \sup_{r\in [0,t]}
										 \left|Y^\ell_r(\omega)-Y^{\ell'}_r(\omega)\right|,\quad
										 n\in \NZ.
			\end{align}
			This leads by \eqref{sto:useful_gl_8}
			to $T_{0,\ell}(\omega)\leq {\tilde{t}}\leq t$, which is a contradiction. Hence for all $t<T_{0,\ell}(\omega)$ we have
			$\sup_{v\in [0,t]}L_v(\omega)< \ell$, which shows $T_{0,\ell}(\omega)\leq
			T_\ell(\omega)$ and finishes our proof.
\end{appendix}

\bibliographystyle{plainnat} 
\bibliography{Literature_ALL}

\end{document}